\documentclass[reqno,11pt]{amsart}

\usepackage{hyperref}
\usepackage{mathabx}
\usepackage{mathdots}
\usepackage{tikz}
\usepackage{tikz-cd}
\usepackage{amssymb}
\usepackage{amsgen}
\usepackage{amsmath}
\usepackage{amsthm}
\usepackage{mathrsfs}
\usepackage{cite}
\usepackage{amsfonts}
\usepackage{enumitem}
\usepackage{url}
\usetikzlibrary{shapes,decorations.pathreplacing}
\usetikzlibrary{calc}
\usetikzlibrary{backgrounds}
\usetikzlibrary{automata}
\usetikzlibrary{positioning}
\usetikzlibrary{decorations.markings}
\usetikzlibrary{arrows}
\usetikzlibrary{scopes}
\usetikzlibrary{intersections}
\usetikzlibrary{fit}

\newcommand{\image}{\operatorname{im}}
\newcommand{\coker}{\operatorname{coker}}
\newcommand{\ab}[1]{#1^{\mathrm {ab}}}

\newcommand{\End}{\operatorname{\mathrm{End}}}

\newcommand{\sour}{\operatorname{\boldsymbol s}}

\newcommand{\ran}{\operatorname{\boldsymbol r}}

\newcommand{\tr}{\operatorname{\mathrm{tr}}\nolimits}
\newcommand{\id}{\operatorname{\mathrm{id}}\nolimits}

\newcommand{\inv}{^{-1}}
\newcommand{\p}{\varphi}

\newcommand{\ov}[1]{\ensuremath{\overline {#1}}}
\newcommand{\til}[1]{\ensuremath{\widetilde {#1}}}

\newcommand{\wh}{\widehat}

\newcommand{\Tor}{\operatorname{\mathrm{Tor}}\nolimits}

\usepackage{xcolor}

\newcommand{\cs}{\mathrm{C}^\ast}
\newcommand{\acts}{\curvearrowright}
\newcommand{\rightacts}{\curvearrowleft}
\newcommand{\Z}{\mathbb Z}

\newcommand{\underlying}[1]{D_{#1}}
\newcommand{\universal}[1]{\mathscr U_{#1}}
\newcommand{\reg}{\mathrm{reg}}
\newcommand{\sing}{\mathrm{sing}}

\newcommand*{\suchthat}{\;\ifnum\currentgrouptype=16 \middle\fi|\;}

\newtheorem*{ThmA*}{Theorem~A}
\newtheorem*{Thm*}{Theorem}
\newtheorem{Thm}{Theorem}[section]
\newtheorem{Prop}[Thm]{Proposition}

\newtheorem{Lemma}[Thm]{Lemma}
{\theoremstyle{definition}
\newtheorem{Def}[Thm]{Definition}}
{\theoremstyle{remark}
\newtheorem{Rmk}[Thm]{Remark}}
\newtheorem{Cor}[Thm]{Corollary}
{\theoremstyle{remark}
}
{\theoremstyle{remark}
\newtheorem{Example}[Thm]{Example}}

{\theoremstyle{remark}
\newtheorem{Claim}{Claim}}
{\theoremstyle{remark}
}
{\theoremstyle{remark}
}
{\theoremstyle{remark}
\newtheorem*{Claim*}{Claim}}

\theoremstyle{plain}
\newtheorem{thmx}{Theorem}

\newtheorem{corx}[thmx]{Corollary}

\numberwithin{equation}{section}

\tikzset{
  curarrow/.style={
  rounded corners=8pt,
  execute at begin to={every node/.style={fill=red}},
    to path={-- ([xshift=60pt]\tikztostart.center)
    |- (#1) node {}
    -| ([xshift=-60pt]\tikztotarget.center)
    -- (\tikztotarget)}
    }
}
\tikzset{
  labcurarrow1/.style={
  rounded corners=8pt,
  execute at begin to={every node/.style={fill=red}},
    to path={-- ([xshift=60pt]\tikztostart.center)
    |- (#1) node[fill=white] {$\scriptstyle{\id - H_n(\mathcal X_\reg,A)}$}
    -| ([xshift=-60pt]\tikztotarget.center)
    -- (\tikztotarget)}
    }
}
\tikzset{
  labcurarrow2/.style={
  rounded corners=8pt,
  execute at begin to={every node/.style={fill=red}},
    to path={-- ([xshift=60pt]\tikztostart.center)
    |- (#1) node[fill=white] {$\scriptstyle{\id - \Phi_n}$}
    -| ([xshift=-60pt]\tikztotarget.center)
    -- (\tikztotarget)}
    }
}
\tikzset{
  labcurarrow3/.style={
  rounded corners=8pt,
  execute at begin to={every node/.style={fill=red}},
    to path={-- ([xshift=60pt]\tikztostart.center)
    |- (#1) node[fill=white] {$\scriptstyle{\id - \Phi_1}$}
    -| ([xshift=-60pt]\tikztotarget.center)
    -- (\tikztotarget)}
    }
}
\tikzset{
  labcurarrow4/.style={
  rounded corners=8pt,
  execute at begin to={every node/.style={fill=red}},
    to path={-- ([xshift=60pt]\tikztostart.center)
    |- (#1) node[fill=white] {$\scriptstyle{\id - \Phi_0}$}
    -| ([xshift=-60pt]\tikztotarget.center)
    -- (\tikztotarget)}
    }
}

\title[Homology and K-theory for self-similar actions]{Homology and K-theory for self-similar actions of groups and groupoids}

\author{Alistair Miller}
\address[A.~Miller]{%
    Department of Mathematics and Computer Science\\
    University of Southern Denmark\\
    Moseskovvej\\
    5230 Odense\\
    Denmark}
\email{mil@sdu.dk}

\author{Benjamin Steinberg}
\address[B.~Steinberg]{%
    Department of Mathematics\\
    City College of New York\\
    Convent Avenue at 138th Street\\
    New York, New York 10031\\
    USA}
\email{bsteinberg@ccny.cuny.edu}

\thanks{The first author was supported by the Independent Research Fund Denmark through the Grant 1054-00094B. The second author was supported by a Simons Foundation Collaboration Grant, award number 849561, and the Australian Research Council Grant DP230103184.}
\date{\today}

\keywords{self-similar group, groupoid homology, Nekrashevych algebra, self-similar groupoid, K-theory, groupoid correspondences}
\subjclass[2020]{20J05,22A22,46L05,46L80,20F65}

\begin{document}
\begin{abstract}
Nekrashevych associated to each self-similar group action an ample groupoid and a $\cs$-algebra. 
We perform complete computations of the homology of the groupoid and the K-theory of the $\cs$-algebra for a myriad of examples, including the Grigorchuk group, the Grigorchuk--Erschler group, Gupta--Sidki groups, and self-similar actions of free abelian groups and lamplighter groups.
The key development is the construction, for arbitrary self-similar group actions, of long exact sequences which compute the homology and K-theory in terms of the homology of the group and K-theory of the group $\cs$-algebra via the transfer map and the virtual endomorphism. Results are proved more generally for self-similar groupoids.  As a consequence of our results and recent results of X.~Li, we are able to show that R\"over's simple group containing the Grigorchuk group  and Thompson's group $V$ is rationally acyclic but has  nontrivial Schur multiplier.  We prove many more R\"over--Nekrashevych groups of self-similar groups are rationally acyclic.
\end{abstract}

\maketitle

\tableofcontents

\section{Introduction}

The theory of self-similar groups, in the guise of automaton groups, began in the seventies and eighties with the work of Aleshin, Sushchanski{\u\i} and Grigorchuk; see~\cite{GNS} for more history.  The initial interest in self-similar groups was as a means to provide concrete constructions of groups with exotic properties, such as  finitely generated infinite torsion groups (Burnside groups)~\cite{Grigtorsion,GuptaSidki} and groups of intermediate growth (first~\cite{griggrowth}, and later others~ \cite{FG85}).  Grigorchuk and \.{Z}uk's discovery that the lamplighter group can be realized as a self-similar group, and their use of self-similarity to compute the spectra of random walks on this group~\cite{GrigZuk} led to a flurry of work around  the strong Atiyah conjecture on $\ell_2$-Betti numbers~\cite{strongatiyah}. 

The modern theory of self-similar groups began with Nekrashevych's monograph~\cite{selfsimilar}; see also~\cite{Nekbook2}. In particular, Nekrashevych showed that self-similar groups arise very naturally in dynamical settings via his iterated monodromy group construction.  Bartholdi and Nekrashevych~\cite{BN06} solved Hubbard's twisted rabbit problem using iterated monodromy groups.

Traditionally the theory of self-similar groups was presented as the theory of groups acting on rooted trees, typically in the language of wreath products~\cite{GNS}. Nekrashevych developed the abstract theory in~\cite{selfsimilar} in terms of proper self-correspondences of discrete groups (using the language of covering bimodules). In the spirit of noncommutative geometry, Nekrashevych introduced a $\cs$-algebra $\mathcal O_{(G,X)}$ to encode the underlying self-similar space of a self-similar group action $(G,X,\sigma)$ with finite alphabet $X$ and cocycle $\sigma \colon G \times X \to G$~\cite{Nekcstar}, and further provided a groupoid model $\mathscr G_{(G,X)}$. 
For faithful self-similar group actions, the groupoid is always purely infinite, 
minimal and effective, but often it is not Hausdorff.  A lot of the effort to understand simplicity of algebras and $\cs$-algebras associated to non-Hausdorff groupoids was motivated in part by the case of Nekrashevych algebras of self-similar groups; in particular establishing simplicity of the Nekrashevych algebra of the Grigorchuk group~\cite{Nekrashevychgpd,nonhausdorffsimple,simplicity,SS23}. Since the groupoid is purely infinite, minimal and effective, the Nekrashevych algebra $\mathcal O_{(G,X)} = \cs(\mathscr G_{(G,X)})$ is purely infinite when it is simple~\cite{KwaMey21}.

The homology of the groupoid $\mathscr G_{(G,X)}$ associated to a self-similar group action $(G,X,\sigma)$ and the K-theory of its Nekrashevych algebra $\mathcal O_{(G,X)}$ provide fundamental invariants for the self-similar system. The K-theory is particularly pertinent as many Nekrashevych algebras, e.g. for the Grigorchuk group and all regular contracting self-similar groups, are UCT Kirchberg algebras, which are classified by their K-theory by the Kirchberg--Phillips Theorem~\cite{Kirchberg, Phillips}.
The groupoid homology of an ample groupoid $\mathscr G$ shares many similarities with the K-theory of its reduced $\cs$-algebra, enjoying a close relationship to the K-theory when the isotropy groups are torsion-free~\cite{PY22, miller2024isomorphisms}. Further, the homology of $\mathscr G_{(G,X)}$ is related to the homology of its topological full group \cite{li2022ample}, which (as pointed out in~\cite{Nekbook2}) is the R\"over--Nekrashevych group $V(G)$ of the self-similar group action $(G,X,\sigma)$.
This group is studied in~\cite{NekraTFG,NekraTFG2}, where the commutator subgroup $V(G)'$ is shown to be simple and the abelianization $V(G)/V(G)'$ is computed.  When $(G,X,\sigma)$ is contracting, the groups $V(G)$ and $V(G)'$ are finitely presented. 

Nekrashevych computed the K-theory of Nekrashevych algebras of iterated monodromy groups of  post-critically finite hyperbolic rational functions~\cite{Nekcstar}. He used a two-step approach, first relating the K-theory of $\mathcal O_{(G,X)}$ with that of its gauge-invariant subalgebra, which is then described as an inductive limit of matrix amplifications of $\cs(G)$. A similar approach, at the groupoid level, was taken by Ortega and Sanchez~\cite{ortegadihedral} to study the homology of the groupoid associated to a certain self-similar action of the infinite dihedral group. Although they were able to prove that the homology was torsion, which was enough to show that the groupoid strongly contradicts Matui's HK conjecture~\cite{Matui16}, this approach does not seem well adapted to computing the homology precisely. Deaconu also outlines a general strategy along these lines for both K-theory and homology~\cite{DeaconuSelfSimilar}.

We compute homology and K-theory in much greater generality by relating the groupoid $\mathscr G_{(G,X)}$ and the $\cs$-algebra $\mathcal O_{(G,X)}$ of a self-similar group action $(G,X,\sigma)$ directly to $G$ and $\cs(G)$ respectively. This provides us a means to compute the homology and K-theory in entirely group-theoretic terms, namely via the transfer map and the virtual endomorphism. 

\begin{thmx}\label{Theorem A}
Let $(G,X,\sigma)$ be a self-similar group action over a finite alphabet $X$.  For $x\in X$, let $\sigma_x\colon G_x\to G$ be the virtual endomorphism $g\mapsto g|_x =\sigma(g,x)$. Then there is a long exact sequence
\[\cdots\to H_{n+1}(\mathscr G_{(G,X)})\to H_n(G)\xrightarrow{\id - \Phi_n} H_n(G)\to H_n(\mathscr G_{(G,X)})\to \cdots\] where $\Phi_n = \sum_{x \in T} H_n(\sigma_x)\circ \mathrm{tr}^{G}_{G_x}$ for any transversal $T$ to $G\backslash X$.
\end{thmx}

The transfer map, which for the finite index stabilizer group $G_x$ of $x \in X$ is realised by the proper correspondence $\tr^{G}_{G_x}\colon G \to G_x$\footnote{Abusing notation slightly we write $\tr^G_{G_x}$ for both the correspondence and its induced maps in homology and K-theory.} with bispace $G$, also plays a key role in Scarparo's work on homology and K-theory of odometers~\cite{Scarparo20}. The power of Theorem \ref{Theorem A} is that it allows us to draw from the arsenal of techniques from nearly a century of development in group homology.  For example, group homology has a well-known interpretation in terms of the homology of Eilenberg--Mac Lane spaces and the transfer map has a covering space interpretation~\cite{Browncohomology}. 

For the K-theory of the Nekrashevych algebra $\mathcal O_{(G,X)}$ we have an analogous six-term sequence.

\begin{thmx}\label{Theorem B}
Let $(G,X,\sigma)$ be a self-similar group action over a finite alphabet $X$.  Let $\sigma_x\colon G_x\to G$ be the virtual endomorphism $g\mapsto g|_x$ for $x\in X$.  Then there is a six-term exact sequence
\[\begin{tikzcd}[arrow style=math font]
	{K_0(\cs(G))} & {K_0(\cs(G))} & {K_0(\mathcal O_{(G,X)})} \\
	{K_1(\mathcal O_{(G,X)})} & {K_1(\cs(G))} & {K_1(\cs(G))}
	\arrow["{1-\Phi_0}", from=1-1, to=1-2]
	\arrow[from=1-2, to=1-3]
	\arrow[from=1-3, to=2-3]
	\arrow[from=2-1, to=1-1]
	\arrow[from=2-2, to=2-1]
	\arrow["{1-\Phi_1}", from=2-3, to=2-2]
\end{tikzcd}\]
where $\Phi_i = \sum_{x \in T} K_i(\sigma_x)\circ \mathrm{tr}^{G}_{G_x}$ for $i = 0,1$ and any transversal $T$ to $G\backslash X$.
\end{thmx}

A large part of this paper is devoted to applying Theorems~\ref{Theorem A} and~\ref{Theorem B} to perform detailed computations of these invariants for many of the most famous self-similar groups.  We include here two particular sample computations.

\begin{thmx}\label{t:grig.com}
Let $\mathscr G$ be the groupoid associated to the Grigorchuk group.  Then 
\[H_n(\mathscr G) = 
    \begin{cases}
        0, & \text{if}\ n=0,\\ 
        (\mathbb Z/2\mathbb Z)^{\frac{n}{3}+1}, & \text{if}\ n\equiv 0\bmod 3, n\geq 1,\\ 
        (\mathbb Z/2\mathbb Z)^{\frac{n-1}{3}}, & \text{if}\ n\equiv 1\bmod 3,\\ 
        (\mathbb Z/2\mathbb Z)^{\frac{n+1}{3}}, & \text{if}\ n\equiv 2\bmod 3.
    \end{cases}\] 
On the other hand, its Nekrashevych algebra $\mathcal O_{\mathrm{Grig}} = \cs(\mathscr G)$ has $K_0(\mathcal O_{\mathrm{Grig}})\cong \mathbb Z $ with $[1]_0 = 0$ and $K_1(\mathcal O_{\mathrm{Grig}}) \cong \Z$.
\end{thmx}

Self-similar actions of free abelian groups generalize the $\cs$-algebras associated to dilation matrices studied in~\cite{EHR11}, as observed in~\cite{LRRW14}.

\begin{thmx}
Let $(\mathbb Z^d,X)$ be a self-similar transitive action on a set $X$ of cardinality $d\geq 2$ with virtual endomorphism $\sigma_x$ for some $x \in X$. Let $A$ the matrix of $\sigma_x\otimes 1_{\mathbb Q}$.  Then:
\begin{align*}
H_q(\mathscr G_{(\mathbb Z^d,X)}) & \cong \ker(\id -d\Lambda^{q-1}(A))\oplus \coker (\id -d\Lambda^q(A))\\
K_0(\mathcal O_{(\mathbb Z^d,X)}) & \cong \bigoplus_{q\geq 0} H_{2q}(\mathscr G_{(\mathbb Z^d,X)})\\
K_1(\mathcal O_{(\mathbb Z^d,X)}) & \cong \bigoplus_{q\geq 0} H_{2q+1}(\mathscr G_{(\mathbb Z^d,X)})
\end{align*}
 where $\Lambda^q(A)$ is the $q^{th}$-exterior power of $A$.
\end{thmx}
The above relationship $K_i(\cs_\lambda(\mathscr G)) \cong \bigoplus_{q \geq 0} H_{2q+i}(\mathscr G)$ for $i=0,1$ is known as the HK property for $\mathscr G$, which was originally conjectured by Matui~\cite{Matui16} for certain ample groupoids. Most of our examples with torsion-free isotropy, as above, satisfy the HK property. However, torsion in the isotropy poses a problem~\cite{Scarparo20} for the HK property,\footnote{It should be noted that sufficiently high-dimensional behaviour, even for principal groupoids~\cite{Deeley}, can also pose a problem for the HK property.} and indeed most of our examples with torsion in the isotropy do not satisfy even the rational HK property, which asks for analogous isomorphisms with $\mathbb Q$-coefficients.

R\"over introduced $V(\Gamma)$ for the Grigorchuk group $\Gamma$ in~\cite{RoverGroup}, where he showed that it was a finitely presented simple group generated by Thompson's group $V$ and the Grigorchuk group.  This group was later shown to have the topological finiteness property $F_{\infty}$~\cite{RoverFinf}.  Recall that a group $G$ is acyclic if $H_n(G)=0$ for all $n\geq 1$ and rationally acyclic if $H_n(G,\mathbb Q)=0$ for all $n\geq 1$.  Brown~\cite{BrownGeometry} proved that $V$ is rationally acyclic, and Szymik and Wahl recently proved that $V$ is acyclic~\cite{SW19}. Li's results about topological full groups~\cite{li2022ample} provide a conceptual explanation for this acyclicity. They also have implications for any R\"over--Nekrashevych group $V(G)$; we highlight below the computation of the rational homology of $V(G)$ and $V(G)'$ from the rational homology of $\mathscr G_{(G,X)}$ and some vanishing implications for the integral homology. Combined with Theorem \ref{t:grig.com}, this implies that R\"over's group $V(\Gamma)$ is rationally acyclic, and enables us to prove that it is not acyclic by computing its Schur multiplier.  We also get similar results for a more general class of R\"over--Nekrashevych groups.

\begin{Thm*}[Li, Corollaries~C and ~D~\cite{li2022ample}]
Let $(G,X,\sigma)$ be a self-similar group action.  Then
\begin{align*}
H_\bullet(V(G),\mathbb Q) & \cong \Lambda(H^{\mathrm{odd}}_{\bullet}(\mathscr G_{(G,X)},\mathbb Q))\otimes \mathrm{Sym}(H^{\mathrm{even}}_{\bullet}(\mathscr G_{(G,X)},\mathbb Q))\\
H_\bullet(V(G)',\mathbb Q) & \cong \Lambda(H^{\mathrm{odd}}_{\bullet>1}(\mathscr G_{(G,X)},\mathbb Q))\otimes \mathrm{Sym}(H^{\mathrm{even}}_{\bullet}(\mathscr G_{(G,X)},\mathbb Q))
\end{align*}
as graded $\mathbb Q$-vector spaces.
Moreover, if $k>0$ with $H_q(\mathscr G_{(G,X)})=0$ for $0\leq q<k$, then $H_q(V(G))=0$ for $0<q<k$ and $H_k(V(G))\cong H_k(\mathscr G_{(G,X)})$.
\end{Thm*}
Here $\Lambda$ is the exterior algebra and $\mathrm{Sym}$ is the symmetric algebra.

\begin{corx}
Let $\Gamma$ be the Grigorchuk group.  Then R\"over's simple group $V(\Gamma)$ is rationally acyclic and has Schur multiplier $H_2(V(\Gamma))\cong \mathbb Z/2\mathbb Z$. More generally, if $G$ is any multispinal group~\cite{SS23} (e.g., a Gupta--Sidki group, GGS-group or \v{S}uni\'c group) or the Hanoi towers group~\cite{Hanoitowers}, then the Nekrashevych--R\"over group $V(G)$ and its commutator subgroup $V(G)'$ are rationally acyclic.  
\end{corx}

One key point in our approach is that we do not require faithfulness of the self-similar action. We introduce the weaker notion of \emph{loose faithfulness}, which is practical to verify for contracting actions.

\begin{thmx}\label{Theorem C}
Let $(G,X,\sigma)$ be a loosely faithful self-similar group action over a finite set $X$ with faithful quotient $(\ov G,X,\sigma)$.  Then the groupoids $\mathscr G_{(G,X)}$ and $\mathscr G_{(\ov G,X)}$ are isomorphic. A contracting action $(G,X,\sigma)$ is loosely faithful if the nontrivial elements of the nucleus of $G$ act nontrivially on the tree of words. 
\end{thmx}

The $\cs$-algebra analogue of this theorem was already observed by Nekrashevych in~\cite{Nekcstar}.
The importance of this observation is that the group $\ov G$ is often more complicated than the group $G$.  For example, the Grigorchuk group is the faithful quotient of a loosely faithful contracting self-similar action of $\mathbb Z/2\mathbb Z\ast(\mathbb Z/2\mathbb Z\times \mathbb Z/2\mathbb Z)$.  The latter group has well-understood homology and K-theory, making it feasible to apply Theorems~\ref{Theorem A} and ~\ref{Theorem B}.  The Grigorchuk group, by way of contrast, has infinitely generated second homology and not much is known about the higher homology groups, nor about the K-theory of its $\cs$-algebra.

Exel and Pardo considered a generalization of self-similar group actions to self-similar group actions on graphs~\cite{ExelPardoSelf}, which were further generalized to self-similar groupoid actions on graphs in~\cite{LRRW18}.  It was observed in~\cite{AKM22} that self-similar groupoid actions are exactly self-correspondences of discrete group\-oids, and that is the language we use in this paper.  Analogues of Theorems~\ref{Theorem A} and~\ref{Theorem B} are established for self-similar groupoids.  We show that Matui's computation of the homology of graph groupoids~\cite{Matui} (see also~\cite{OrtegaNylandgraph}) and Nyland and Ortega's~\cite{OrtegaNylandKatsura} computation of the homology of Exel--Pardo--Katsura groupoids follow directly from these analogues.  

We address amenability of our groupoids, which in particular implies that our computations for the full $\cs$-algebra hold for the reduced $\cs$-algebra.
The following sufficient condition for amenability is the most general to date. 

\begin{thmx}
Let $(G,E,\sigma)$ be a self-similar groupoid action with $G,E$ countable.  If $G\ltimes \partial E$ is amenable, then so is $\mathscr G_{(G,E)}$. If $(G,E,\sigma)$ is loosely faithful, then $\mathscr G_{(G,E)}$ is amenable if and only if the groupoid of germs for $G \acts \partial E$ is amenable.
\end{thmx}

In particular, the groupoid of any contractible self-similar groupoid action is amenable. We provide both Nekrashevych's original proof~\cite{Nekcstar} and an alternative proof.

Our approach in this paper is based on the use of \'etale groupoid correspondences and their induced mappings on homology~\cite{miller2023ample} and K-theory (via~\cite{AKM22}).  Theorem~\ref{Theorem B} is an application of Katsura's six-term exact sequence for relative Cuntz--Pimsner algebras~\cite{Katsura04}.  
The main idea for Theorem~\ref{Theorem A} is to model the Toeplitz extension $I\hookrightarrow \mathcal T_{(G,X)}\twoheadrightarrow \mathcal O_{(G,X)}$ in terms of the groupoids associated to the corresponding inverse semigroup $S_{(G,X)}$.  Namely,  $\mathcal T_{(G,X)}$ is the $\cs$-algebra of the universal groupoid $\universal{S_{(G,X)}}$, $\mathcal O_{(G,X)}$ is the $\cs$-algebra of the tight groupoid $\mathscr G_{S_{(G,X)}}$ and $\universal{S_{(G,X)}}\setminus \mathscr G_{S_{(G,X)}}$ is isomorphic to the underlying groupoid 
of $S_{(G,X)}$. The underlying groupoid of $S_{(G,X)}$ is Morita equivalent to $G$, whereas the universal groupoid $\universal{S_{(G,X)}}$ has the same homology as $G$ by the first author's results~\cite{miller2023ample}.  These facts, together with the long exact sequence associated to an invariant closed subgroupoid, lead to Theorem \ref{Theorem A}.

\subsubsection*{Acknowledgements}
The first author would like to thank Kevin Aguyar Brix, Chris Bruce, Jeremy Hume, Xin Li and Mike Whittaker for valuable discussions and references. The second author would like to thank Volodymyr Nekrashevych for pointing out that his result on amenability of groupoids of contracting groups works in the non-Hausdorff setting. 
The authors would like to give special thanks to Bartosz Kwaśniewski for pointing out that the subgroupoid $\mathscr H_0$ in Theorem \ref{t:amenability} need not be closed, leading us to improve Proposition \ref{p:AmenabilityPermanence}, and for pointing out a flaw in our original definition of the tight kernel in the presence of singular vertices.


\section{$\cs$-algebras and groupoids associated to self-similar groupoids}

\subsection{Self-similar groupoids}

We recall the definitions for \'etale correspondences and their composition~\cite{AKM22}, which will play a prominent role in the sequel, mostly for discrete groups and groupoids.

\begin{Def}[\'Etale correspondence]
Let $\mathscr G$ and $\mathscr H$ be ample groupoids. An \textit{\'etale correspondence} $\Omega \colon \mathscr G \to \mathscr H$ is a $\mathscr G$-$\mathscr H$-bispace $\Omega$ such that the right action $\Omega \rightacts \mathscr H$ is free, proper and \'etale. We write $\ran \colon \Omega \to \mathscr G^0$ and $\sour \colon \Omega \to \mathscr H^0$ for the anchor maps. That $\Omega \rightacts \mathscr H$ is \'etale means that $\sour \colon \Omega \to \mathscr H^0$ is \'etale,\footnote{We use \'etale map as a synonym for local homeomorphism.} while free and proper together mean that the map $\ran \times \sour \colon \Omega \rtimes \mathscr H \to \Omega \times \Omega$ is a closed embedding. We say $\Omega$ is \textit{proper} if the induced map $\Omega/\mathscr H \to \mathscr G^0$ is proper.

The \textit{composition} $\Lambda \circ \Omega \colon \mathscr G \to \mathscr K$ of \'etale correspondences $\Omega \colon \mathscr G \to \mathscr H$ and $\Lambda \colon \mathscr H \to \mathscr K$ is an \'etale correspondence whose bispace is the \textit{fibre product} $\Omega \times_\mathscr H \Lambda$ of $\Omega$ and $\Lambda$ over $\mathscr H$. This is the quotient of $\Omega \times_{\mathscr{H}^0}\Lambda$ by the diagonal action of $\mathscr H$. We write $[\omega, \lambda]_{\mathscr H}$ for elements of $\Omega \times_{\mathscr H} \Lambda$, or $[\omega,\lambda]$ if $\mathscr H$ is understood, so that $[\omega \cdot h, \lambda]_{\mathscr H} = [\omega, h \cdot \lambda]_{\mathscr H}$ for compatible $\omega \in \Omega$, $h \in \mathscr H$ and $\lambda \in \Lambda$. The $\mathscr G$-$\mathscr K$-bispace structure on $\Omega \times_{\mathscr H} \Lambda$ is given by $g \cdot [\omega, \lambda]_{\mathscr H} = [g \cdot \omega, \lambda]_{\mathscr H}$ and $[\omega, \lambda]_{\mathscr H}\cdot k = [\omega, \lambda \cdot k]_{\mathscr H}$ whenever $\sour(g) = \ran(\omega)$ and $\ran(k) = \sour(\lambda)$.
\end{Def}

For an ample groupoid $\mathscr H$ and a free, proper, \'etale right $\mathscr H$-space $\Omega$ and $(\omega_1,\omega_2) \in \Omega \times_{\Omega/\mathscr H} \Omega$, we write $\langle \omega_1, \omega_2 \rangle$ for the unique $h \in \mathscr H$ satisfying $\omega_2 = \omega_1 h$. The map $\langle - , - \rangle \colon \Omega \times_{\Omega/\mathscr H} \Omega \to \mathscr H$ is continuous~\cite[Lemma~3.4]{AKM22}.

One formulation of the notion of self-similar groups is via a proper correspondence from a group to itself, see Nekrashevych~\cite{selfsimilar} where the term `covering bimodule'  is used.

\begin{Def}[Self-similar groupoid action via correpondences]\label{def:v.corr}
A \textit{self-similar groupoid action} $(G,\mathcal X)$ consists of a a discrete groupoid $G$ and an \'etale correspondence $\mathcal X\colon G\to G$ with anchor maps $\ran,\sour\colon \mathcal X \to G^0$. We may refer to $\mathcal X$ as a \textit{self-similarity} of $G$. The self-similar groupoid action is called row finite if $\mathcal X$ is proper. An object/vertex $v\in G^0$ is called regular if $0<|\ran\inv(v)/G|<\infty$; otherwise $v$ is called singular. One says a singular vertex $v$ is a source if $\ran\inv(v)=\emptyset$, and an infinite receiver otherwise.
\end{Def}

The regular objects form an invariant subset $G^0_\reg$ of the unit space and thus determine an invariant subgroupoid $G_\reg = G |_{G^0_\reg}$. We write $G^0_\sing$ for the singular objects. To be consistent with~\cite{ExelPardoSelf}, we say that $\mathcal X$ is pseudo\-free if the left action of $G$ is free.

There is a reformulation of this notion in terms of actions on graphs that can be found in~\cite[Example~4.4]{AKM22}.  

\begin{Def}[Self-similar groupoid action via graph actions]\label{def:v.graph}
A \textit{self-similar groupoid action} $(G,E,\sigma)$ is discrete groupoid $G$ whose unit space is the vertex set $E^0$ of a directed graph $\ran, \sour \colon E \to E^0$, with a left action\footnote{As noted in \cite{AKM22} this is not necessarily an action by graph partial automorphisms.} $G \acts E$ with anchor $\ran \colon E \to E^0$, written $(g,e) \mapsto g(e)$, and a $1$-cocycle $\sigma \colon G \times_{E^0} E \to G$, written $(g,e) \mapsto g |_e$, such that $\ran(g |_e) = \sour(g(e))$, $\sour(g |_e) = \sour(e)$ and $(hg)|_e = h |_{g(e)} g |_e$ (equivalently $\sigma \colon G \ltimes E \to G$ is a groupoid homomorphism with unit space map $\sigma |_E = \sour \colon E \to E^0$). The element $g|_e$ is called the section of $g$ at $e$.
\end{Def}

The translation between the two definitions is as follows.  If $(G,E,\sigma)$ is a self-similar groupoid action in the sense of Definition~\ref{def:v.graph}, the associated correspondence is $\mathcal X_{(G,E)} = E\times_{E^0} G$ with $\ran(e,g) = \ran(e)$, $\sour(e,g)=\sour(g)$, right action $(e,h)g =(e,hg)$ and left action $g(e,h) = (g(e),g|_eh)$.  Moreover, $E$ is a set of representatives for $\mathcal X_{(G,E)}/G$.

Conversely, if $(G,\mathcal X)$ is a self-similar groupoid action in the sense of Definition~\ref{def:v.corr}, then we can choose a transversal $E \subseteq \mathcal X$ to $\mathcal X/G$.
By freeness of the right action, each $x\in \mathcal X$ can be uniquely written in the form $eg$ with $e\in E$, $g\in G$.  Putting $E^0=G^0$, we can define $\ran,\sour\colon E\to E^0$ by restriction.  If $g\in G$ and $\sour(g)=\ran(e)$, then $ge=g(e)g|_e$ for a unique $g(e)\in E$ and $g|_e\in G$.  This defines a left action of $G$ with anchor $\ran$ isomorphic to the action of $G$ on $\mathcal X/G$,  
and $\sigma(g,e) = g|_e$ defines a $1$-cocycle with  $\ran(g|_e) = \sour(g(e))$, $\sour(g|_e) = \sour(e)$ and $(hg)|_e = h|_{g(e)}g|_e$.   Thus $(G,E,\sigma)$ is a self-similar groupoid action in the sense of Definition~\ref{def:v.graph} and $\mathcal X_{(G,E)}\to \mathcal X$ given by $(e,g)\mapsto eg$ is an isomorphism of correspondences.  It is easy to see that $\mathcal X_{(G,E)}$ is proper if and only if $E$ is row finite. Moreover, a vertex $v\in E^0$ is regular if and only if it is not a source and not an infinite receiver.

The action and cocycle extend by design to finite paths\footnote{We use here the convention that edges $e,f$ are composable with composition $ef$ if $\sour(e)=\ran(f)$.} (where vertices are viewed as paths of length $0$) by putting $g|_{\sour(g)}=g$, $g|_{ep} = (g|_e)|_p$ and $g(\sour(g))=\ran(e)$, $g(ep) = g(e)g|_e(p)$ for $e\in E$ and $p$ a path with $\sour(e)=\ran(p)$.  The action can be extended to infinite paths by putting $g(e_1e_2\cdots ) = g(e_1)g|_{e_1}(e_2)\cdots$.  
From the graph point of view, $(G,E,\sigma)$ is pseudo\-free if and only if $g(e)=e$ and $g|_e=\sour(e)$ implies $g=\sour(g)$.  

By a \emph{self-similar group action} we shall always mean in this paper $(G,\mathcal X)$ with $\mathcal X$ a nontrivial proper correspondence over the group $G$, equivalently $(G,E,\sigma)$ with finite alphabet $E$ of size $|E| \geq 2$, although the case of nonproper correspondences was considered in~\cite{SS23}.

\subsection{$\cs$-algebras, inverse semigroups and ample groupoids}

Given an \'etale correspondence $\mathcal X\colon G\to H$ of a discrete groupoids, we can extend the map $\langle - , - \rangle \colon \mathcal X \times_{\mathcal X/H} \mathcal X\to H$ to a map $\langle -,-\rangle\colon \mathcal X\times \mathcal X\to H\cup \{0\}$ by 
\[\langle x,y\rangle = \begin{cases} h, & \text{if}\ xh=y\\ 0, & \text{else.}\end{cases}\] 
Trivially, $\langle x,yh\rangle = \langle x,y\rangle h$, $\langle y,x\rangle = \langle x,y\rangle\inv$, $\langle x,x\rangle = \sour(x)$ and $\langle gx,y\rangle = \langle x,g\inv y \rangle$.  
For example, if $(G,E,\sigma)$ is a self-similar groupoid action in the sense of Definition~\ref{def:v.graph}, then for $\mathcal X_{(G,E)}$, one has $\langle (e,g),(f,h)\rangle = g\inv h\delta_{e,f}$.

The full $\cs$-algebra $\cs(\mathscr G)$ of an ample groupoid $\mathscr G$ is the universal completion of the groupoid algebra $\mathbb C \mathscr G$ with respect to $*$-representations~\cite{CZ24}. An \'etale correspondence $\Omega \colon \mathscr G \to \mathscr H$ induces a $\cs$-correspondence 
\[ \cs(\Omega) \colon \cs(\mathscr G) \to \cs(\mathscr H)\] between the associated full groupoid $\cs$-algebras by \cite[Section 7]{AKM22}. Antunes, Ko and Meyer show that this respects composition of correspondences and that $\cs(\Omega)$ is proper if and only if $\Omega$ is. We also note here that an open invariant subset $U \subseteq \mathscr G^0$ with complement $C = \mathscr G^0 \setminus U$ induces a short exact sequence
\[ 0 \to \cs(\mathscr G |_U) \to \cs(\mathscr G) \to \cs(\mathscr G |_C) \to 0.  \]
The only subtle point here is injectivity, which follows for example from \cite{AKM22} by considering the factorisation $\cs(\mathscr G|_U) \to \cs(\mathscr G) \to M(\cs(\mathscr G|_U))$ of the inclusion into the multiplier algebra.   

We write $M_{\mathcal X} \colon \cs(G) \to \cs(G)$ for the $\cs$-correspondence of a groupoid self-similarity $\mathcal X \colon G \to G$. The space $M_{\mathcal X}$ is densely spanned by elements $m_x$ for $x \in \mathcal X$, and $\cs(G)$ is densely spanned by the partial isometries $u_g \in \cs(G)$ for $g \in G \cup \{0\}$ with $u_0 = 0$, and thus the $\cs$-correspondence $M_{\mathcal X}$ is determined by the relations
\begin{itemize}
\item $u_g \cdot m_x = m_{gx} \delta_{\sour(g), \ran(x)}$
\item $m_x \cdot u_g = m_{xg} \delta_{\sour(x), \ran(g)}$
\item $\langle m_x, m_y \rangle = u_{\langle x,y \rangle}$
\end{itemize}
for $x,y \in \mathcal X$ and $g \in G$. The regular objects determine an ideal $\cs(G_\reg)$ in $\cs(G)$. Moreover, this ideal acts by compact operators on $M_{\mathcal X}$ because for each $v \in G^0_\reg$ the projection $u_v$ acts as the compact operator $\sum_{xG \in \ran_{\mathcal X/G} \inv (v)} m_xm_x^*$. The Hilbert module $M_{\mathcal X}$ is full if and only if there are no sinks.

\begin{Def}
The \textit{$\cs$-algebra} $\mathcal O_{\mathcal X}$ of a self-similar groupoid action $(G,\mathcal X)$ is the relative Cuntz--Pimsner algebra\footnote{See for example \cite[Section 11]{Katsura07}.} of the $\cs$-correspondence $M_{\mathcal X}$ over $\cs(G)$ with respect to the ideal $\cs(G_\reg)$, and the \textit{Toeplitz algebra} $\mathcal T_{\mathcal X}$ of $(G,\mathcal X)$ is the Pimsner--Toeplitz algebra of $M_{\mathcal X}$. Concretely, this means that 
 $\mathcal T_{\mathcal X}$ is the universal $\cs$-algebra generated by elements $u_g$ for $g \in G$ and $m_x$ for $x \in \mathcal X$ with the groupoid relations $u_gu_h = u_{gh} \delta_{\sour(g),\ran(h)}$ on $(u_g)_{g \in G}$ and the relations
\begin{enumerate}
\item[(T1)] $u_g m_x = m_{gx} \delta_{\sour(g),\ran(x)}$
\item[(T2)] $m_x u_g = m_{xg} \delta_{\sour(x),\ran(g)}$ 
\item[(T3)] $m_{x}^* m_{y} = u_{\langle x, y \rangle}$, where $u_0=0$,
\end{enumerate}
for $x,y \in \mathcal X$ and $g \in G$. The $\cs$-algebra $\mathcal O_{\mathcal X}$ has the additional relations
\begin{itemize}
\item[(CK)] $u_v = \sum_{xG \in \ran_{\mathcal X/G} \inv (v)} m_{x} m_{x}^*$ 
\end{itemize}
for $v \in G^0_\reg$. We note as in \cite{Katsura04} that (T2) is redundant. 
\end{Def}

\begin{Rmk}
We do not use Katsura's nonrelative Cuntz--Pimsner algebra \cite{Katsura04}, as it is okay for us that $\cs(G_\reg)$ may act non-faithfully on $M_{\mathcal X}$. In fact, this is crucial to our approach, because in many of our applications $G_\reg$ does not even act faithfully on $\mathcal X$. 
\end{Rmk}

We take a moment to spell out what this means in the graph picture. For a right $G$-transversal $E \subseteq \mathcal X$ the elements $(m_e)_{e \in E} \subseteq M_{\mathcal X}$ generate $M_{\mathcal X}$ as a Hilbert $\cs(G)$-module, so the above presentation reduces to the following.

\begin{Prop}
Let $(G,E,\sigma)$ be a self-similar groupoid action with correspondence $\mathcal X = \mathcal X_{G,E} \colon G \to G$. The Toeplitz algebra $\mathcal T_{(G,E)} = \mathcal T_{\mathcal X}$ of $(G,E,\sigma)$ is the universal $\cs$-algebra generated by elements $u_g$ for $g \in G$ and $m_e$ for $e \in E$ with the groupoid relations $u_gu_h = u_{gh} \delta_{\sour(g),\ran(h)}$ on $(u_g)_{g \in G}$ and the relations
\begin{enumerate}
\item[(T1)] $u_g m_e = \delta_{\sour(g),\ran(e)} m_{g(e)} u_{g |_e}$
\item[(T3)] $m_{e}^* m_{f} = \delta_{e,f} u_{\sour(f)}$
\end{enumerate}
for $e,f \in E$ and $g \in G$. The $\cs$-algebra $\mathcal O_{(G,E)} = \mathcal O_{\mathcal X}$ has the additional relations
\begin{itemize}
\item[(CK)] $u_v = \sum_{\ran(e) = v} m_{e} m_{e}^*$ for each regular vertex $v \in E^0_\reg$.
\end{itemize}
\end{Prop}

As in~\cite{ExelPardoSelf,Nekcstar} we construct a groupoid model for $\mathcal O_{\mathcal X}$ (or, equivalently, $\mathcal O_{(G,E)}$), which is moreover the tight groupoid of an inverse semigroup. Let $(G,\mathcal X)$ be a self-similar groupoid action.  We first define an inverse semigroup $S_{\mathcal X}$ with $0$.  We then show that if we convert $(G,\mathcal X)$ to $(G,E,\sigma)$, with $E$ a transversal to $\mathcal X/G$, we obtain the familiar inverse semigroup.  This has the advantage that it depends only on the correspondence and not the choice of $E$.  Also, its universal property is directly apparent.

A $\ast$-representation of a discrete groupoid $G$ in a $\ast$-semigroup $S$ with $0$ is a  zero-preserving homomorphism $\pi\colon G\cup \{0\}\to S$, where $G\cup \{0\}$  is the inverse semigroup obtained from $G$ by declaring all undefined products in $G$ to be $0$.   A representation of the correspondence $\mathcal X$ over $G$ in $S$ is a pair $(\pi,t)$ where $\pi$ is a $\ast$-representation of $G$ in $S$ and $t\colon \mathcal X\to S$ is a map such that:
\begin{enumerate}
    \item $\pi(g)t(x) = t(gx)\delta_{\sour(g),\ran(x)}$.
    \item $t(x)\pi(g) = t(xg)\delta_{\sour(x), \ran(g)}$.
    \item $t(x)^*t(y) = \pi(\langle x,y\rangle)$.
\end{enumerate}
We now construct the universal representation of $\mathcal X$ in a $\ast$-semigroup, which will turn out to be an inverse semigroup.  

Let $\mathcal X^0 = G$ and $\mathcal X^{n+1} = \mathcal X^n\circ \mathcal X= \mathcal X^n\times_G \mathcal X$. Put $\mathcal X^+ = \bigsqcup_{n\geq 0} \mathcal X^n$.\footnote{We reserve the notation $\mathcal X^\ast$ for the adjoint bispace. The notation $\mathcal X^+$ is typically used for the free semigroup on $\mathcal X$; our usage mainly differs in that we allow empty paths.} Then $\mathcal X^+$ is naturally a category with object set $G^0$, source and range map inherited from the anchors of the $\mathcal X^n$ and composition $\mathcal X^+\times_{G^0} \mathcal X^+\to \mathcal X^+$ given by  $\mathcal X^n\times_{G^0} \mathcal X^m\to \mathcal X^n\times_{G} \mathcal X^m\cong \mathcal X^{n+m}$.  Associativity is immediate from the associativity of composition of correspondences. One can verify that $\mathcal X^+$ is left cancellative and singly-aligned, and that it is cancellative if and only if $\mathcal X$ is pseudo\-free, but we shall not need this fact.  

For example, if $G=G^0$, then $\ran,\sour\colon \mathcal X\to G^0$ is just a graph, and $\mathcal X^+$ is the path category of $\mathcal X$.  We put $|p|=n$ if $p\in \mathcal X^n$.  Length is a functor $\mathcal X^+\to \mathbb N$.  If $p,q\in \mathcal X^+$ and $g\in G$ with $\sour(p)=\ran(g)$, $\sour(g)=\ran(q)$, then $pgq$ is defined unambiguously as $(pg)q=p(gq)$.  Since $\mathcal X^n$ is a correspondence, we have that $\langle p,q\rangle$ makes sense for $|p|=|q|$.

Define $S_{\mathcal X} = \{0\} \cup (\mathcal X^+ \times_G (\mathcal X^+)^\ast)$.  Writing  $pq^*$ for the element $[p,q^*]_G$ given $p,q \in \mathcal X^+$ with $\sour(p) = \sour(q)$, products of nonzero elements are given by 
\[p_1 q_1^* p_2 q_2^* =\begin{cases} p_1 \langle q_1 , r \rangle s q_2^*, & \text{if}\ p_2 = rs, |r|=|q_1|, \\ p_1 (q_2 \langle p_2, r \rangle s)^*  &\text{if}\ q_1 = rs, |r|=|p_2|,\end{cases}\] where we interpret $0 u^* = 0 = v 0^*$. It is straightforward to verify that the product is well defined and that $S_{\mathcal X}$ is an inverse semigroup where $(pq^*)^* = qp^*$ and $0^*=0$. The nonzero idempotents are the elements of the form $pp^*$ with $p\in \mathcal X^+$. We may safely write $p = p \sour(p)^*$ for $p \in \mathcal X^+$, noting that $g^* = (g \sour(g)^*)^* = \sour(g) g^* = g \inv \sour(g\inv)^* = g \inv$ for $g \in \mathcal X^0 = G$. Then for $p,q \in \mathcal X^+$ we have $pq = p \langle \sour(p), \ran(q) \rangle q = pq \delta_{\sour(p),\ran(q)}$ and $p^*q = \langle p, q\rangle $ when $|p| = |q|$. Note also that for $g,h \in G$ and $p\in \mathcal X^+$, taking products $gp$ and $ph$ in $S_{\mathcal X}$ agrees with performing the actions on $\mathcal X^+$. One can verify that this is the inverse hull of $\mathcal X^+$, but we shall not need this fact.
The inverse semigroup $S_\mathcal X$ is $E^*$-unitary if and only if $\mathcal X$ is pseudo\-free.  

\begin{Thm}\label{t:toeplitz.c}
Let $(G,\mathcal X)$ be a self-similar groupoid with correspondence $\mathcal X$.  Define $\pi\colon G\cup \{0\}\to S_{\mathcal X}$ and $t\colon \mathcal X\to S_{\mathcal X}$ by $\pi(g) = g$ and $t(x) = x$.  Then $(\pi,t)$ is the universal representation of $\mathcal X$ in a $\ast$-semigroup. 
\end{Thm}
\begin{proof}
It is immediate that $\pi$ is a $\ast$-representation.   We compute $\pi(g)t(x) = g x = t(g x) \delta_{\sour(g),\ran(x)}$, and similarly we have $t(x)\pi(g) 
= t(x g)\delta_{\sour(x),\ran(g)}$.  Also, $t(x)^*t(y) =x^*y= \langle x,y\rangle = \pi(\langle x,y\rangle)$.  Thus $(\pi,t)$ is a representation of $\mathcal X$. 

Suppose that $(\pi',t')$ is a representation of $\mathcal X$ in a $\ast$-semigroup $T$. We can define an extension $t'\colon \mathcal X^+\to T$ by putting $t'(g) = \pi'(g)$ for $g\in G=\mathcal X^0$ and $t'(x_1\cdots x_n)= t'(x_1)\cdots t'(x_n)$.  This is well defined because $t'(xg)t'(g\inv y) = t'(x)\pi'(g)\pi'(g\inv)t'(y) = t'(x)t'(y)$.

Define $\tau\colon S_{\mathcal X}\to T$ by $\tau(pq^*) = t'(p)t'(q)^*$ and $\tau(0)=0$.  This is well defined because $t'(pg)t'(qg)^* = t'(p)\pi'(g)(t'(q)\pi'(g))^* = t'(p)t'(q)^*$.  Notice that $\tau(\pi(g)) = \tau(g) = t'(g)t'(\sour(g))^* = \pi'(g)$ and $\tau(t(x)) = \tau(x) = t'(x)t'(\sour(x))^* = t'(x)$.  Now we check that $\tau$ is a $\ast$-homo\-mor\-phism.
Note that $\tau((pq^*)^*) = \tau(q p^*)= t'(q)t'(p)^* = \tau(p q^*)^*$.  
One  shows by induction on length that if $|u|=|v|$, then $t'(u)^*t'(v) = \pi'(\langle u,v\rangle)$ with the case $|u|=0=|v|$ being trivial.  Else note that if $x,y\in \mathcal X$ and $u,v\in \mathcal X^n$, then $xuG =yvG$ if and only if $xG=yG$ and $uG=\langle x,y\rangle vG$ by left cancellativity of $\mathcal X^+$.  But then $xu\langle u, \langle x,y\rangle v\rangle = x\langle x,y\rangle v = yv$, and by induction $t'(xu)^*t'(yx) = t'(u)^*t'(x)^*t'(y)t'(v) = t'(u)\pi'(\langle x,y\rangle)t'(v) = t'(u)t'(\langle x,y\rangle v)=\pi'(\langle u,\langle x,y\rangle v\rangle)=\pi'(\langle xu,yv\rangle)$, as required.   Suppose that $p_1 q_1^*, p_2 q_2^*\in S_{\mathcal X}$.  Then either $|q_1|\leq |p_2|$ or $|q_1|\geq |p_2|$.  We handle just the first case, as the second is dual.  Write $p_2 = rs$ with $|r|=|q_1|$.  Then 
\begin{align*}
\tau(p_1 q_1^*)\tau(p_2 q_2^*) & = t'(p_1)t'(q_1)^*t'(p_2)t'(q_2)^*=t'(p_1)t'(q_1)^*t'(r)t'(s)t'(q_2)^* \\ & = t'(p_1)\pi'(\langle q_1,r\rangle) t'(s)t'(q_2)^* = \tau( p_1 \langle q_1 ,r\rangle s q_2^*)
\end{align*}
This completes the proof that $\tau$ is a $\ast$-homomorphism. Uniqueness follows from the observation that  $\pi(G)\cup t(\mathcal X)$ generates $S_{\mathcal X}$.  
\end{proof}

Next we show that if we use the graph theoretic formulation of self-similar groupoids, then we obtain the natural analogue of the inverse semigroup considered by Nekrashevych~\cite{Nekcstar} and Exel and Pardo~\cite{ExelPardoSelf,Self-similar-arb-graph}.

The inverse semigroup $S_{(G,E)}$ associated to a self-similar groupoid action $(G,E,\sigma)$ as per Definition~\ref{def:v.graph} consists of a zero element $0$ and all triples of the form $(p,g,q)$ where $g\in G$ and $p,q\in E^+$ are paths with $\sour(p)=\ran(g)$ and $\sour(q)=\sour(g)$.   The product of nonzero elements is given by
\[(p,g,q)(p',h,q')=\begin{cases} (pg(r),g|_rh,q'), & \text{if}\ p'=qr,\\ (p,g h|_{h\inv(s)},q'h\inv(s)), & \text{if}\ q=p's,\\ 0, & \text{else.}\end{cases}\] 
The involution is given by $(p,g,q)^* = (q,g\inv ,p)$.  Note that elements of the form $(p,v, q)$ with $v\in E^0$ and $\sour(p)=v=\sour(q)$ form an inverse subsemigroup isomorphic to the graph inverse semigroup~\cite{Paterson} $S_E$ of $E$.  The idempotents are the elements of the form $(p,\sour(p),p)$ and the maximal subgroup at this idempotent consists of all elements of the form $(p,g,p)$ with $g\in G_{\sour(p)}^{\sour(p)}$.  If we write $p$ as shorthand for $(p,\sour(p),\sour(p))$ and $g$ as shorthand for $(\ran(g),g,\sour(g))$, then $(p,g,q) =pgq^*$.

\begin{Prop}
Given a self-similar groupoid action $(G,\mathcal X)$ and a transversal $E$ to $\mathcal X/G$, there is an isomorphism $S_{\mathcal X}\cong S_{(G,E)}$.    
\end{Prop}
\begin{proof}
Observe that $E^+$ is a transversal for the correspondence $\mathcal X^+$ by a standard argument (cf.~\cite{selfsimilar}).  Thus each element of $\mathcal X^+$ can be written uniquely as $pg$ with $p\in E^+$ and $g\in G$ with $\sour(p)=\ran(g)$.    Define $\tau\colon S_{(G,E)}\to  S_{\mathcal X}$ by  $(p,g,q)\mapsto pgq^* = p (qg\inv)^*$.  The inverse sends $ pg (qh)^*$ to $(p,gh\inv,q)$ (for $p,q\in E^+$, $g,h\in G$).  If $p'=qr$, then $\tau(p,g,q)\tau(p',h,q')= pgq^*  p' (q'h\inv)^* =  pgr (q'h\inv)^* = pg(r)g|_rh (q')^* = \tau((p,g,q)(p',h,q'))$.  If $q=p's$, then similarly the identity $(h \inv)s = h \inv(s) (h \inv)|_s$ implies that $\tau(p,g,q)\tau(p',h,q') 
=\tau((p,g,q)(p',h,q'))$.  Finally, if neither of these cases hold, then  the prefixes $q_0,p_0$ of $q,p$ of length $\min\{|p|,|q|\}$ are not equal.  Therefore, $\langle q_0,p_0\rangle =0$, and so $\tau(p,g,q)\tau(p',h,q') =  pg q^* p' (q'h\inv)^* =0$.  This completes the proof.
\end{proof}

For an inverse semigroup $S$ with $0$ with idempotent semilattice $E$, the \textit{underlying groupoid} $\underlying{S}$ of $S$ is the discrete groupoid of nonzero
elements $S^\times$ in $S$ with range and source maps $\ran, \sour \colon \underlying{S} \to E^\times$ given by $\sour(s)=s^*s$, $\ran(s)=ss^*$, composition given by multiplication in $S$ and the inversion by the involution in $S$. This may also be viewed as the transformation groupoid\footnote{Called by some authors groupoid of germs. We use groupoid of germs for the transformation groupoid of the pseudogroup generated by a semigroup of partial homeomorphisms.} of the conjugation action of $S$ on $E^\times$.

The \textit{universal groupoid} $\universal{S}$ of $S$ is the transformation groupoid of the conjugation action of $S$ on the filter/character space $\wh E$. Here $\wh E$ is the space of characters of $E$, that is,  nonzero, zero-preserving semigroup homomorphisms $\chi\colon E\to \{0,1\}$, equipped with the topology of pointwise convergence.  These are precisely the characteristic functions of filters (proper, nonempty, upward-closed  subsemigroups of $E$).  In particular, the characteristic function $\chi_e$ of the principal filter generated by $e$ belongs to $\wh E$.  We shall use the terms `filter' and `character' interchangeably.  

The \textit{tight} groupoid $\mathscr G_S$ of $S$ is the reduction of $\universal{S}$ to the space of tight filters/characters. A character $\chi$ is \emph{tight} in the sense of Exel if $\chi(e)=\bigvee_{i=1}^n\chi(e_i)$ whenever $e_1,\ldots, e_n\leq e$ cover $e$ in the sense that $0\neq f\leq e$ implies $fe_i\neq 0$ for some $i=1,\ldots, n$.  The subspace of tight filters is closed and invariant and can be described as the closure of the space of ultrafilters in $\wh E$. See~\cite{Exel} for details.

The \textit{full} or \textit{universal $\cs$-algebra} $\cs(S)$ of $S$ is the universal $\cs$-algebra generated by elements $t_s$ for $s \in S$ satisfying $t_0 = 0$, $t_s^* = t_{s^*}$ and $t_s t_{s'} = t_{ss'}$ for $s,s' \in S$. There is an isomorphism $\cs(S) \to \cs(\universal{S})$ by~\cite[Theorem 4.4.1]{Paterson},\footnote{We warn the reader that Paterson's notions of universal groupoid and $*$-representation of an inverse semigroup do not respect the $0$ of the inverse semigroup. However, the isomorphism still follows from his theorem, because not respecting the $0$ just means that each $\cs$-algebra has an extra copy of $\mathbb C$ as a direct summand which is respected by the isomorphism.} which sends $t_s \in \cs(S)$ to the indicator on the compact open set $\{[s,\chi] \in \universal{S} \mid \chi(s^*s) = 1\}$. 

The \textit{tight $\cs$-algebra} $\cs_{\mathrm{tight}}(S)$ is the universal $\cs$-algebra generated by elements $w_s$ for $s \in S$ satisfying $w_0 = 0$, $w_s^* = w_{s^*}$, $w_s w_t = w_{st}$ for $s,t \in S$ and $w_e=\prod_{i=1}^n(w_e-w_{e_i})$ whenever $e_1,\ldots, e_n\leq e$ cover $e$.  There is an isomorphism $\cs_{\mathrm{tight}}(S)\to \cs(\mathscr G_S)$ that sends $w_s$ to the indicator function of the compact open set $\{[s,\chi] \in \mathscr G_S \mid \chi(s^*s) = 1\}$.  This follows from~\cite[Corollary 2.14]{simplicity} and~\cite[Proposition~5.2]{CZ24}.

For a self-similar groupoid action $(G,\mathcal X)$ the universal groupoid of $S_{\mathcal X}$ will be written $\mathscr U_{\mathcal X}$ and the tight groupoid will be denoted by $\mathscr G_{\mathcal X}$. If it is presented by $(G,E,\sigma)$, we may write  $\mathscr G_{(G,E)}$. Note that since the idempotents of $S_{\mathcal X}$ are those of $S_E$ for any graph realisation $E \subseteq \mathcal X$, it follows that the filters of $S_{\mathcal X}$ are the same as those of $S_E$, and in particular the space of tight filters  is homeomorphic to the \emph{boundary path space} $\partial E$, consisting of all infinite paths in $E$ and finite paths in $E$ whose source is a singular vertex. 
The filter corresponding to a path $z\in \partial E$ is the set of all $pp^*$ with $p$ a finite prefix of $z$.  The topology on $\partial E$ has basis all sets of the form $p\partial E\setminus \bigcup_{e\in F}pe\partial E$, where $F$ is a finite set of edges $e$ with $\ran(e)=\sour(p)$.  It is shown in~\cite[Proposition~4.11]{steinberg2022stable}\footnote{This is also a consequence of the uniqueness theorem proof that $\cs(E)\cong \cs(\mathscr G_{S_E})$.} that the relations of $\cs_{\mathrm{tight}}(S_{\mathcal X})$ obtained from tight covers of idempotents are all in the ideal of $\cs(S_{\mathcal X})$ generated by elements of the form $t_v-\sum_{\ran(e)=v} t_et_e^*$ where $v \in G^0_\reg$ is a regular vertex.  
Since $S_{\mathcal X}$ is $E^*$-unitary when $\mathcal X$ is pseudo\-free, it follows that $\mathscr U_{\mathcal X}$ and $\mathscr G_{\mathcal X}$ are Hausdorff in the pseudo\-free setting.

Putting together Theorem~\ref{t:toeplitz.c} and the above discussion we obtain:

\begin{Prop}\label{Toeplitz isomorphism}
Let $(G,\mathcal X)$ be a self-similar groupoid action. There is an isomorphism $\mathcal T_{\mathcal X} \to \cs(S_{\mathcal X})$ sending $u_g$  to $v_g \in \cs(S_{\mathcal X})$ for each $g \in G$ and sending $m_x \in M_{\mathcal X}$ to $v_x \in \cs(S_{\mathcal X})$ for each $x \in \mathcal X$. Moreover, this isomorphism identifies the quotient $\mathcal O_{\mathcal X}$ with the tight $\cs$-algebra $\cs_{\mathrm{tight}}(S_{\mathcal X})$ 
of $S_{\mathcal X}$. That is, we have a commutative diagram 
\[\begin{tikzcd}[arrow style=math font]
	{\mathcal T_{\mathcal X}} & {\cs(S_{\mathcal X})} & {\cs(\mathscr U_{\mathcal X})} \\
	{\mathcal O_{\mathcal X}} & {\cs_{\mathrm{tight}}(S_{\mathcal X})} & {\cs(\mathscr G_{\mathcal X})}.
	\arrow["\cong", from=1-1, to=1-2]
	\arrow[two heads, from=1-1, to=2-1]
	\arrow["\cong", from=1-2, to=1-3]
	\arrow[two heads, from=1-2, to=2-2]
	\arrow[two heads, from=1-3, to=2-3]
	\arrow["\cong", from=2-1, to=2-2]
	\arrow["\cong", from=2-2, to=2-3]
\end{tikzcd}\]
\end{Prop}

The kernel of the quotient map $\mathcal T_{\mathcal X} \to \mathcal O_{\mathcal X}$ is Morita equivalent to $\cs(G_\reg)$. On the groupoid level, the kernel corresponds to the reduction $\universal{\mathcal X} |_U$ of the universal groupoid $\universal{\mathcal X}$ to the complement $U$ 
of the tight filters. Each filter in $U$ corresponds to a finite path $p$ beginning at a regular vertex, which means it is the principal filter $\chi_{pp^*}$ of the idempotent $pp^* \in S_{\mathcal X}$; these principal filters are isolated points. Let $F = \{ p p^* \in S_{\mathcal X} \mid p \in \mathcal X^+ \text{ a finite path with } \sour(p) \in G^0_\reg \}$. This is an invariant set of nonzero idempotents, and $\universal{\mathcal X} |_U$ is isomorphic to the reduction $\underlying{S_{\mathcal X}} |_F$ of the underlying groupoid $\underlying{S_{\mathcal X}}$ of $S_{\mathcal X}$. The regular vertices $v \in G^0_\reg$ form  a transversal for $\underlying{S_{\mathcal X}} |_F$, and therefore the inclusion $G_\reg \hookrightarrow \underlying{S_{\mathcal X}} |_F$ is a Morita equivalence.

We now show that Morita equivalent correspondences give rise to Morita equivalent inverse semigroups and hence Morita equivalent universal and tight groupoids. A correspondence $\mathcal X$ over $G$ is Morita equivalent to a correspondence $\mathcal Y$ over $H$ if there is a Morita equivalence $\Omega\colon G\to H$ such that $\Omega\times_H \mathcal Y\cong \mathcal X\times_G \Omega$. This is an equivalence relation on correspondences.

Let us recall the notion of Morita equivalence of inverse semigroups~\cite{strongmorita,FunkLawsonSteinberg}.  The characterization most useful for our purposes is the following.  Associated to any inverse semigroup $S$ is a left cancellative category $L(S)$.  Here $L(S)^0=E(S)=E$ and $L(S) = \{(f,s)\in E\times S\mid fs=s\}$.  One has $\sour(f,s)=s^*s$ and $\ran(f,s)=f$.  The product is given by $(f,s)(s^*s,t) = (f,st)$.   Note that the underlying groupoid $\underlying{S}$ embeds as the groupoid of isomorphisms of $L(S)$ via $s\mapsto (ss^*,s)$.  The inverse semigroups $S$ and $T$ are Morita equivalent if $L(S)$ is equivalent to $L(T)$ as categories.  Note that if $S$ has a zero, then $0$ is the unique initial object of $L(S)$.   Let $L(S^\times)=L(S)|_{E^\times}$.  Since equivalences preserve initial objects,  it easily follows that inverse semigroups with zero $S,T$ are Morita equivalent if and only if $L(S^\times)$ is equivalent to $L(T^\times)$.

\begin{Prop}\label{p:morita.equiv.corr}
Suppose that the correspondence $\mathcal X$ over $G$ is Morita equivalent to the  correspondence $\mathcal Y$ over $H$.  Then $S_{\mathcal X}$ and $S_{\mathcal Y}$ are Morita equivalent, and hence $\universal{\mathcal X}$ and $\universal{\mathcal Y}$ are Morita equivalent and $\mathscr G_{\mathcal X}$ and $\mathscr G_{\mathcal Y}$ are Morita equivalent.
\end{Prop}
\begin{proof}
First we claim that $L(S_{\mathcal X}^\times)$ is equivalent to $\mathcal X^+$.  Indeed, the underlying groupoid $D_{\mathcal X}$ embeds as the groupoid of isomorphisms of  $L(S_{\mathcal X}^\times)$, and in $D_{\mathcal X}$ each idempotent $pp^*$ is isomorphic to $\sour(p)\in G^0$.  Thus $L(S_{\mathcal X}^\times)$ is equivalent to the full subcategory on $G^0$.  But $L(S_{\mathcal X}^\times)|_{G_0}= \{(\ran(p),p)\mid p\in \mathcal X^+\}\simeq \mathcal X^+$.  Similarly, $L(S_{\mathcal Y}^\times)\simeq \mathcal Y^+$, and so it suffices to show that $\mathcal X^+$ is equivalent to $\mathcal Y^+$.  

Let $\Omega \colon G \to H$ be a Morita equivalence intertwining $\mathcal X$ and $\mathcal Y$. We obtain a $G$-bispace isomorphism $\mathcal X \cong \Omega \times_H \mathcal Y \times_H \Omega^*$ which, as $\Omega$ is a Morita equivalence, induces an isomorphism $\mathcal X^+ \cong \Omega \times_H \mathcal Y^+ \times_H \Omega^*$ of categories, where $\Omega \times_H \mathcal Y^+ \times_H \Omega^*$ has objects $\Omega/H$ and multiplication $[v, p, w^*] [w,q,z^*] = [v,pq,z^*]$.
If we pick an $H$-transversal $T$ in $\Omega$, then $\Omega \times_H \mathcal Y^+ \times_H \Omega^*\cong T\times_{H^0} \mathcal Y^+\times_{H^0} T^*$ and the projection to the middle coordinate is an equivalence  $\Omega \times_H \mathcal Y^+ \times_H \Omega^* \simeq \mathcal Y^+$. The key point is that $\sour\colon \Omega\to H^0$ is surjective, and so if $w\in H^0$, then $w=\sour(th)$ with $t\in T$ and $h\in H$.  Then $h\colon w\to \sour(t)$ is an isomorphism in $\mathcal Y^+$, showing that the projection is essentially surjective.

It is shown in~\cite{strongmorita} that a Morita equivalence of inverse semigroups with zero induces a Morita equivalence of universal groupoids that restricts to a Morita equivalence of tight groupoids.  The result follows.  
\end{proof}

\subsection{Faithful quotients of self-similar groupoids}

A self-similar group\-oid action $(G,\mathcal X)$ is \emph{faithful} if the left action of $G$ on $\mathcal X^+/G$ is faithful. The \emph{kernel} of the action is $N=\bigsqcup_{v\in G^0} N^v_v$ where $N^v_v=\{g\in G^v_v\mid gpG=pG, \forall p\in \mathcal X^+\}$. Then the action is faithful if and only if $N=G^0$. Note that $N$ is closed under conjugation and $Nx\subseteq xN$ for all $x\in \mathcal X$.  If $E$ is a right $G$-transversal,  then $(G,\mathcal X)$ is faithful precisely when the action of $G$ on $E^+$ is faithful.  Moreover, if $g\in N$, then $g|_x\in N$ for all $x$ with $\sour(g)=\ran(x)$. 

The quotient $G/N$ (which makes sense as $N$ is closed under conjugation and contains $G^0$) identifies $g,h\in G$ if they act the same on the left of $\mathcal X^+/G$ (but identifies no distinct objects). Then $\mathcal X/N$ is a correspondence over $G/N$, where $G/N$ acts on the left of $\mathcal X/N$ by $gNxN=gxN$.   If $E$ is a right $G$-transversal, then $\mathcal X/N \cong E\times_{G^0} G/N$ with left action given by $gN(e, hN) = (g(e), g|_eN\cdot hN)=(g(e),g|_ehN)$.   It is straightforward to verify that the self-similar groupoid $(G/N,\mathcal X/N)$ is faithful since if $E$ is a right $G$-transversal, then one checks inductively that $\ov g(p) = g(p)$  for $p\in E^+$.  

For a self-similar group action $(G,E,\sigma)$ over a finite alphabet $E$, any element which fixes $E^n$ for some integer $n \geq 0$ with $g|_w = 1$ for all $w \in E^n$ must be in the kernel of the action on $E^+ \cong \mathcal X_{(G,E)}^+/G$. Nekrashevych observed~\cite{Nekcstar} that if every element of the kernel satisfies this property for some $n \geq 0$, then the $\cs$-algebras associated to $(G,E,\sigma)$ and $(G/N,E,\sigma)$ are isomorphic. We shall see in fact that the groupoids are isomorphic.

Let us suppose that $\mathcal X$ is a self-similarity of $G$.  Then any element $g\in G$ that fixes $\sour(g)\mathcal X^n$ also fixes $\sour(g) \mathcal X^m$ for all $m\geq n$. If it also fixes all elements $p \in \sour(g)\mathcal X^k$ with $k < n$ and $\sour(p) \in G^0_{\sing}$ then it fixes $\sour(g) \mathcal X^+/G$. 

\begin{Def}
Let $(G,\mathcal X)$ be a self-similar action with kernel $N$. For each $n \geq 0$ consider the subgroupoid
\[ K_n = \left\{g\in G \suchthat gp=p, \forall p\in \sour(g)\mathcal X^n \sqcup \bigsqcup_{0 \leq k < n} \sour(g) \mathcal X^k G^0_\sing \right\}\]
and note that $K_0 \subseteq K_1 \subseteq \dots \subseteq N$. We call
$ K = \bigcup_{n \geq 0} K_n $
the \emph{tight kernel} of the action. We say $(G,\mathcal X)$ is \emph{loosely faithful} if $K = N$.
\end{Def}
It is straightforward to check that $G_\sing \cap K = G^0_\sing$ and that $x K_n \subseteq K_{n-1} x$ for $x \in \mathcal X$.

If $E$ is a right $G$-transversal for $\mathcal X$, we say $g \in G$ \emph{strongly fixes} $p \in \sour(g)E^n$ if $g(p) = p$ and $g|_p = \sour(p)$; this simply means $gp=p$ in $\mathcal X$. Thus $g\in K_n$ if and only if $g$ strongly fixes all $p\in \sour(g)E^n$ and all $p \in \sour(g)E^k$ with $\sour(p) \in G^0_\sing$ for $k < n$. Moreover, one may check that if we put $K_0=G^0$, then
\[ K_{n+1} = G^0_\sing \sqcup \{ g \in G_\reg \mid g(e) = e, g|_e \in K_n, \forall e \in \sour(g)E \}, \]
and so $K$ is the smallest subgroupoid containing $G^0$ such that if $g \in G_\reg$ with $g(e) = e$ and $g|_e \in K$ for each $e \in \sour(g)E$, then $g \in K$.
The quotient $G/K$ is a well-defined groupoid (as $K$ consists of isotropy, is closed under conjugation and contains $G^0$) and $\mathcal X/K$ is a correspondence over $G/K$ with left action $gK(xK) = gxK$.  If $E$ is a right $G$-transversal, then $\mathcal X/K\cong E\times_{G^0} G/K$ with left action $gK(e,hK) = (g(e), (g|_eK)hK)=(g(e),g|_ehK)$.

The proof of the following lemma is routine so we omit it.

\begin{Lemma}\label{l:functoriality.of.sorts}
Let $\p\colon S\to T$ be a surjective homomorphism of inverse semigroups, and suppose that $T$ has a nondegenerate action by partial homeomorphisms on a locally compact Hausdorff space $X$.  Then there is a surjective \'etale homomorphism $\Phi\colon S\ltimes X\to T\ltimes X$ given by $\Phi([s,x]) = [\p(s),x]$, that restricts to a homeomorphism of unit spaces.     
\end{Lemma}

We now show that the groupoid of a self-similar action depends only on the quotient by the tight kernel. 

\begin{Thm}\label{t:nek.cond}
Let $(G,E,\sigma)$ be a self-similar groupoid action.   Let $N$ denote the kernel of the action of $G$ on $E^+$ and let $K$ denote the tight kernel.  Then there is an isomorphism $f\colon \mathscr G_{(G,E)}\to \mathscr G_{(G/K,E)}$.  In particular, if $(G,E,\sigma)$ is loosely faithful, then $\mathscr G_{(G,E)}\cong \mathscr G_{(G/N,E)}$. 
\end{Thm}
\begin{proof}
If $(H,F,\sigma)$ is a self-similar groupoid over a graph $F$, then $\mathscr G_{(H,F)}$ is isomorphic to the transformation groupoid of the action of $S_{(H,F)}$ on the boundary path space $\partial F$.  Here $0$ acts as the empty map, and if $p,q\in F^+$ and $h\in H$, then $phq^*$ has domain the cylinder set $q\partial F$, range the cylinder set $p\partial F$ and action given by $phq^*(qw) = ph(w)$.  The isomorphism sends the germ $[s,w]$ with $p\in \partial F$ to $[s,\chi_w]$ where $\chi_w$ is the filter of all $pp^*$ with $p$ a finite prefix of $w$.  

Put $\ov G=G/K$ and write $\ov g$ for $gK$.  Note that $\ov g|_e = \ov{g|_e}$.  
In our setup, we have a surjective homomorphism $\p\colon S_{(G,E)}\to S_{(\ov G,E)}$ that is bijective on idempotents, and the action of $S_{(G,E)}$ on $\partial E$ factors through that of $S_{(\ov G,E)}$.  We then have an induced surjective \'etale homomorphism $\Phi\colon \mathscr G_{(G,E)}\to \mathscr G_{(\ov G,E)}$ with $\Phi([s,w]) = [\p(s),w]$ by Lemma~\ref{l:functoriality.of.sorts}.  It remains to show that $\Phi$ is injective.  Since $\Phi$ is the identity on the unit space, we must show that if $\Phi([s,w])$ is a unit, then $[s,w]$ is a unit. 
Write $s=ugv^*$.  Then $w$ must begin with $v$, so write $w=vz$.  We have $[u\ov gv^*,w]=[1,w]$.  Thus we can find a find a prefix $z_0$ of $z$ with $u\ov gv^*vz_0(vz_0)^* = vz_0(vz_0)^*$. Moreover, if $|z|<\infty$, we may assume that $z=z_0$.   But  $u\ov gv^*vz_0(vz_0)^* = u\ov g(z_0)\cdot \ov{g|_{z_0}}\cdot (vz_0)^*$.  Therefore, $u\ov g(z_0)=vz_0$ and $\ov {g|_{z_0}} =\sour(z_0)$. Since $\ov g(z_0)=g(z_0)$,
this means that    $ug(z_0) = vz_0$ and  $g|_{z_0}\in K$.  Note that since $|g(z_0)|=|z_0|$, it follows that $u=v$ and $g(z_0)=z_0$.  Thus $s=vgv^*$, $g(z_0)=z_0$ and $g|_{z_0}\in K$.  Suppose first that $z=z_0$ with $\sour(z_0)=\sour(g|_{z_0})$ singular. Then since $K \cap G_\sing = G^0_\sing$, we have $g|_{z_0} = \sour(z_0)$ and so
$svz_0(vz_0)^*  = vz_0g|_{z_0}(vz_0)^* = vz_0(vz_0)^*$, and hence $[s,w]$ is a unit.  On the other hand, if $z$ is infinite, then
by assumption, we can find $n\geq 0$ so that $g|_{z_0}(r)=r$ and  $(g|_{z_0})|_r=1$ for all $r\in \sour(z_0)E^n$. Let $z=z_0z_1z'$ with $|z_1|=n$.  Then $w\in vz_0z_1\partial E$ and $svz_0z_1(vz_0z_1)^* = vgz_0z_1(vz_0z_1)^* = vz_0g|_{z_0}z_1(vz_0z_1)^*= vz_0g|_{z_0}(z_1)(g|_{z_0})|_{z_1}(vz_0z_1)^* = vz_0z_1(vz_0z_1)^*$ by choice of $n$.  It follows that $[s,w]=[vz_0z_1(vz_0z_1)^*,w]$ is a unit, as required. 
\end{proof}

A key notion in the theory of self-similar groups is that of contraction.  The contraction phenomenon was discovered by Grigorchuk and formalized in~\cite{selfsimilar}.   A generalization of the notion for self-similar actions of groupoids on finite graphs without sources was given in~\cite{BBGHSW24}.  The groupoid associated to any self-replicating finitely contracting group is amenable~\cite{Nekcstar}, and we shall see the same is true for all contracting self-similar groupoid actions in Corollary~\ref{c:contracting.amenable}.

\begin{Def}[Contracting action]
A self-similar groupoid action $(G,E,\sigma)$ on a finite graph $E$ without sources is contracting if there is a finite subset $F\subseteq G$ such that, for all $g\in G$, there exists $n\geq 0$ such that $g|_p\in F$ for all $p\in \sour(g)E^+$ with $|p|\geq n$.  An action being contracting depends only on the correspondence and not on the choice of transversal $E$, cf.~\cite[Proposition~4.5.4]{Nekbook2}. 
\end{Def}

There is a unique smallest choice of $F$ (depending on $E$), called the \emph{nucleus} $\mathcal N$~\cite{selfsimilar, BBGHSW24}.  One has that $\mathcal N=\mathcal N\inv$, the cocycle sends $\mathcal N\times_{E^0} E\to \mathcal N$ and every object $e\in E^0$ which  is not a sink belongs to $\mathcal N$, cf.~\cite{selfsimilar} or~\cite[Lemma~3.4]{BBGHSW24}.

If $(G,E,\sigma)$ is contracting with nucleus $\mathcal N$, then the faithful quotient $(\ov G,E,\sigma)$ is contracting with nucleus contained in the image of $\mathcal N$.  

We prove the analogue of~\cite[Theorem~4.3.21]{Nekbook2}  in our context, namely that if $(G,E,\sigma)$ is contracting and every nontrivial element of the nucleus acts nontrivially on $E^+$, then the action is loosely faithful.

\begin{Cor}\label{c:contracting.case}
Let $(G,E,\sigma)$ be a contracting self-similar groupoid action on a finite graph $E$ without sources with nucleus $\mathcal N$, and let $(\ov G,E, \sigma)$ be the faithful quotient. Suppose that within the nucleus $\mathcal N$ only the units act trivially on $E^+$. Then $(G,E,\sigma)$ is loosely faithful and $\mathscr G_{(G,E)}\cong \mathscr G_{(\ov G,E)}$.
\end{Cor}
\begin{proof}
It suffices by Theorem~\ref{t:nek.cond} to show that $(G,E,\sigma)$ is loosely faithful.   This is the case because for $g$ in the kernel of the action there is an integer $n \geq 0$ such that $g|_w\in \mathcal N$ belongs to the nucleus for all $w\in \sour(g)E^n$.  Clearly, $g$ fixes $\sour(g)E^n$ and each of these sections $g |_w$ acts trivially on $E^+$, and so by assumption belongs to $G^0$.  Thus $g\in K_n$.
\end{proof}

We shall exploit Corollary~\ref{c:contracting.case} because many complicated self-similar groups, like the Grigorchuk group, are faithful quotients of actions of much nicer groups, like free products of finite groups.

Let us consider the isotropy of $\mathscr G_{(G,E)}$ for a self-similar action $(G,E,\sigma)$.

\begin{Lemma}\label{l:torsion.free}
Let $(G,E,\sigma)$ be a self-similar groupoid action and suppose that for each $z \in \partial E$ the stabilizer $G_z$ is torsion-free. Then every isotropy group of $\mathscr G_{(G,E)}$ is torsion-free.
\end{Lemma}
\begin{proof}
Suppose that $\gamma \in \mathscr G_{(G,E)}$ is an isotropy element of finite order at $z \in \partial E$. Then by \cite[Lemma 3.4]{miller2024isomorphisms}  there is a maximal subgroup of $S_{(G,E)}$ with an element $s$ of finite order such that $\gamma = [s,z]$. The unit of the maximal subgroup is $pp^*$ for some $p \in E^+$ and $s = pgp^*$ for some $g \in G^{\sour(p)}_{\sour(p)}$. It follows that $g$ has finite order and stabilizes $p^* z \in \partial E$ and is therefore trivial by assumption.
\end{proof}

\subsection{Amenability}

Let us address the amenability of $\mathscr G_{\mathcal X}$. It is argued in \cite[Corollary 10.18]{ExelPardoSelf} that, for amenable $G$, since $\mathcal O_{\mathcal X} = C^*(\mathscr G_{\mathcal X})$ is nuclear, the groupoid $\mathscr G_{\mathcal X}$ must be amenable. The cited result \cite[Theorem 5.6.18]{BrownOzawa} is only present in the literature for Hausdorff groupoids, although experts seem to be aware that it should hold in the locally Hausdorff setting.\footnote{Since this article was first put online, this result has indeed been extended to the locally Hausdorff setting~\cite{BM25, BGHL25}.} We have decided to present here a direct proof of amenability.  In fact, we derive a more general result that also encompasses Nekrashevych's result on amenability of the groupoid of a self-replicating contracting group~\cite[Theorem~5.6]{Nekcstar}.\footnote{This theorem is only stated in~\cite{Nekcstar} for Hausdorff groupoids, but Nekrashevych informed the second author (private communication) that it holds in general, and indeed the result Nekrashevych uses was generalized to non-Hausdorff groupoids in~\cite{RenaultScand}.}

Let $(G,E,\sigma)$ be a self-similar groupoid action. 
Let us 
assume that $E$ and $G$ are countable.  In what follows we identify the space of tight filters on $S_{(G,E)}$  with $\partial E$ in the usual way.
Let $\mathscr H_0=\{[g,z]\in \mathscr G_{(G,E)}\mid g\in G\}$.  Then $\mathscr H_0$ is an open subgroupoid and it is the quotient of $G\ltimes \partial E$ by the equivalence relation identifying $(g,z)$ with $(h,z')$ if and only if $z=z'$ and there is a prefix $z_0$ of $z$ with $g(z_0)=h(z_0)$ and $g|_{z_0}=h|_{z_0}$.   The quotient map $G\ltimes \partial E\to \mathscr H_0$ is an \'etale homomorphism that restricts to a homeomorphism of unit spaces. It is an isomorphism if and only if $(G,E,\sigma)$ is pseudo\-free.   In the case that $G$ acts faithfully on $E^+$, or more generally $(G,E,\sigma)$ is loosely faithful, the groupoid $\mathscr H_0$ is isomorphic to the groupoid of germs of the action of $G$ on $\partial E$, that is, to  the quotient of $G\ltimes \partial E$  obtained by identifying $(g,z)$ and $(h,z')$ if $z=z'$ and $g,h$ agree on a neighborhood of $z$.

We first state some permanence properties of amenability that we use. These are mostly covered in \cite{AmenableGroupoids,Williams19} for Hausdorff groupoids, but we cannot find them all written explicitly in the literature in the locally Hausdorff \'etale setting. The main idea, from \cite{RenaultScand}, is that amenability depends only on the underlying Borel structure, and therefore we can deduce the results from their Borel counterparts. 

\begin{Prop}\label{p:AmenabilityPermanence}
For second countable locally Hausdorff \'etale groupoids $\mathscr G$ and $\mathscr H$, amenability satisfies the following permanence properties:
\begin{enumerate}
\item If $\mathscr H \hookrightarrow \mathscr G$ is a topological embedding and $\mathscr G$ is amenable, then $\mathscr H$ is amenable.
\item If $\varphi \colon \mathscr G \to \mathscr H$ is a surjective \'etale homomorphism which is a homeomorphism on unit spaces and $\mathscr G$ is amenable, then $\mathscr H$ is amenable.
\item If $\mathscr G$ is amenable and $\mathscr H$ is Morita equivalent to $\mathscr G$, then $\mathscr H$ is amenable.
\item If $\mathscr G = \bigcup_{k \geq 0} \mathscr G_k$ is the increasing union of open amenable subgroupoids $\mathscr G_k$ with $\mathscr G^0 \subseteq \mathscr G_k$, then $\mathscr G$ is amenable.
\item If $K$ is a countable discrete amenable groupoid and $c \colon \mathscr G \to K$ is a continuous groupoid homomorphism with $\mathscr H = c\inv(K^0)$ amenable, then $\mathscr G$ is amenable.
\item If $U\subseteq \mathscr G^0$ is open and invariant and $\mathscr G|_U$, $\mathscr G|_{\mathscr G^0\setminus U}$ are amenable, then $\mathscr G$ is amenable.
\end{enumerate}
\end{Prop}
\begin{proof}
Corollaries~2.15 and~2.16 in~\cite{RenaultScand} say that a second countable locally Hausdorff \'etale groupoid $\mathscr G$ is amenable if and only if it is Borel amenable, if and only if $(\mathscr G,\mu)$ is an amenable measured groupoid for any quasi-invariant measure $\mu$ on $\mathscr G^0$ (i.e., $\mathscr G$ is measurewise amenable). We can therefore apply the Borel and measured groupoid versions of the above permanence properties.

We claim that any topological embedding $\mathscr H \hookrightarrow \mathscr G$ induces a proper embedding of the underlying Borel groupoids, so  item (1) follows from~\cite[Corollary 5.3.22]{AmenableGroupoids}. To see this, note that $\mathscr G^{\mathscr H^0}$ is basic as a left $\mathscr H$-space in the sense of~\cite[Definition 2.7]{AKM22} as the map $\mathscr H\times_{\mathscr H^0}\mathscr G^{\mathscr H^0}\to \mathscr G^{\mathscr H^0}\times \mathscr G^{\mathscr H^0}$ given by $(h,g)\mapsto (hg,g)$ is a homeomorphism onto its image (with inverse map $(g',g)\mapsto (g'g\inv,g)$).
Then~\cite[Lemma 2.10]{AKM22} says that the orbit space projection $p \colon \mathscr G^{\mathscr H^0} \to \mathscr H \backslash \mathscr G^{\mathscr H^0}$ is a local homeomorphism. Let $(U_i)_{i \in \mathbb N}$ be a countable open cover of $\mathscr G^{\mathscr H^0}$ on which $p$ is injective (using second-countability). Then $s \colon \mathscr H \backslash \mathscr G^{\mathscr H^0} \to \mathscr G^{\mathscr H^0}$ defined on $U_i \setminus \bigcup_{j < i} U_j$ by $p|_{U_i}\inv$ is a Borel section to $p$. Thus~\cite[Lemma 9.65]{Williams19} implies that $\mathscr G^{\mathscr H^0}$ is proper as a Borel $\mathscr H$-space.

For a surjective \'etale homomorphism $\varphi \colon \mathscr G \to \mathscr H$ which is a homeomorphism on unit spaces, the underlying Borel homomorphism is strongly surjective. Moreover, the left orbit space $\mathscr G \backslash (\mathscr G^0 \times_{\mathscr H^0} \mathscr H)$ of the graph of $\varphi$ is homeomorphic to 
$\mathscr H^0$ via the map induced by the source map. 
Item (2) therefore follows from \cite[Corollary 5.3.32]{AmenableGroupoids}. 

A Morita equivalence of $\mathscr G$ and $\mathscr H$ becomes a Borel equivalence of the underlying Borel groupoids, so \cite[Theorem 3.2.16]{AmenableGroupoids} proves item (3). Item (4) is implied by \cite[Corollary 5.3.37]{AmenableGroupoids} since we have already shown that any topological embedding of \'etale groupoids is a proper embedding of underlying Borel groupoids.

For item (5), fix a $K$-transversal  $T\subseteq K^0$ and fix $k_x\colon t(x)\to x$ in $K$ for each $x\in K^0$, with $t(x)\in T$.  Define $\ov c\colon \mathscr G\to \bigoplus_{x\in T}K^x_x$ by $\ov c(g) = k_{\ran(c(g))}\inv c(g)k_{\sour(c(g))}\in K^{t(\sour(c(g)))}_{t(\sour(c(g)))}$.  This is a continuous cocycle. The group $K^x_x$ is a closed subgroupoid of $K$ and hence amenable by (1). 
As the class of amenable groups is closed under finite direct products and direct limits, it follows that $\bigoplus_{x\in T}K^x_x$ is an amenable group.    Setting $\mathscr H'=\ov c\inv(1)$, it suffices to show that $\mathscr H'$ is amenable by \cite[Corollary 4.5]{RenaultWilliamsAmenability}.  Now $\mathscr H$, $\mathscr H'$ are clopen subgroupoids of $\mathscr G$, hence \'etale. Then $X=\bigcup_{x\in K^0}c\inv(k_x)$ is a clopen subset of $\mathscr G^0$, and one checks that $X$ is an $\mathscr H'$-$\mathscr H$-bispace providing a Morita equivalence between these groupoids.  
As we are assuming that $\mathscr H$ is amenable, we deduce that $\mathscr H'$ is amenable by (3), as was required.

As observed in the proof of~\cite[Proposition~9.83]{Williams19},  if $\{f_n^U\}$ and $\{f_n^{\mathscr G^0\setminus U}\}$ are topological invariant densities for $\mathscr G|_U$ and $\mathscr G|_{\mathscr G^0\setminus U}$, respectively, then 
\[f_n(g) = \begin{cases} f_n^U(g), & \text{if}\ g\in \mathscr G|_U,\\ f_n^{\mathscr G^0\setminus U}(g), & \text{if}\ g\in \mathscr G|_{\mathscr G^0\setminus U},\end{cases}\] is a Borel approximate invariant density for $\mathscr G$, yielding item (6).
\end{proof}

\begin{Thm}\label{t:amenability}
Let $(G,E,\sigma)$ be a self-similar groupoid action with $G$ and $E$ a countable.  Then $\mathscr G_{(G,E)}$ is amenable if and only if the open subgroupoid \[\mathscr H_0=\{[g,z]\in \mathscr G_{(G,E)} \mid g\in G, z \in \partial E \}\] is amenable.  In particular, if $G\ltimes \partial E$ is amenable (e.g., if $G$ is amenable), then $\mathscr G_{(G,E)}$ is amenable.
\end{Thm}
\begin{proof}
Suppose first that $\mathscr G_{(G,E)}$ is amenable.  Then by permanence of amen\-a\-bil\-i\-ty under open subgroupoids, $\mathscr H_0$  is amenable.

For the converse, assume that $\mathscr H_0$ is amenable.
Let $c \colon \mathscr G_{(G,E)} \to \Z$ be the cocycle defined by $c([pgq^*,x]) = |p| - |q|$ and set $\mathscr H = c \inv(0)$. Since $\Z$ is an amenable group, by Proposition~\ref{p:AmenabilityPermanence} (5), it suffices to show that $\mathscr H$ is amenable.  

For $k\geq 0$, let $\mathscr H_k=\{[pgq^*,qz]\mid |p|=|q|=k\}$.  
These are open subgroupoids of $\mathscr H$, and one has that the $\mathscr H'_k = \bigcup_{i=0}^k \mathscr H_k$, with $k\geq 0$, are open subgroupoids with $\mathscr H^0\subseteq \mathscr H'_k\subseteq \mathscr H'_{k+1}$ and $\mathscr H=\bigcup_{l=0}^\infty \mathscr H_l'$.  In the case that $E$ is row finite, without sources, $\partial E$ consists of all infinite paths in $E$, and so $\mathscr H_k\subseteq \mathscr H_{k+1}$, whence $\mathscr H'_k=\mathscr H_k$.

It suffices to prove that each $\mathscr H_k'$ is amenable by Proposition~\ref{p:AmenabilityPermanence}~(4).
By hypothesis $\mathscr H_0$  is amenable.  We claim that $\mathscr H_k$  is Morita equivalent to an open subgroupoid of $\mathscr H_0$ for all $k\geq 0$. It will then follow that $\mathscr H_k$ is amenable by  Proposition~\ref{p:AmenabilityPermanence}~(1) and~(3). 
Consider the subset $C_k=\{z\in \partial E\mid  \exists g\in G,  E^kg(z)\neq \emptyset\}$ of $\partial E$.  
Note that $C_k$ is clopen and invariant in $\mathscr H_0$.  Indeed, if $z\notin C_k$, then $\sour(z)\partial E$ is a neighborhood of $z$ disjoint from $C_k$, whereas if $z\in C_k$, then $\sour(z)\partial E$ is a neighborhood of $z$ contained in $C_k$.
 An equivalence of $\mathscr H_k$ and $\mathscr H_0|_{C_k}$ is given by the $\mathscr H_k$-$\mathscr H_0|_{C_k}$-bispace $\{ [pg,z] \in \mathscr G_{(G,E)} \mid |p| = k \}$, an open subspace of $\mathscr G_{(G,E)}$.
It now follows that $\mathscr H$ is amenable if $E$ is row finite without sources, as $\mathscr H_k'=\mathscr H_k$. 

For the general case, let $U_k\subseteq \partial E$ be the $\mathscr H$-invariant open set $E^k\cdot \partial E$.  Then $\mathscr H'_k|_{U_k} = \mathscr H_k$ and hence is amenable.  Thus it suffices, by Proposition~\ref{p:AmenabilityPermanence}~(6) to show that $\mathscr H'_k|_{\partial E\setminus U_k}$ is amenable.  Note that  $\mathscr H'_k|_{\partial E\setminus U_k}= \{[pgq^*,q]\mid |p|=|q|<k, \sour(q)\in G_{\sing}^0\}$.  First of all, we claim that $G_\sing$ is amenable.  Indeed, $G_\sing$ is discrete and hence amenable if and only if all its isotropy subgroups are amenable.  Now if $x\in G^0_\sing$, 
then $G^x_x$ is isomorphic to $(\mathscr H_0)^x_x$, which is a closed subgroupoid and hence amenable by Proposition~\ref{p:AmenabilityPermanence}~(1).  There is a continuous functor $c\colon \mathscr H'_k|_{\partial E\setminus U_k}\to G_\sing$ given by $c([pgq^*,q])= g$ and $c\inv(G^0_\sing) = \{[pq^*,q]\mid |p|=|q|<k, \sour(q)\in G^0_\sing\}$.  But this latter groupoid is isomorphic to a closed subgroupoid of the boundary path groupoid $\mathscr G_E$, and hence is amenable by Proposition~\ref{p:AmenabilityPermanence}~(1) since $\mathscr G_E$ is amenable (cf.~\cite[Proposition~3.1]{RenaultScand}, or use that $\mathscr G_E$ is Hausdorff, and $C^*(E)$ is nuclear).   

 The `in particular' statement follows because if $G\ltimes \partial E$ is amenable, then so is $\mathscr H_0$ by Proposition \ref{p:AmenabilityPermanence}~(2).  
 If $G$ is amenable, then $G\ltimes \partial E$ is amenable by~\cite[Corollary~2.2.10]{AmenableGroupoids}.
\end{proof}

As a corollary we obtain that the groupoid associated to a contracting self-similar groupoid is always amenable, generalizing  slightly~\cite[Theorem~5.6]{Nekcstar}, who considered only self-replicating contracting groups. We offer two proofs, the first following Nekrashevych~\cite{Nekcstar} and the second novel.

\begin{Cor}\label{c:contracting.amenable}
Let $(G,E,\sigma)$ be a contracting self-similar groupoid action where $E$ is a finite graph without sources.  Then $\mathscr G_{(G,E)}$ is amenable.  Moreover, if $H$ is the subgroupoid of $G$ generated by  $G^0$ and the nucleus $\mathcal N$, then $(H,E,\sigma|_{H\times E})$ is a contracting self-similar action with $H$  finitely generated and $\mathscr G_{(G,E)}\cong \mathscr G_{(H,E)}$.   
\end{Cor}
\begin{proof}
We begin with the final statement. 
Using that $\mathcal N\cup G^0$ is closed under sections, it is easy to see that $H$ is closed under sections.   By construction it is contracting with nucleus $\mathcal N$.  Note that $S_E\subseteq S_{(H,E)}\subseteq S_{(G,E)}$ and $S_E$ contains all the idempotents.  Hence to show that $\mathscr G_{(G,E)}\cong \mathscr G_{(H,E)}$ it suffices to show that each element of $\mathscr G_{(G,E)}$ can be written in the form $[uhv^*,z]$ with $h\in H$ and $z\in \partial E$.  Let $[pgq^*,qz]\in \mathscr G_{(G,E)}$.   Then we can find $n$ with $g|_w\in \mathcal N$ for all $w\in \sour(g)E^n$.  Let $v$ be the prefix of $z$ of length $n$.  Then $[pgq^*,qz] = [pg(v)g|_v(qv)^*,qz]$ and $g|_v\in H$.  
Henceforth, replacing $G$ by $H$ if necessary, we assume that $G$ is finitely generated and thus countable.
   
For the first proof that $\mathscr G_{(G,E)}$ is amenable, we show that $G\ltimes \partial E$ is amenable, which suffices by Theorem~\ref{t:amenability}. 
   It was shown by Nekrashevych that, for finitely generated contracting self-similar groups,  the Schreier graph for the action on the set of infinite words has polynomial growth~\cite[Proposition~2.13.6]{selfsimilar} (see also~\cite[Proposition~4.3.14]{Nekbook2}). The proof works \textit{mutatis mutandis} for finitely generated contracting self-similar groupoids.  It follows that $G\ltimes \partial E$ is amenable by~\cite[Corollary~3.17]{RenaultScand}.  

 For the second proof, we prove that $\mathscr H_0 = \{[g,z]\in \mathscr G_{(G,E)} \mid g\in G\}$ is amenable and apply Theorem~\ref{t:amenability}.  It was observed by Nekrashevych~\cite{selfsimilar} that the isotropy groups of $\mathscr H_0$ are finite of cardinality bounded by $|\mathcal N|$.  Indeed, suppose that $A=\{[g_a,w]\in (\mathscr H_0)_w^w \}$ is a finite set with more than $|\mathcal N|$ elements.  Then we can find a finite prefix $v$ of $w$ such that $g_a|_v\in \mathcal N$ for all $a$. Since $|A|>|\mathcal N|$, there exist $[g_{a_1},w]\neq [g_{a_2},w]\in A$ with $g_{a_1}|_v=g_{a_2}|_v$, a contradiction as $g_{a_1}(v) = v=g_{a_2}(v)$.  Therefore, $\mathscr H_0$ has amenable isotropy groups.   Next we claim that the equivalence relation on $\partial E$ associated to  $\partial E/\mathscr H_0$ is hyperfinite, and hence amenable. It will then follow that $\mathscr H_0$ is amenable by ~\cite[Corollary~5.3.33 and Theorem~5.3.42]{AmenableGroupoids} and the fact that topological amenability coincides with measurewise amenability~\cite[Corollary 2.16]{RenaultScand}. 

 Since $\mathscr H^0\subseteq \mathscr H_0\subseteq \mathscr H$, it suffices to show that the equivalence relation $R$ on $\partial E$ of being in the same $\mathscr H$-orbit is hyperfinite, since hyperfiniteness passes to Borel subequivalence relations, cf.~\cite[Proposition~1.3]{JKL02}. Let $T\colon \partial E\to \partial E$ be the shift map. We claim that $x\mathrel{R} y$ if and only if there is $k\geq 0$ and $n\in \mathcal N$ with $n(T^k(x))=T^k(y)$.  Indeed, given $k$ and $n$, if $x_0,y_0$ are the prefixes of length $k$ of $x,y$, respectively, then $[y_0nx_0^*,x]\colon x\to y$.  Conversely, if $[pgq^*,qw]\colon x\to y$ with $|p|=|q|$, then choosing a prefix $w_0$ of $w$ sufficiently long that  $g|_{w_0}\in \mathcal N$, we get $y=pg(w_0)g|_{w_0}(T^{|q|+|w_0|}(x)))$.  Thus $g|_{w_0}(T^{|q|+|w_0|}(x))=T^{|q|+|w_0|}(y)$.  
 Note that if $n(T^k(x)) = T^k(y)$ with $n\in \mathcal N$, then $n|_e(T^{k+1}(x)) = T^{k+1}(y)$ where $e$ is the first edge of $T^k(x)$ and $n|_e\in \mathcal N$. 

 Let $R'$ be the equivalence relation on $\partial E$ given by $x\mathrel{R'} y$ if $T^k(x)=T^k(y)$ for some $k\geq 0$.  Then $R'$ is hyperfinite and $R'\subseteq R$.  We claim that each $R$-class contains no more than $|\mathcal N|$ $R'$-classes.    It will then follow that $R$ is hyperfinite, cf.~\cite[Proposition~1.3]{JKL02}.  Indeed, suppose that $\{[x_i]_{R'}\mid i\in J\}$ with $|\mathcal N|<|J|<\infty$ are distinct $R'$-classes contained in an $R$-class. Fix $i_0\in J$.   Then we can find a single $k\geq 0$ and $n_j\in \mathcal N$, with $n_j(T^k(x_{i_0}))=T^k(x_j)$ by the last sentence of the previous paragraph.  By assumption on $J$, $n_i=n_{i'}$ for some $i\neq i'$.  It follows that $[x_i]_{R'}=[x_{i'}]_{R'}$, a contradiction.  This completes the proof. 
\end{proof}

\section{The Dramatis Person\ae}

This section lists our favorite examples. 

\subsection{Miscellaneous examples}
\begin{Example}[Graphs]\label{Ex:Graphs} If $\ran,\sour\colon E\to E^0$ is a directed graph, then we can view it as a correspondence over the discrete groupoid $G$ with $G=G^0=E^0$.  The resulting inverse semigroup $S_E$ is the usual graph inverse semigroup of $E$~\cite{Paterson}, $\mathscr G_{(G,E)}$ is the usual boundary path groupoid and $\mathcal O_E$ is the graph $\cs$-algebra $\cs(E)$.
\end{Example}

Our next example is the self-similar actions Exel and Pardo introduced to model Katsura algebras~\cite[Section~18]{ExelPardoSelf}. 
Due to our convention on graphs, the adjacency matrix $A$ of a graph is defined by $A_{ij}=|\ran\inv(i)\cap \sour\inv(j)|$.  

\begin{Example}[Exel--Pardo--Katsura]\label{ex:EPK}
Let $A,B$ be integer matrices over some index set $J$ (perhaps infinite) such that $A$ is the adjacency matrix of a row finite graph.  We assume that $A_{ij}=0$ implies $B_{ij}=0$ (but not conversely).   Our groupoid $G_{A,B}$ is $\mathbb Z\times J$ where $J$ is viewed as a groupoid with $J^0=J$.  Let  $E = \{e_{i,j,\ov n}\mid A_{ij}\neq 0,  \ov n\in \mathbb Z/A_{ij}\mathbb Z\}$. We make $E$ a graph via $\sour(e_{i,j}) = j$ and $\ran(e_{i,j}) = i$. 

Define an action of $\mathbb Z\times J$ on $E$ by $(m,i)e_{i,j, \ov n}=e_{i,j,\ov{mB_{i,j}+n}}$.
The $1$-cocycle is given by the rule $(m,i)|_{e_{i,j,\ov n}}=(k,j)$  where $mB_{ij}+n = kA_{ij}+r$ with $0\leq n<A_{ij}$ and $0\leq r<A_{ij}$.  Note that $(m,i)(e_{i,j,\ov n}) = e_{i,j,\ov r}$ with the above notation. 

It is known that $\mathcal O_{(G_{A,B},E)}$  is Katsura's algebra $\mathcal O_{A,B}$. We compute the homology and K-theory for this example in Subsection \ref{subs:EPK}. Since $\mathbb Z\times J$ is amenable, $\mathscr G_{(G_{A,B},E)}$ is amenable.
\end{Example}

\begin{Example}
The following is~\cite[Example~2.2]{BBGHSW24}.  Let $G$ be the combinatorial fundamental groupoid of the graph on the left and let $E$ be the graph on the right in:
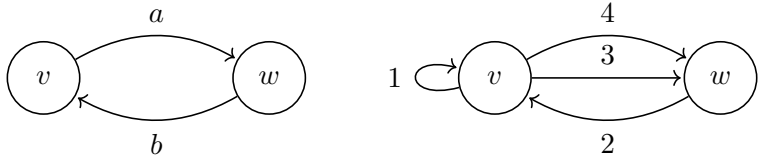
\[
\begin{tikzpicture}[->,shorten >=1pt,%
auto,node distance=3cm,semithick,
inner sep=5pt,bend angle=30]
\begin{scope}[xshift=-4cm]
\node[state] (A) {$v$};
\node[state] (B) [right of=A] {$w$};
\path (A) edge [bend left]  node [above]  {$a$} (B)
      (B) edge [bend left] node [below] {$b$} (A);
\end{scope}
\begin{scope}[scale=.5, xshift=4cm]
\node[state] (A) {$v$};
\node[state] (B) [right of=A] {$w$};
\path   (A) edge [loop left] node [left] {$1$} (A) 
        (A) edge [left]  node [above]  {$3$} (B)
        (A) edge [bend left]  node [above]  {$4$} (B)
        (B) edge [bend left] node [below] {$2$} (A);
\end{scope}
\end{tikzpicture}
\]
The proper correspondence $\mathcal X$ is $E\times_{E^0} G$ with left action $a(1,g) = (4,g)$, $a(2,g) = (3,bg)$, $b(3,g) = (1,g)$, $b(4,g) = (2,ag)$.  It was shown in~\cite{BBGHSW24} that $\mathcal O_{\mathcal X}$ is isomorphic to the Cuntz--Pimsner algebra of the dyadic odometer using the Kirchberg--Phillips Theorem.  We give a more direct proof that they are Morita equivalent using the correspondence viewpoint. 

The isotropy group $G^v_v=\langle ba\rangle$ is infinite cyclic.  Consider the dyadic odometer action of $ba$ on $\{0,1\}$ given by $ba(0)=1$, $ba(1)=0$, $(ba)|_0 = v$, $(ba)|_1 = ba$; the induced action on $\{0,1\}^{\mathbb N}$ is adding $1$ to a dyadic integer. Write $\mathcal Y = \{0,1\}\times \langle ba\rangle$ for the associated proper correspondence.   The inclusion $\langle ba\rangle\hookrightarrow G$ is a Morita equivalence with bispace $\Omega=vG$.  It suffices to show that $\Omega\times_G \mathcal X\cong \mathcal Y\times_H \Omega$ by Proposition~\ref{p:morita.equiv.corr}. 
Note that $\Omega\times_G\mathcal X\cong vE\times_{E_0} G$ and $\mathcal Y\times_H \Omega\cong \{0,1\}\times vG$ with the obvious $G_v^v$-$G$-biactions.  The isomorphism $\psi\colon  vE\times_{E_0} G\to \{0,1\}\times vG$ is given by $\psi(1,g) = (1,g)$, $\psi(2,g) = (0,bg)$.  We omit the routine verification.
\end{Example}


The rest of our examples are self-similar group actions on finite alphabets \cite{selfsimilar}, as they are the primary motivation for this theory.  
Hence we deal only with proper correspondences over a group.  So, a self-similar action is given by a group $G$ acting on a finite set $X$ together with a $1$-cocycle $G\times X\to G$ written $g\mapsto g|_x$.  We do \emph{not} require the action of $G$ on $X^+$ to be faithful.

The inverse semigroup $S_{(G,X)}$, the groupoid $\mathscr G_{(G,X)}$ and the $\cs$-algebra $\mathcal O_{(G,X)}$ are precisely the inverse semigroup, groupoid and $\cs$-algebra considered by Nekrashevych in~\cite{Nekcstar}.  We now proceed to describe several families of self-similar groups.

\begin{Example}[The Aleshin automaton]\label{ex:aleshin}
Aleshin~\cite{Aleshin} considered the following self-similar action of the free group ${F_3}$ on $a,b,c$ on the two-letter alphabet $\{0,1\}$.  Namely, $a(0)=1=b(0)$, $a(1)=0=b(1)$ and $c(0)=0$, $c(1)=1$.  The sections are given by $a|_0 =c$, $a|_1=b$, $b|_0= b$, $b|_1=c$, $c|_0=a$ and $c|_1=a$.   It was shown by Vorobets and Vorobets~\cite{Vorobets} that this self-similar action of ${F_3}$ on $\{0,1\}^+$ is faithful.  The groupoid associated to the Aleshin automaton has torsion-free isotropy by Lemma~\ref{l:torsion.free}, and the action is pseudo\-free~\cite{SVV}, whence the groupoid is Hausdorff.

We compute the homology and K-theory for this example in Subsection \ref{subs:Aleshin}.
\end{Example}

Cocycles for free products can be defined on the factors.

\begin{Lemma}\label{l:cocycle.extension}
Let $A,B$ be groups acting on a set $X$ and let $H$ be a group.  Suppose that $\sigma_1\colon A\times X\to H$ and $\sigma_2\colon B\times X\to H$ are $1$-cocycles.  Then there is a unique $1$-cocycle $\sigma\colon (A\ast B)\times X\to H$ extending $\sigma_1,\sigma_2$.   
\end{Lemma}
\begin{proof}
Let $G$ be a group acting on $X$ via a homomorphism $\p\colon G\to S_X$ and $c\colon G\times X\to H$ be a mapping with $H$ a group. The permutational wreath product $S_X\wr H$ is the semidirect product $S_X\ltimes H^X$ where the right action of $S_X$ is on $H^X$ is given by precomposition.
Define a map $\Phi\colon G\to S_X\wr H$ by $\Phi(g) = (\p(g),c_g)$ where $c_g(x) = c(g,x)$.  It is well known and easy to see that $c$ is a $1$-cocycle if and only if $\Phi$ is a homomorphism.  
The lemma now follows immediately from the universal property of a free product.
\end{proof}

\begin{Example}[Hanoi towers group]\label{ex:hanoi}
The \emph{Hanoi towers} group $H$ is a self-similar group which models the classical Towers of Hanoi puzzle  and was first studied by Grigorchuk and \v{S}uni\'{c}~\cite{Hanoitowers}.  It is also the iterated monodromy group of the rational function $z^2-\frac{16}{27z}$, whose Julia set is a Sierpi\'nski gasket. 

Let $A,B,C$ be cyclic groups of order $2$ generated by $a,b,c$, respectively.  Then one can define a contracting self-similar action of $A\ast B\ast C$ on $\{0,1,2\}$ where $a$ acts by $(01)$, $b$ acts by $(02)$ and $c$ acts by $(12)$.  One has $a|_2=a$, $b|_1=b$ and $c|_0=c$.  All remaining sections are trivial. Lemma~\ref{l:cocycle.extension} implies that this gives a self-similar action of $A\ast B\ast C$. The nucleus is $a,b,c$ and the identity.  The faithful quotient is the Hanoi towers group $H$.  The groupoid associated to $H$ is minimal, effective, Hausdorff and amenable.

We compute the homology and K-theory for this example in Subsection \ref{subs:hanoi}.
\end{Example}

\subsection{Multispinal self-similar groups}\label{ss:multispinal}  We consider here multi\-spinal gr\-oups \cite{SS23}, a family of contracting groups generalizing the famous Grigorchuk group\cite{Grigtorsion,griggrowth}, the Gupta--Sidki groups~\cite{GuptaSidki}, as well as \v{S}uni\'c groups \cite{sunicgroups}.

A multispinal group~\cite{SS23} consists of the following data.  Two finite groups $A,B$, with $|A|\geq 2$, and a map $\Phi\colon A\to \mathrm{Aut}(B)\cup \mathrm{Hom}(B,A)$ such that with $A_0 = \Phi\inv(\mathrm{Aut}(B))$ and $A_1=A\setminus A_0$,
\begin{enumerate}
    \item $A_0\neq \emptyset$.
    \item $\Phi(A_1)\cdot \langle \Phi(A_0)\rangle\subseteq\mathrm{Hom}(B,A)$ separates points of $B$.
\end{enumerate}
We can then define a self-similar action of $A\ast B$ on the alphabet $A$ as follows.  The group $A$ acts by left multiplication on $A$.  The group $B$ acts trivially on $A$.  If $a\in A$, then $a|_x=1$.  If $b\in B$, then $b|_a=\Phi_a(b)$ where we write $\Phi_a$ instead of $\Phi(a)$. Trivially, this gives $1$-cocycles $A\times A\to A\ast B$ and $B\times A\to A\ast B$, which 
extend uniquely to a $1$-cocycle $(A\ast B)\times A\to A\ast B$ by Lemma~\ref{l:cocycle.extension}.   The action of $A\ast B$ is not usually faithful.  The corresponding faithful self-similar group $(G_{(A,B)},A,\sigma)$ is then the multispinal group\footnote{In~\cite{SS23} a more general class of groups is called multispinal; we are restricting here to those multispinal groups that act transitively on the alphabet.  It should also be mentioned that \cite{SS23} only considers the faithful quotient of the action of $A\ast B$.} associated to this data.        The action of $A\ast B$ is contracting with nucleus contained in $A\cup B$.  Moreover, (2) implies that the nontrivial elements of $A,B$ act nontrivially~\cite{SS23}.  Hence the action is loosely faithful and the groupoids for the self-similar action of the free product and the multispinal group coincide by Corollary~\ref{c:contracting.case}.  The groupoid of a multispinal group is Hausdorff if and only if each element of $\Phi(A_1)$ is injective.  
Note that it is known precisely when the algebra over a field and the $\cs$-algebra of the groupoid of a multispinal group is simple~\cite{SS23,Yoshida21,GNSV}.

Since multispinal groups are contracting their groupoids are amenable by Corollary~\ref{c:contracting.amenable}.
General results on the homology and K-theory for multispinal groups appear in Section \ref{s:multispinal}.

Let us now present a number of examples of multispinal groups.

\begin{Example}[\v{S}uni\'c groups]\label{ex:sunic.groups}
 The following family of multispinal groups was introduced by \v{S}uni\'c as generalizations of the Grigorchuk group~\cite{sunicgroups}.    Let $A=\mathbb Z/p\mathbb Z$ and $B=(\mathbb Z/p\mathbb Z)^n$ with $p$ prime, $\Phi_0$ be the projection to the last coordinate, $\Phi_i=0$, for $i=1,\ldots,p-2$ and $\Phi_{p-1}(b) = C_fb$ where $C_f \in M_n(\mathbb F_p)$ is the companion matrix of a degree $n\geq 1$ polynomial $f(x)\in \mathbb F_p[x]$ with $f(0)\neq 0$.  The corresponding multispinal group is denoted $G_{p,f}$.  We write $\mathscr G_{p,f}$ for the associated ample groupoid.

Homology and K-theory computations for \v{S}uni\'c groups groups appear in Subsection \ref{subs:Sunic}.
\end{Example}

Of particular importance is the special case of a primitive polynomial $f$.  This means that $f$ is the minimal polynomial of a primitive element of a field extension $\mathbb F_q/\mathbb F_p$.  In this case, $C_f$ acts transitively on $B\setminus \{0\}$, since $B$ can be identified with $\mathbb F_q$ and the action of $C_f$ corresponds to multiplication by a generator of the cyclic group $\mathbb F_q^\times$.  If $f$ is a primitive polynomial of degree at least $2$, then $G_{p,f}$ is an infinite $p$-group of intermediate growth~\cite{sunicgroups}.

\begin{Example}[Infinite dihedral group]\label{ex:infinite.dihedral.gp}
The group $G_{2,x-1}$ is the infinite dihedral group $D_{\infty}$.   We write $a,b$ for the respective generators of $A,B$.   Writing $e$ for the identity of $D_{\infty}$, we have $a(0)=1$, $a(1)=0$, $a|_0=e=a|_1$, $b(0)=0$, $b(1)=1$, $b|_0 =a$, $b|_1=b$.  The K-theory of $\cs(\mathscr G_{2,x-1})$ and a partial computation of the homology of $\mathscr G_{2,x-1}$ was obtained in~\cite{ortegadihedral}.  The group $G_{3,x-1}$ is known as the Fabrykowski--Gupta group~\cite{FG85}, which has intermediate growth, and the groups $G_{p,x-1}$ were studied in general by Grigorchuk~\cite{justinfinite}.
\end{Example}

The most famous self-similar group is the Grigorchuk group.

\begin{Example}[The Grigorchuk group]\label{ex:grigorchuk.gr}
The Grigorchuk group is $G_{2,1+x+x^2}$.  This is a primitive polynomial, and so the group is an infinite torsion group of intermediate growth~\cite{Grigtorsion}.  It was the first example of a group of intermediate growth~\cite{griggrowth}.  It is also just infinite, meaning that all its nontrivial normal subgroups have finite index.  The companion matrix  is  \[C_{1+x+x^2} = \begin{bmatrix}0 & 1\\ 1& 1\end{bmatrix}.\]   If we set $a$ to be the generator of $A=\mathbb Z/2\mathbb Z$, $b= (0,1)$, $c=(1,1)$, $d=(1,0)$ and $e=(0,0)$, then $G_{2,1+x+x^2}=\langle a,b,c,d\rangle$ with action given by $a(0)=1$, $a(1)=0$, $a|_0=e=a|_1$, $b,c,d$ acting trivially on $\{0,1\}$ and $b|_0=a$, $b|_1=c$, $c|_0= a$, $c|_1=d$, $d|_0=e$, $d|_1=b$.  The groupoid $\mathscr G_{2,1+x+x^2}$ is amenable, minimal, effective and non-Hausdorff.  It was proved in~\cite{nonhausdorffsimple} that $\cs(\mathscr G_{2,1+x+x^2})$ is simple and the  algebra $K\mathscr G_{2,1+x+x^2}$ is simple for any field $K$ of characteristic different than $2$, but not over fields of characteristic $2$; see also~\cite{Nekrashevychgpd,SS23}. 

The homology and K-theory for this example is computed in Subsection \ref{subs:Grig}.
\end{Example}

The next group was constructed by Grigorchuk, and its rate of intermediate growth was analyzed very precisely by Erschler~\cite{Erschler}, and hence it is widely known as the Grigorchuk--Erschler group.  

\begin{Example}[Grigorchuk--Erschler group]\label{ex:Grig-Erschler}
 The Grigorchuk--Erschler group is $G_{2,1+x^2}$.    The companion matrix  is \[C_{1+x^2} = \begin{bmatrix}0 & 1\\ 1& 0\end{bmatrix}.\]   If we set $a$ to be the generator of $A=\mathbb Z/2\mathbb Z$, $b= (0,1)$, $c=(1,1)$, $d=(1,0)$ and $e=(0,0)$, then $G_{2,1+x+x^2}=\langle a,b,c,d\rangle$ with action given by $a(0)=1$, $a(1)=0$, $a|_0=e=a|_1$, $b,c,d$ acting trivially on $\{0,1\}$ and $b|_0=a$, $b|_1=d$, $c|_0= a$, $c|_1=c$, $d|_0=e$, $d|_1=b$.  The groupoid $\mathscr G_{2,1+x^2}$ for $G_{2,1+x^2}$ is amenable, minimal, effective and non-Hausdorff.  It was observed by Nekrashevych that $\cs(\mathscr G_{2,1+x^2})$ is not simple and that the algebra $K\mathscr G_{2,1+x^2}$ is not simple for any field $K$; see~\cite{SS23} for a proof.

The homology and K-theory for this example is computed in Subsection \ref{subs:GE}.
\end{Example}

\begin{Example}[GGS groups]\label{ex:Gupta-Sidki groups}
A GGS-group is a multispinal group where $A=\mathbb Z/m\mathbb Z$, $B=C_m=\langle b\rangle$ is cyclic of order $m$, $\Phi_{m-1}=\id_B$ and $\Phi_i\colon B\to A$ for $0\leq i\leq m-2$.  Condition (2) in the definition of a multispinal group is satisfied if and only if $\gcd(\Phi_0(b),\cdots,\Phi_{m-2}(b),m)=1$~\cite{branchgroups}, which we henceforth assume.  For example, \v{S}uni\'c groups of the form $G_{p,x-1}$ are GGS-groups.   
Let $p$ be an odd prime.  The Gupta--Sidki group $G_p$ is the GGS group with $\Phi_0(b)=1$, $\Phi_1(b)=-1$, $\Phi_k(b)=0$ for $2\leq k<p-1$ and $\Phi_{p-1}=\id_B$.   The group $G_p$ is an infinite $p$-group~\cite{GuptaSidki}.

Homology and K-theory computations for GGS groups appear in Subsection \ref{subs:Sunic}.
\end{Example}

\subsection{Solvable self-similar groups}
In this subsection we consider some solvable self-similar groups.  Since solvable groups are amenable, the groupoids associated to self-similar actions of solvable groups are amenable.

\begin{Example}[Lamplighter groups]\label{ex:lamplighter}
Let $F$ be a finite group.  Then the restricted wreath product $F\wr \mathbb Z$ is $\bigoplus_{i\in \mathbb Z} F\delta_i\rtimes \mathbb Z$, where the generator of $\mathbb Z$ acts via the shift $a\delta_i\mapsto a\delta_{i+1}$, is the called the $F$-lamplighter group. 
Grigorchuk and \.{Z}uk~\cite{GrigZuk} famously realized  $\mathbb Z/2\mathbb Z\wr \mathbb Z$ as a self-similar group, which led to a counterexample to the strong Atiyah conjecture on $\ell_2$-Betti numbers~\cite{strongatiyah}.  The second author and Silva~\cite{SilvaSteinberg} generalized this construction to give a faithful self-similar realization of $A\wr \mathbb Z$ for any finite abelian group $A$.  Further self-similar actions of $A\wr \mathbb Z$ were given by the second author and Skipper~\cite{SkipperSteinberg}.  Note that if $F$ is nonabelian, then $F\wr \mathbb Z$ is not residually finite and hence cannot have a faithful self-similar representation.   

The construction in~\cite{SkipperSteinberg} is as follows.  We view $A$ as the additive group of a ring $R$ and identify $A^{\mathbb N}$ with $R[\!{[}x]\!{]}$.  If $f\in R[\!{[}x]\!{]}$, then $\alpha_f,\mu_f\colon  R[\!{[}x]\!{]}\to R[\!{[}x]\!{]}$ are given by addition and multiplication by $f$ respectively.  We fix a rational function 
\[f= r\left(\frac{1-ax}{1-bx}\right)\] with $r\in R^\times$ and $a-b\in R^\times$; note that $f$ is a multiplicative unit in $R[\!{[}x]\!{]}$. The generator $t$ of $\mathbb Z$ acts on $R[\!{[}x]\!{]}$ via $\mu_f$, and $s\delta_i$, with $s\in A$, acts by $\alpha_{s(-ar+bf)f^i}$. The action is faithful. Note that $t$ acts on the alphabet $A$ as multiplication by $r$ and $s \delta_i$ acts as addition by $s(b-a)r^{i+1}$.
It is shown in the proof of~\cite[Proposition~3.3 and Theorem~3.6]{SkipperSteinberg}  that $t|_d = \alpha_{-dra}\mu_f\alpha_{db}=(d\delta_0)t$.  The paper~\cite{SilvaSteinberg} uses a direct product of rings of the form $\mathbb Z/n\mathbb Z$ and $f=1-x$, with Grigorchuk and \.{Z}uk using the ring $\mathbb Z/2\mathbb Z$~\cite{GNS}. 

Notice that the nontrivial elements of $\bigoplus_{i\in \mathbb Z}A\delta_i$ act as translations on $R[\!{[}x]\!{]}$ and hence have no fixed points. It follows that  the isotropy groups of the associated groupoid $\mathscr G$ are torsion-free by Lemma~\ref{l:torsion.free}.  Also, the action of $A\wr \mathbb Z$ on $A$ is pseudo\-free. Indeed, suppose that $g\in A\wr \mathbb Z$ with $g(a)=a$ and $g|_a=1$.  By construction, $g$ acts on $R[\!{[}x]\!{]}$ as $f\mapsto hf+d$ for some $h,d\in R[\!{[}x]\!{]}$.  From $g(a)=a$, we have that $ha+d=a$. By assumption $a+x=h(a+x)+d = ha+d+hx=a+hx$, that is, $0=(h-1)x$.  It follows that $h=1$ and $d=0$.     The associated groupoid $\mathscr G$ is therefore Hausdorff, effective and amenable. 

Notice that if $g\in R[\!{[}x]\!{]}$, then $\alpha_g(d+xh(x)) = d+xh(x)+g(x) = d+g(0)+x(h(x)+\frac{g(x)-g(0)}{x})$, from which it follows that $\alpha_g|_d$ does not depend on $d$ and is a translation by $\frac{g(x)-g(0)}{x}$.  In particular, since $\bigoplus_{i \in \mathbb Z}A \delta_i$ acts on $R[\!{[}x]\!{]}$ by translations, and only these elements of $A\wr \mathbb Z$ act by translations, it follows that the $1$-cocycle $\sigma\colon A\wr\mathbb Z\times A\to A\wr \mathbb Z$ restricts to a $1$-cocycle $\til \sigma\colon (\bigoplus_{i \in \mathbb Z}A \delta_i)\times A\to \bigoplus_{i \in \mathbb Z}A\delta_i$ with the property that $g|_d=\til \sigma(g,d)= \til \sigma(g,0)=g|_0$ for all $d\in A$.  This will play a crucial role when we later compute the homology for the groupoid $\mathscr G$. 

Homology and K-theory computations for 
lamplighter groups appear in Subsection \ref{subs:lamp}.
\end{Example}

\begin{Example}[Solvable Baumslag--Solitar groups]\label{ex:BS.group}
    The group with the presentation $\langle a,b\mid bab\inv =a^m\rangle$ is the Baumslag--Solitar group $BS(1,m)$.  It can be identified with the semidirect product $\mathbb Z[1/m]\rtimes \mathbb Z$ where the generator of $\mathbb Z$ acts on $\mathbb Z[1/m]$ via multiplication by $m$.
We will view $BS(1,m)$ as the group of affine transformations of $\mathbb Z[1/m]$ of the form $x\mapsto m^kx+c$ with $k\in \mathbb Z$ and $c\in \mathbb Z[1/m]$. From this viewpoint, $a$ corresponds to the transformation of addition by $1$ and $b$ corresponds to the transformation of multiplication by $m$.

Let $n\geq 2$ be relatively prime to $m$.   Then Bartholdi and \v{S}uni\'c \cite{selfsimilarBS} defined a faithful self-similar action of $BS(1,m)$ on the alphabet $\mathbb Z/n\mathbb Z$.   Under their construction, $a(\ov i) =\ov{i+1}$, $a|_{\ov i} =1$ if $0\leq i<n-1$ and $a|_{\ov{n-1}} = a$.  One has $b(\ov i) = \ov {mi}$ and $b|_{\ov i} =  a^jb$ where $j=\lfloor mi/n\rfloor$ for $0\leq i\leq n-1$.   We shall write $\mathscr G_{(m,n)}$ for the groupoid associated to this self-similar group.  It is easy to check that $\mathscr G_{(m,n)}$ is Hausdorff, minimal, effective and each isotropy group is torsion-free. Indeed, $(\mathbb Z/n\mathbb Z)^{\mathbb N}$ can be identified with the ring of $n$-adic integers $\mathbb Z_n$ as a topological space, and the action of $a$ is by adding $1$ and the action of $b$ is multiplication by $m$~\cite{selfsimilarBS}. The action of $BS(1,m)$  is therefore pseudo\-free, and hence $\mathscr G_{(m,n)}$ is Hausdorff. Indeed, if $g(\ov i)=\ov i$ and $g|_{\ov i}=1$ with $g=a^cb^k$, then $m^k\ov i+c = \ov i$ and $\ov i+n=m^k(\ov i+n)+c = \ov i +m^kn$, that is $(m^k-1)n=0$ in $\mathbb Z_n$.  It follows that $k=0$, $c=0$, i.e., $g=1$. The isotropy is torsion-free by Lemma~\ref{l:torsion.free} as $BS(1,m)$ is torsion-free.

Homology and K-theory computations for Baumslag--Solitar groups appear in Subsection \ref{subs:BS}.
\end{Example}

\begin{Example}[Free abelian groups]\label{Ex:FreeAb}
Let $G=\mathbb Z^n$ be a free abelian group of rank $n$.  The self-similar actions of $G$ on an alphabet $X$ of size $d$, which are transitive on $X$, are described in~\cite{selfsimilar}. Fix $x\in X$.  Then $[G:G_x]=d$.  By the Smith normal form theorem, we can find a basis $e_1,\ldots, e_n$ for $G$ and $d_1,\dots,d_n \in \mathbb N$ such that $d_1e_1,\ldots, d_ne_n$ is a basis for $G_x$.   Note that $d=d_1\cdots d_n$.  The virtual endomorphism $\sigma_x\colon G_x\to G$ can be described by an $n\times n$-matrix $B$ with respect to these bases.   Tensoring with $\mathbb Q$ we obtain an endomorphism $\sigma_x\otimes 1_{\mathbb Q}$ of $\mathbb Q^n$, which is given by a matrix $A\in M_n(\mathbb Q)$ in the $e_i$-basis.  Notice that $B$ is obtained by multiplying column $i$ of $A$ by $d_i$ for $i=1,\ldots, n$.

The isotropy groups are torsion-free by Lemma~\ref{l:torsion.free}.
Nekrashevych prov\-ed that the action of $G$ on $X^+$ is faithful if and only if no eigenvalue of $A$ is an algebraic integer and the action is contracting if and only if the spectral radius of $A$ is less than one; see~\cite{selfsimilar}.  If the action is faithful, it is pseudo\-free and the groupoid is Hausdorff.

A self-similar group action $(G,X,\sigma)$ is called \emph{self-replicating}\footnote{The obsolete terminology ``recurrent" is used in~\cite{selfsimilar}.} if $G$ acts transitively on $X$ and the virtual endomorphism $\sigma_x\colon G_x\to G$ is onto for some (equivalently, all) $x\in X$.  It follows that a transitive self-similar action of $\mathbb Z^n$ is self-replicating if and only if the virtual endomorphism is an isomorphism $G_x\to \mathbb Z^n$.  In this case $A\inv\in M_n(\mathbb Z)$, and is the matrix of the inverse of $\sigma_x$ (viewed as a map to $\mathbb Z^n$, rather than $G_x$).   Recall that $(\mathbb Z^n,X)$ is contracting if and only if  the spectral radius of $A$ is less than $1$.  This is equivalent to $A\inv$ having spectral radius greater than $1$, and hence $A$ is a dilation.  In this case, $\cs(\mathscr G_{(\mathbb Z^n,X)})$ can be viewed as the Exel crossed product for the transpose of $A\inv$; see~\cite[Section~3.1]{LRRW14} and~\cite{EHR11}. 

Homology and K-theory computations for free abelian groups appear in Subsection \ref{subs:freeabelian}.
\end{Example}

\begin{Example}[Sausage automaton]\label{ex:sausage}
There is a faithful self-similar action of $\mathbb Z^n$ over the alphabet $\{0,1\}$ given as follows.  Let $e_0,\ldots, e_{n-1}$ be the standard basis of $\mathbb Z^n$.  Then $e_0$ acts on $\{0,1\}$ as the nontrivial permutation and $e_1,\ldots, e_{n-1}$ act trivially.  The $1$-cocycle is given by $e_0|_0= 0$, $e_0|_1=e_{n-1}$ and $e_i|_j = e_{i-1}$ for $1\leq i\leq n-1$ and $j\in \{0,1\}$.  Note that the stabilizer of $0$ is $\langle 2e_0,e_1,\ldots, e_{n-1}\rangle$ and the virtual endomorphism $\sigma_0$ is given by $2e_0\mapsto e_{n-1}$ and $e_i\mapsto e_{i-1}$ for $1\leq i\leq n-1$.  Thus the matrix for $\sigma_0\otimes 1_{\mathbb Q}$ is given by 
\[A= \begin{bmatrix}0 & 1 & 0 &\cdots & 0\\ 0 & 0 & 1 &\cdots & 0\\ \vdots & \vdots & \ddots &\ddots & \vdots\\ 0 & 0 & \cdots& \ddots& 1\\ 1/2 & 0 &\cdots& \cdots & 0\end{bmatrix}.\]   Clearly, $\sigma_0$ is surjective and $A$ has spectral radius $1/2$.   Thus the action is contracting and self-replicating.
\end{Example}

\section{Tools for computing homology and K-theory}

\subsection{Homology of ample groupoids}
 If $X$ is a space with a basis of compact\footnote{In this paper compactness includes the Hausdorff axiom.} open sets and $A$ is an abelian group (written additively), then $AX$ denotes the abelian group of mappings $X\to A$ spanned by elements of the form $a1_U$ where $a \in A$ and $1_U$ is the characteristic function of a compact open set $U \subseteq X$. If $X$ is Hausdorff, these are precisely the compactly supported locally constant mappings $X\to A$.  
We shall use, frequently without comment, that $AX \cong \Z X\otimes_{\Z} A$; cf.~\cite[Corollary~2.3]{li2022ample}.

The construction $X\mapsto AX$ is functorial with respect to \'etale maps  and contravariantly functorial with respect to proper maps.
If $f\colon X\to Y$ is  \'etale, then $f_\ast\colon AX\to AY$ is given by 
$f_*(g)(y)= \sum_{x\in f\inv(y)}g(x)$. If $p\colon X\to Y$ is a proper map, then $p^*\colon AY \to AX $ is given by $p^*(f) =f\circ p$.
Notice that an open inclusion $i$ is \'etale with $i_*$ extension by $0$,   and a closed inclusion $i$ is proper with $i^*$ restriction of functions.

We recall now the definition of the homology~\cite{CMhomology00} of an ample groupoid $\mathscr G$ with coefficients in an abelian group $A$ via the formulation of Matui~\cite{Matui}.  There are \'etale maps $d_i\colon \mathscr G^n\to \mathscr G^{n-1}$ for $n\geq 2$ and  $i=0,\ldots, n$ given by
\begin{equation}\label{eq:face.maps}
d_i(g_1,\ldots,g_n) = \begin{cases}(g_2,\ldots, g_n), & \text{if}\ i=0\\ (g_1,\ldots,g_ig_{i+1},\ldots,g_{n}), & \text{if}\ 1\leq i\leq n-1,\\ (g_1,\ldots, g_{n-1}), & \text{if}\ i=n.\end{cases}
\end{equation}
We define $d_0,d_1\colon \mathscr G^1\to \mathscr G^0$ by $d_0(g) = \sour(g)$ and $d_1(g)=\ran(g)$, which are again \'etale.
It is well known that these maps satisfy the semisimplicial identities, 
and so we can define a chain complex $C_\bullet(\mathscr G,A)$ with $C_n(\mathscr G,A) =  A \mathscr G^n$ for $n\geq 0$, and $\partial_n\colon C_n(\mathscr G,A)\to C_{n-1}(\mathscr G,A)$ given by $\partial_n=\sum_{i=0}^n(-1)^i(d_i)_*$ for $n\geq 1$.  As usual, we take $\partial_0=0$.  The homology of this chain complex is denoted $H_\bullet(\mathscr G,A)$.  When $A=\mathbb Z$, we often write $C_\bullet(\mathscr G)$ and $H_\bullet(\mathscr G)$. There is a picture of groupoid homology in terms of the Tor functor with $H_\bullet(\mathscr G,A) \cong \Tor^{\Z \mathscr G}_\bullet(\Z \mathscr G^0, A \mathscr G^0)$ \cite[Theorem 2.5]{li2022ample}, and for general $\Z \mathscr G$-modules $M$ we set $H_\bullet(\mathscr G ; M) = \Tor^{\Z \mathscr G}_\bullet(\Z \mathscr G^0, M)$;\footnote{We warn readers of the subtlety that the $\Z \mathscr G$-module associated to an abelian group $A$ is $A\mathscr G^0$.}
this is the left derived functor of the $\mathscr G$-coinvariants $M\mapsto M_{\mathscr G}=M/\langle 1_Um-1_{\sour(U)}m\rangle$ where $U$ ranges over compact open bisections of $\mathscr G$ and $m$ ranges over $M$. Already in \cite{CMhomology00} it was shown that groupoid homology is invariant under Morita equivalence.

We shall need later the Universal Coefficient Theorem for the homology of ample groupoids.  Since we could not find a reference in the literature, we record the proof here.

\begin{Thm}[Universal Coefficient Theorem]\label{T:uct}
Let $\mathscr G$ be an ample groupoid and $A$ an abelian group.   Then, for all $n\geq 0$, there is an exact sequence (natural in $\mathscr G$ with respect to \'etale homomorphisms and proper ones) \[0\to H_n(\mathscr G)\otimes_{\mathbb Z} A\to H_n(\mathscr G,A)\to \Tor_1^{\mathbb Z}(H_{n-1}(\mathscr G),A)\to 0\] which splits, but not naturally.
\end{Thm}
\begin{proof}
A classical result of N\"obeling says that if $X$ is a set, then the additive group of bounded functions from $X\to \mathbb Z$ is free abelian; see~\cite[Corollary~1.2]{B72}.  Since $\mathbb ZX$ consists of bounded functions, it follows that it is free abelian, and hence $C_\bullet(\mathscr G)$ is a chain complex of free abelian groups.  Since $C_\bullet(\mathscr G,A) = C_\bullet(\mathscr G)\otimes_{\mathbb Z} A$ by~\cite[Corollary~2.3]{li2022ample},  the result follows from the Universal Coefficient Theorem for chain complexes of free abelian groups, cf.~\cite[Corollary~7.56]{Rotmanhom}. 
\end{proof}

In particular, $H_n(\mathscr G,\mathbb Q)\cong H_n(\mathscr G)\otimes_{\mathbb Z}\mathbb Q$.

\subsection{Six-term and long exact sequences}

Our goal is to compute the K-theory of $\mathcal O_{\mathcal X}$ and the groupoid homology of $\mathscr G_{\mathcal X}$. For the K-theory, we apply the six-term sequence associated to the relative Cuntz--Pimsner algebra \cite[Proposition 8.7]{Katsura04}:
\begin{equation}\label{Katsura six term}
\begin{tikzcd}[arrow style=math font, column sep = 1.9cm]
	{K_0(\cs(G_\reg))} & {K_0(\cs(G))} & {K_0(\mathcal O_{\mathcal X})} \\
	{K_1(\mathcal O_{\mathcal X})} & {K_1(\cs(G))} & {K_1(\cs(G_\reg))}
	\arrow["{[\iota] - [M_{\mathcal X}]}", from=1-1, to=1-2]
	\arrow[from=1-2, to=1-3]
	\arrow[from=1-3, to=2-3]
	\arrow[from=2-1, to=1-1]
	\arrow[from=2-2, to=2-1]
	\arrow["{[\iota] - [M_{\mathcal X}]}", from=2-3, to=2-2]
\end{tikzcd}
\end{equation}
Here $[M_{\mathcal X}] \colon K_i(\cs(G_\reg)) \to K_i(\cs(G))$ is the map in K-theory induced by the proper correspondence $M_{\mathcal X} \colon \cs(G_\reg) \to \cs(G)$, and the unmarked horizontal maps are induced by the nondegenerate $*$-homomorphism $\cs(G) \to \mathcal O_{\mathcal X}$. In general, a proper $\cs$-correspondence $E \colon A \to B$ induces a map in K-theory $[E] \colon K_i(A) \to K_i(B)$, e.g., by \cite[Remark B.4]{Katsura04}. It is straightforward to check that this is compatible with isomorphism and composition of $\cs$-correspondences, and that when $E$ is isomorphic to $B^n$ as a Hilbert $B$-module, this is simply induced by the resulting $*$-homomorphism $A \to M_n(B)$.

For groupoid homology we will construct an analogue of the six-term K-theory sequence \eqref{Katsura six term}. The Toeplitz extension is modelled at the groupoid level by the universal groupoid $\universal{\mathcal X}$ of the inverse semigroup $S_{\mathcal X}$, and $\mathscr G_{\mathcal X}$ is the reduction to the closed invariant set  of tight filters. 

The analogue of the K-theoretic isomorphism $K_*(\cs(G)) \to K_*(\mathcal T_{\mathcal X})$ is as follows. It was observed in \cite[Example 3.10]{miller2023ample} that $H_*(\universal{\mathcal X})$ is isomorphic to the homology $H_*(\underlying{S_{\mathcal X}})$ of the underlying groupoid $\underlying{S_{\mathcal X}}$ of $S_{\mathcal X}$. The underlying groupoid $\underlying{S_{\mathcal X}}$ is Morita equivalent to $G$ and therefore $H_*(\universal{\mathcal X}) \cong H_*(G)$. Moreover, the reduction of $\universal{\mathcal X}$ to the complement of the tight filters is Morita equivalent to $G_\reg$, so all the groupoid homology groups analogous to those in the six-term K-theory sequence \eqref{Katsura six term} are present. The groupoid homology analogue of the six-term exact sequence in K-theory of an extension of $\cs$-algebras is the following long exact sequence.

\begin{Prop}\label{t:tool1}
Let $A$ be an abelian group, $\mathscr G$ an ample groupoid and $C\subseteq \mathscr G^0$ a closed invariant subspace with complement $U=\mathscr G^0\setminus C$.  Then there is a long exact sequence in homology
\[\cdots\to H_{n+1}(\mathscr G|_C,A)\to  H_n(\mathscr G|_U,A)\to H_n(\mathscr G,A)\to H_n(\mathscr G|_C,A)\to \cdots\]
induced by the inclusions $U\to \mathscr G^0$ and $C\to \mathscr G^0$ of $\mathscr G$-spaces.  
\end{Prop}
This is immediate for Hausdorff $\mathscr G$ because it is clear then that $0 \to A(\mathscr G|_U)^n \to A\mathscr G^n \to A(\mathscr G |_C)^n \to 0$
is exact, as $(\mathscr G |_C)^n = \mathscr G^n \setminus (\mathscr G |_U)^n$. The long exact sequence is presumably well-known outside of the Hausdorff setting too, and there are numerous ways to see this, but we present a proof for convenience along the above lines, as a consequence of \cite[Lemma 2.2]{li2022ample}. This states that if $X$ is a space with an open cover $\mathcal O$ by locally compact Hausdorff totally disconnected spaces, then $AX$ is $(\bigoplus_{V\in \mathcal U} A1_V)/N$ where $\mathcal U$ consists of all compact open subsets contained in some element of $\mathcal O$, and $N$ is generated by all $a1_V+a1_{V'} - a1_{V\cup V'}$ with $V,V'$ disjoint compact open subsets of some common element of $\mathcal O$. 

\begin{Prop}
Let $X$ be a space with a basis of compact open sets,
and let $U \subseteq X$ be an open subspace with complement $C = X \setminus U$. Then for any abelian group $A$, the sequence $0 \to AU \to AX \to AC \to 0$ is exact, where the first map is extension by $0$ 
and the second restriction. 
\end{Prop}
\begin{proof}
Let $\mathcal O$ be an open cover of $X$ by locally compact Hausdorff totally disconnected spaces.
We write $\mathcal U$ for the set of compact open subspaces of $X$ which are contained within some member of $\mathcal O$, and we write $\mathcal U_C$ for the set of compact open subspaces of $C$ contained in some element of $\mathcal O$.

We first claim that each $V \in \mathcal U_C$ is of the form $W \cap C$ for some $W \in \mathcal U$. Well $V = W_0 \cap C$ for some open Hausdorff $W_0$ contained in an element of $\mathcal O$. By considering compact open neighbourhoods construct a compact open set $W \subseteq W_0$ containing $V$, so that $W\cap C=V$.  It follows that restriction $AX\to AC$ is surjective by \cite[Lemma 2.2]{li2022ample}, and $AU$ is contained in the kernel.  We construct an inverse to the induced map $AX/AU \to AC$.

To achieve this, we claim that given $W_1, W_2 \in \mathcal U$ with $W_1 \cap C = W_2 \cap C$, then $a1_{W_1} - a1_{W_2} \in AU$. Since $W_1$ is Hausdorff and $W_1\cap C$ is closed in $W_1$, hence compact, by considering compact open neighbourhoods construct a compact open set $W \subseteq W_1 \cap W_2$ containing $C \cap W_1$. Then $W$ is clopen in $W_i$, whence $W_i \setminus W \subseteq U$ is compact open, for $i=1,2$, and $a1_{W_1} - a1_{W_2} = a1_{W_1 \setminus W} - a1_{W_2 \setminus W}\in AU$.

It follows that there is a well-defined map $\bigoplus_{V\in \mathcal U_C}A1_V \to  AX/AU$ that sends  $a1_V$ to $a1_W + AU$ where $W\in \mathcal U$ with $W\cap C=V$. All we need is to check that this respects disjoint unions within a fixed member of $\mathcal O$. So suppose $V_1,V_2 \in \mathcal U_C$ are disjoint with $V_1 \sqcup V_2 \in \mathcal U_C$. Find $W \in \mathcal U$ with $W \cap C = V_1 \sqcup V_2$.   By considering compact open neighbourhoods and using that $W$ is Hausdorff, construct a compact open set $W_2 \subseteq W \setminus V_1$ containing $V_2$. Setting $W_1 = W \setminus W_2$ and noting that $W_2$ is clopen in $W$, whence $W_1$ is compact open, we see that $a1_{V_1 \sqcup V_2}$ is sent to $a1_W +AU= a1_{W_1} + a1_{W_2}+AU$, which is the sum of what $a1_{V_1}$ and $a1_{V_2}$ are sent to, so it is indeed compatible with disjoint unions.

Thus by \cite[Lemma 2.2]{li2022ample} we obtain a homomorphism $\psi\colon AC \to AX/AU$ with the property that if $W\in \mathcal U$, then $\psi(a1_{W\cap C}) = a1_W+AU$. It then follows that $\psi(f|_C) = f+AU$ for all $f\in AX$, and so if $f|_C=0$, then $f\in AU$, as required.
\end{proof}

In the discussion proceeding~\cite[Proposition~5.3]{gcrccr} it is shown that that if $U$ is an open invariant subspace of $\mathscr G^0$ and $C=\mathscr G^0\setminus U$, then restriction of functions gives a surjective $K$-algebra homomorphism $\pi\colon K\mathscr G\to K\mathscr G|_C$, and it is asserted without proof that $\ker \pi=K\mathscr G|_U$. We may now remedy this gap using the above result.  

\begin{Cor}
Let $\mathscr G$ be an ample groupoid and let $C\subseteq \mathscr G^0$ be a closed invariant subspace with complement $U=\mathscr G^0\setminus C$.  Let $K$ be any commutative ring with unit.  Then $K\mathscr G|_U$ is an ideal of $K\mathscr G$ and $K\mathscr G/K\mathscr G|_U\cong K\mathscr G|_C$.
\end{Cor}

\subsection{\'Etale correspondences}

A proper \'etale correspondence $\Omega \colon G \to H$ of ample groupoids induces a map 
\[ H_*(\Omega) \colon H_*(\mathscr G) \to H_*(\mathscr H) \]
by \cite{miller2023ample}. This is functorial with respect to composition of correspondences, so when $\Omega$ is a Morita equivalence $H_*(\Omega)$ is an isomorphism.  Moreover, isomorphic correspondences induce the same maps on homology (and similarly for K-theory). 
The induced map $H_\ast(\Omega)$ can be understood as follows.   The right $\Z\mathscr H$-module $\Z\Omega$ is flat by~\cite[Proposition~2.7]{miller2023ample}. There is a natural $\mathscr H$-equivariant map $(\Z \Omega)_{\mathscr G}  \to \Z \mathscr H^0$ induced by $\sour$, and hence we have a natural transformation $(\Z\Omega\otimes_{\Z \mathscr H} (-))_{\mathscr G}\to (-)_{\mathscr H}$, which induces a natural homomorphism $H_\bullet(\mathscr G;\Z\Omega\otimes_{\Z\mathscr H} M)\to H_\bullet(\mathscr H;M)$ in $\Z\mathscr H$-modules $M$ since $\Z \Omega \otimes_{\Z \mathscr H} (-)$ is exact.  In the case that $M=A\mathscr H^0$, $\Z\Omega\otimes_{\Z\mathscr H}A\mathscr H^0\cong (\Z\Omega/\mathscr H) \otimes_{\mathbb Z}A$.  The proper map $\Omega/\mathscr H\to \mathscr G^0$ induced by $\ran$ yields a $\Z\mathscr G$-module homomorphism $\Z\mathscr G^0\to \Z\Omega/\mathscr H$, and hence a homomorphism $A\mathscr G^0\to (\Z\Omega/\mathscr H)\otimes_{\mathbb Z}A$, which induces a homomorphism $H_\bullet(\Omega,A)\colon H_\bullet(\mathscr G,A)\to H_\bullet(\mathscr G; \Z\Omega\otimes_{\Z\mathscr H}A\mathscr H^0)\to H_\bullet(\mathscr H,A)$.  It is shown in~\cite{miller2023ample} that the map induced on homology with $\mathbb Z$-coefficients commutes with composition.  The proof 
works \textit{mutatis mutandis} with coefficients $A$ by tensoring all relevant diagrams with $A$.  In particular, since Morita equivalences are given by invertible \'etale correspondences,  $H_\bullet(-,A)$ is invariant under Morita equivalence for any coefficient group $A$.

 Note that a finite disjoint union of proper \'etale correspondences $\Omega_i \colon \mathscr G \to \mathscr H$ is again a proper \'etale correspondence and that disjoint unions become sums on the level of homology (and also in K-theory). 
 Let us describe $H_*(\Omega)$ in a few key examples. 

An \'etale homomorphism $\varphi \colon \mathscr G \to \mathscr H$ induces an \'etale correspondence $\Omega_\varphi \colon \mathscr G \to \mathscr H$ with bispace $\mathscr G^0 \times_{\mathscr H^0} \mathscr H$. For each $n \geq 0$ we obtain an \'etale map $\varphi^n \colon \mathscr G^n \to \mathscr H^n$, which together induce a chain map $(\varphi_\bullet)_* \colon C_\bullet(\mathscr G,A) \to C_\bullet(\mathscr H,A)$. Adding coefficients to \cite[Example 3.8]{miller2023ample}, $H_*(\Omega_\varphi,A)$ is induced by the homology of this chain map.

For an action $\mathscr G \acts X$ on a (totally disconnected) locally compact Hausdorff space $X$ the space $\mathscr G \ltimes X$ is the bispace for an \'etale correspondence $\mathscr G \ltimes X \colon \mathscr G \to \mathscr G \ltimes X$ which we call the associated \textit{action correspondence}. This is proper if and only if the anchor map $\tau \colon X \to \mathscr G^0$ is proper. In this case we obtain a proper map $\tau_n \colon (\mathscr G \ltimes X)^n \to \mathscr G^n$ for each $n \geq 0$, which together induce a chain map $(\tau_\bullet)^* \colon C_\bullet(\mathscr G,A) \to C_\bullet(\mathscr G \ltimes X,A)$. This induces $H_*(\mathscr G \ltimes X,A) \colon H_*(\mathscr G,A) \to H_*(\mathscr G \ltimes X,A)$ in homology by \cite[Example 3.9]{miller2023ample}.  

We can therefore compute the induced map in homology of a proper \'etale correspondence $\mathcal X \colon \mathscr G \to\mathscr H$ if we have a decomposition of $\mathcal X$ into a proper action correspondence and an \'etale homomorphism. We may obtain a decomposition exactly of this form from an $\mathscr H$-transversal $E \subseteq \mathcal X$. We say $E$ is a \textit{continuous} $\mathscr H$-transversal if the transversal map $\mathcal X \to E$ is continuous.
\begin{Prop}\label{p:actionhomdecomposition}
Let $\mathcal X \colon \mathscr G \to \mathscr H$ be an \'etale correspondence and let $E \subseteq \mathcal X$ be a continuous $\mathscr H$-transversal. Then there is an action $\mathscr G \acts E$ with anchor $\ran$, written $(g,e) \mapsto g(e)$, determined by $g(e) \in ge\mathscr H$ and an \'etale homomorphism $\sigma \colon \mathscr G \ltimes E \to \mathscr H$, written $(g,e) \mapsto g |_e$, determined by $ge = g(e) g |_e$. Moreover, the associated \'etale correspondences compose to form $\mathcal X \colon \mathscr G \to \mathscr H$.
\end{Prop}
\begin{proof}
For each $g \in \mathscr G$ and $e \in E$ with $\sour(g) = \ran(e)$ there is a unique $g(e) \in E \cap ge\mathscr H$ by transversality, and this assignment is continuous by continuity of $E$. This defines our action $\mathscr G \acts E$ with anchor $\ran$. Given $(g,e) \in \mathscr G \ltimes E$ there is a unique $g |_e \in \mathscr H$ with $ge = g(e) g|_e$ because the action $\mathcal X \rightacts \mathscr H$ is free. This defines our homomorphism $\sigma \colon \mathscr G \ltimes E \to \mathscr H$ which is continuous by continuity of $\langle - , - \rangle \colon \mathcal X \times_{\mathcal X/\mathscr H} \mathcal X \to \mathscr H$ and \'etale because it restricts to $\sour$ on the unit space.

The \'etale correspondence of $\sigma \colon \mathscr G \ltimes E \to \mathscr H$ has bispace $E \times_{\mathscr H^0} \mathscr H$, which is $\mathscr H$-equivariantly homeomorphic to $\mathcal X$ via the map $(e,h) \mapsto eh$. Through this homeomorphism $(g,e) \in \mathscr G \ltimes E$ acts on $x \in \mathcal X$ with $x = eh$ by $(g,e)x = g(e) g|_e h = gx$. Composition with the action correspondence $\mathscr G \to \mathscr G \ltimes E$ does not change the underlying $\mathscr H$-space $\mathcal X$, and the left action becomes $(g,x) \mapsto gx$, which is to say we recover $\mathcal X \colon \mathscr G \to \mathscr H$.
\end{proof}

In this setting we typically write without mention $(g,e) \mapsto g(e)$ for the action $\mathscr G \acts E$ and $(g,e) \mapsto g |_e$ for the homomorphism $\sigma \colon \mathscr G \ltimes E \to \mathscr H$. Putting this decomposition together with the above description of the map induced in homology by proper action correspondences and \'etale homomorphisms, we obtain: 

\begin{Prop}\label{p:description of map in homology}
Let $\mathcal X \colon \mathscr G \to \mathscr H$ be a proper \'etale correspondence with continuous $\mathscr H$-transversal $E \subseteq \mathcal X$. Then the induced map in homology
\[H_*(\mathcal X) \colon H_*(\mathscr G,A) \to H_*(\mathscr H,A)\] 
is induced by the chain map $C_\bullet(\mathscr G) \to C_\bullet(\mathscr H)$ given at $n \geq 0$ by $\xi \mapsto \tilde \xi \colon C_n(\mathscr G) \to C_n(\mathscr H)$, where 
\begin{equation}\label{ample description of map}
\tilde \xi  \colon (h_1, \dots, h_n) \mapsto \sum_{e \in E, \sour(e) = \sour(h_n)} \sum_{\substack{g_\bullet \in \mathscr G^n, \sour(g_n) = \ran(e) \\ g_i |_{g_{i+1}\cdots g_n(e)} = h_i }} \xi(g_1,\dots, g_n).
\end{equation}
If $\mathscr G = G$ and $\mathscr H = H$ are discrete, this can be expressed as
\begin{align}\label{discrete description of map}
C_n(G) & \to C_n(H) \notag \\ 
(g_1,\dots,g_n) & \mapsto \sum_{e \in E, \ran(e) = \sour(g_n)} (g_1 |_{g_2\cdots g_n(e)}, \dots, g_n |_e).
\end{align}
\end{Prop}

Note that the projection $\mathcal X\to \mathcal X/\mathscr H$ is \'etale and $\mathcal X/\mathscr H$ is locally compact, Hausdorff and totally disconnected~\cite{AKM22}.  Thus if $\mathcal X$ is $\sigma$-compact, then it admits a continuous $\mathscr H$-transversal by a standard argument.

Proposition~\ref{p:description of map in homology}, naturality of the Universal Coefficient Theorem and the short five lemma imply that if $\mathcal X\colon \mathscr G\to \mathscr H$ is a proper \'etale correspondence with a continuous $\mathscr H$-transversal and $H_*(\mathcal X)\colon H_*(\mathscr G)\to H_*(\mathscr H)$ is an isomorphism, then $H_*(\mathcal X)\colon H_*(\mathscr G,A)\to H_*(\mathscr H,A)$ is an isomorphism for all abelian groups $A$.

Up to Morita equivalence, discrete groupoids are given by disjoint unions of groups. Explicitly, given a transversal $T \subseteq G^0$ for a discrete groupoid $G$, the inclusion of the isotropy groups $\bigsqcup_{v \in T} G^v_v \hookrightarrow G$ is a Morita equivalence, and thus we get isomorphisms
\begin{align*}
H_*(G) & \cong \bigoplus_{v \in T} H_*(G^v_v) \\
K_*(\cs(G)) & \cong \bigoplus_{v \in T} K_*(\cs(G^v_v))
\end{align*}
in homology and K-theory. Through this principle an \'etale correspondence $\mathcal X \colon G \to H$ of discrete groupoids can be broken down into group-theoretic information. We start with action correspondences:

\begin{Prop}\label{p:actiondecomposition}
Let $G$ be a discrete groupoid with a discrete $G$-space $E$, suppose that $T_G \subseteq G^0$ is a transversal for $G$ and pick a $G$-transversal $T_E \subseteq E$ with $\ran(T_E) \subseteq T_G$. For $t \in T_E$ consider the stabilizer group $G_t = \{ g \in G^{\ran(t)}_{\ran(t)} \mid g t = t \}$, which includes into $G \ltimes E$ as the isotropy group at $t$. Then there is a commutative diagram up to isomorphism
\[\begin{tikzcd}[arrow style=math font]
	G & {} & {G \ltimes E} \\
	{\bigsqcup_{v \in T_G} G^v_v} && {\bigsqcup_{t \in T_E} G_t}
	\arrow[from=1-1, to=1-3]
	\arrow[hook, from=2-1, to=1-1]
	\arrow["{\bigsqcup_{t \in T_E} \tr_t}", from=2-1, to=2-3]
	\arrow[hook, from=2-3, to=1-3]
\end{tikzcd}\]
where $\tr_t\colon G^{\ran(t)}_{\ran(t)} \to G_t$ is the \'etale correspondence with bispace $G^{\ran(t)}_{\ran(t)}$. 
\end{Prop}
\begin{proof}
For $t \in T_E$ the composition of $\tr_t \colon G^{\ran(t)}_{\ran(t)} \to G_t$ with the inclusion $G_t \hookrightarrow G \ltimes E$ has bispace $G^{\ran(t)}_{\ran(t)} \times_{G_t} t(G \ltimes E)$, which is isomorphic to \[ G^{\ran(t)}_{\ran(t)} t (G \ltimes E) = \{ (g,e) \in G \ltimes E \mid \ran(g) = \ran(t), e \in Gt \} \]
via the map $[g,(h,e)]_G \mapsto (gh,e)$. The inverse map is $(g,e)\mapsto [gh,(h\inv,e)]_G$ where $e=ht$.  If we fix $v \in T_G$ and take the disjoint union over $t \in T_E \cap \ran \inv(v)$ we obtain $\{ (g,e) \in G \ltimes E \mid \ran(g) = v \}$ which is the bispace for the composition $G^v_v \to G \to G \ltimes E$.
\end{proof}
Note that the correspondence $\tr_t \colon G^{\ran(t)}_{\ran(t)} \to G_t$ is proper if and only if $G_t$ has finite index in $G^{\ran(t)}_{\ran(t)}$, which happens for every $t \in T_E$ if and only if $\ran \colon E \to G^0$ is proper. We combine Propositions \ref{p:actionhomdecomposition} and \ref{p:actiondecomposition} into a single statement for convenience. 

\begin{Prop}\label{p:transversal decomposition}
Let $\mathcal X \colon G \to H$ be an \'etale correspondence of discrete groupoids, and let $T_G \subseteq G^0$ and $T_H \subseteq H^0$ be transversals for $G$ and $H$. Pick an $H$-transversal $E \subseteq \mathcal X$ and, for each $w \in H^0$, write $t_w \in T_H$ for its image under the transversal map and pick $h_w \in H$ with $\ran(h_w) = w$ and $\sour(h_w) = t_w$. Pick a $G$-transversal $T_E \subseteq E$ with $\ran(T_E) \subseteq T_G$ and for $e \in T_E$ consider the stabilizer group $G_e = \{ g \in G^{\ran(e)}_{\ran(e)} \mid ge \in eH \}$. 

Then there is a commutative diagram up to isomorphism
\[ \begin{tikzcd}[arrow style=math font, column sep = 1.5cm]
	G & & H \\
	{\bigsqcup_{v \in T_G} G^v_v}  & & {\bigsqcup_{w \in T_H} H^w_w}
	\arrow["{\mathcal X}", from=1-1, to=1-3]
	\arrow[hook, from=2-1, to=1-1]
        \arrow["{\bigsqcup_{e \in T_E} C_{\sour(e)} \circ \Sigma_e \circ \tr_e }", from=2-1, to=2-3]
	\arrow[hook, from=2-3, to=1-3]
\end{tikzcd} \]
where for $e \in T_E$ the \'etale correspondence $\tr_e \colon G^{\ran(e)}_{\ran(e)} \to G_t$ has bispace $G^{\ran(e)}_{\ran(e)}$, $\Sigma_e \colon G_e \to H^{\sour(e)}_{\sour(e)}$ is the \'etale correspondence of the homomorphism $\sigma_e\colon G_e \to H^{\sour(e)}_{\sour(e)}$ given by $g \mapsto g |_e$ and, for $w \in \sour(E)$, $C_w \colon H^{w}_{w} \to H^{t_{w}}_{t_{w}}$ is the \'etale correspondence of the homomorphism $c_w\colon H^{w}_{w} \to H^{t_{w}}_{t_{w}}$ given by $h \mapsto  h_{w} \inv h h_{w} $.
\end{Prop}
\begin{proof}
Through Proposition \ref{p:actionhomdecomposition} we construct the following diagram.
\begin{equation}\label{eq:bicategory.commuting}
\begin{tikzcd}[arrow style=math font, column sep = 1.2cm]
	G & {} & {G \ltimes E} && H \\
	{\bigsqcup_{v \in T_G} G^v_v} && {\bigsqcup_{e \in T_E} G_e} && {\bigsqcup_{w \in T_H}H^w_w}
	\arrow[from=1-1, to=1-3]
	\arrow["\sigma", from=1-3, to=1-5]
	\arrow[hook, from=2-1, to=1-1]
	\arrow["{\bigsqcup_{e \in T_E} \tr_e}", from=2-1, to=2-3]
	\arrow[hook, from=2-3, to=1-3]
	\arrow["{\bigsqcup_{e \in T_E} C_{\sour(e)} \circ \Sigma_e}", from=2-3, to=2-5]
	\arrow[hook, from=2-5, to=1-5]
\end{tikzcd}
\end{equation}
The left square commutes up to isomorphism by Proposition \ref{p:actiondecomposition}.  By construction, there is a commutative diagram of functors
\[\begin{tikzcd}[arrow style=math font, column sep = 1.2cm]
	{G\ltimes E} && H \\
	 {\bigsqcup_{e \in T_E} G_e} && {\bigsqcup_{w \in T_H}H^w_w}
	\arrow["\sigma", from=1-1, to=1-3]
	\arrow[hook, from=2-1, to=1-1]
	\arrow["{\bigsqcup_{e \in T_E} c_{\sour(e)} \circ \sigma_e}", from=2-1, to=2-3]
	\arrow["{\psi}", from=1-3, to=2-3]
\end{tikzcd}\]
where $\psi(h) = h_{\ran(h)}\inv hh_{\sour(h)}$.  Moreover, the composition of $\psi$ with the inclusion $\bigsqcup_{w \in T_H}H^w_w\hookrightarrow  H$ is naturally isomorphic to the identity functor on $H$.  Therefore, 
the right square of \eqref{eq:bicategory.commuting} commutes up to isomorphism.  
\end{proof}

\subsection{The transfer map}
Let us justify the notation $\tr_e \colon G^{\ran(e)}_{\ran(e)} \to G_e$. For a group $G$ with a subgroup $H$ the \textit{transfer correspondence} $\tr^G_H \colon G \to H$ is the \'etale correspondence with bispace $G$ via left and right multiplication. This is proper if and only if $H$ has finite index, in which case it induces maps in K-theory and homology. 
Given a group $G$ with a finite index subgroup $H$ and a $\Z G$-module $M$, the transfer map $\mathrm{tr}^{G}_H\colon H_n(G;M)\to H_n(H;M)$ is a homomorphism that can be described in a number of equivalent ways~\cite{Browncohomology}. For example, a slight modification of ~\cite[Page~81, (C)]{Browncohomology} says that for a projective resolution $P_\bullet \to M$ of $\Z G$-modules, $\tr^G_H$ is induced by the chain map 
\begin{equation}\label{transfer equation}
[x]_G \mapsto \sum_{gH \in G/H} [g^{-1} \cdot x]_H \colon (P_\bullet)_G \to (P_\bullet)_H. 
\end{equation}

This is related to the transfer correspondence $\tr^G_H \colon G \to H$ as follows. Consider the $\mathbb ZG$-module map
\begin{align*}
\iota_M \colon M & \to \mathbb Z G \otimes_{\mathbb ZH} M \\
m & \mapsto \sum_{gH \in G/H} g \otimes g^{-1} \cdot m.
\end{align*}
Following \cite[Theorem 3.5]{miller2023ample}, $\tr^G_H \colon G \to H$ and $\iota_M \colon M \to \mathbb Z G \otimes_{\mathbb ZH} M$ induce a map $H_*(\tr^G_H; \iota_M) \colon H_*(G;M) \to H_*(H;M)$ in homology. Moreover, for $M = A$, an abelian group with trivial action, this is the map $H_*(\tr^G_H,A) \colon H_*(G,A) \to H_*(H,A)$ induced by $\tr^G_H \colon G \to H$ as a proper \'etale correspondence.
\begin{Prop}
Let $G$ be a group with a finite index subgroup $H$, and let $M$ be a $\Z G$-module. Then the transfer map $\tr^G_H \colon H_*(G;M) \to H_*(H;M)$ is equal to $H_*(\tr^G_H; \iota_M)$. In particular, $\tr^G_H \colon H_*(G,A) \to H_*(H,A)$ is the map induced by the proper \'etale correspondence $\tr^G_H \colon G \to H$ for any abelian group $A$.
\end{Prop}
\begin{proof}
Let us explain in some more detail the construction of 
\[ H_*(\tr^G_H; \iota_M) \colon H_*(G;M) \to H_*(H;M).\]
There is a right $\Z H$-module map $\delta_{\tr^G_H} \colon (\Z G)_G \to \Z$ which sends $[g]_G$ to $1$ (see \cite[Proposition 3.3]{miller2023ample}). 
Then for any $H$-module $N$, after the identifications $(\Z G)_G = \Z \otimes_{\Z G} \Z G$ and $N_H = \Z \otimes_{\Z H} N$, consider the map $\delta_{\tr^G_H} \otimes \id_N \colon (\Z G \otimes_{\Z H} N)_G \to N_H$, which sends $[g \otimes n]_G$ to $[n]_H$. 
For any projective resolution $P_\bullet \to M$ of $\Z G$-modules and projective resolution $Q_\bullet \to M$ of $\Z H$-modules and any chain map $f \colon P_\bullet \to \Z G \otimes_{\Z H} Q_\bullet$ of $\Z G$-modules over $\iota_M \colon M \to \Z G \otimes_{\Z H} M$, then $H_*(\tr^G_H ; \iota_M)$ is induced by the chain map \[(\delta_{\tr^G_H} \otimes \id_{Q_\bullet}) \circ f_G \colon (P_\bullet)_G \to (Q_\bullet)_H.\]
We may take any projective resolution $P_\bullet \to M$ of $\Z G$-modules and then set $Q_\bullet = P_\bullet$ and $f = \iota_{P_\bullet}$. The composition is then given by
\begin{align*}
(\delta_{\tr^G_H} \otimes \id_{P_\bullet}) \circ (\iota_{P_\bullet})_G \colon (P_\bullet)_G & \to (P_\bullet)_H \\
[x]_G & \mapsto \sum_{gH \in G/H} [g^{-1} \cdot x]_H,
\end{align*}
which is precisely the map given in \eqref{transfer equation}.
\end{proof}

We also write $\tr^G_H \colon K_*(\cs(G)) \to K_*(\cs(H))$ for the induced map in K-theory. We make use of the following description of this map.

\begin{Prop}\label{p:transfer.monomial.form}
Let $G$ be a group and $H$ a finite index subgroup with a (finite) transversal $T \subseteq G$. Then the transfer map $\mathrm{tr}^G_H \colon K_i(\cs(G)) \to K_i(\cs(H))$ is induced by the $*$-homomorphism $\psi \colon \cs(G) \to M_T(\cs(H))$ given for $g \in G$ by 
\[ \psi(u_g) =  \left( u_{t_1^{-1}gt_2}  \delta_{t_1 H, g t_2 H} \right)_{t_1,t_2 \in T}. \]
In particular, $\tr^G_H([1]_0) = [G:H][1]_0$.
\end{Prop}
\begin{proof}
Write $s_T \colon G \to T$ and $s_H \colon G \to H$ for functions satisfying $g = s_T(g)s_H(g)$ for each $g \in G$. As a right $H$-set, the transfer correspondence decomposes as a disjoint union $G = \bigsqcup_{t \in T} tH \cong T \times H$. Through this the Hilbert $\cs(H)$-module of the proper $\cs$-correspondence $\cs(\mathrm{tr}^G_H) \colon \cs(G) \to \cs(H)$ is isomorphic to $\bigoplus_T \cs(H)$. For $g \in G$, the action of $u_g \in \cs(G)$ on $\bigoplus_T \cs(H)$ is given at $(t,h) \in \mathbb C(T \times H) \subseteq \bigoplus_T \cs(H)$ by 
\[ u_g \cdot (t,h) = (s_T(gth),s_H(gth)). \]   Note that $s_T(gth)=s_T(gt)$ is the unique element $t_1\in T$ with $t_1\inv gt\in H$, and $s_H(gth) = s_H(gt)h=t_1\inv gth$.  
Thus $u_g$ acts as the matrix $\psi(u_g)$.
\end{proof}

\begin{Rmk}
The definition of $\psi(u_g)$ can be written more succinctly as $\psi(g)=T\inv (u_g\id)P_gT$  where, abusing notation, $T$ is the diagonal matrix with entries $u_t$, $t\in T$, and $P_g$ is the permutation matrix for the action of $g$ on $T \cong G/H$.  From this, it is immediate that a change of transversal results in a unitarily equivalent $\ast$-homomorphism.
\end{Rmk}

\begin{Prop}[Mackey decomposition]\label{p:mackey}
Let $G$ be a group and $H,K$ be subgroups with $H$ of finite index.  Let $\iota_K\colon K\to G$ be the inclusion.  Then \[\mathrm{tr}^G_H\circ I_K\cong\bigsqcup_{KsH\in K\backslash G/H} C_s\circ \tr^K_{K\cap sHs\inv}\] where $I_K$ is the correspondence associated to $\iota_K$ and $C_s$ is the correspondence associated to $c_s:K\cap sHs\inv\to H$, $c_s(k)=s\inv ks$.  
\end{Prop}
\begin{proof}
As a $K$-$H$-bispace $G=\bigsqcup_{KsH\in K\backslash G/H}KsH$.  So it suffices to observe that $K\times_{K\cap sHs\inv} H\to KsH$ given by $[k,h]\mapsto ksh$ is a well defined $K$-$H$-bispace map with inverse $ksh\mapsto [k,h]$. Indeed, if $g\in K\cap sHs\inv$, then $kgsh =ks(s\inv gs)h$.  If $ksh=k'sh'$, then $k\inv k' = sh(h')\inv s\inv\in K\cap sHs\inv$ and $(k(k\inv k'),s\inv(sh(h')\inv s\inv)\inv sh)=(k',h')$. Therefore, these two maps are well defined, and they are clearly inverse.
\end{proof}

Specializing to $K=H$ a normal subgroup, we obtain.
\begin{Cor}\label{c:inc.transfer}
Let $G$ be a group and $N$ a finite index normal subgroup.  Let $\iota_N\colon N\to G$ be the inclusion.  Then $\mathrm{tr}^G_N\circ K_i(\iota_N)= [G:N]\id$ for $i=0,1$ and $\tr^G_N \circ H_n(\iota_N) = [G:N] \id$ for $n \geq 0$.   
\end{Cor}
\begin{proof}
By the Mackey decomposition, $\tr^G_N\circ I_N\cong \bigsqcup_{sN\in G/N} C_s$ where $C_s$ is the correspondence associated to $c_s\colon N\to N$ given by $c_s(n) = s\inv ns$.  Since inner automorphisms are trivial on homology and K-theory,  we deduce that $\mathrm{tr}^G_N\circ K_i(\iota_N)=[G:N]\id$ and $\tr^G_N \circ H_n(\iota_N) = [G:N] \id$.  
\end{proof}

\subsection{The Rukolaine map}

The missing ingredient in the long exact sequence in homology for a self-similar groupoid action $(G,\mathcal X)$, compared to the six-term sequence in K-theory, is a description of the map between the `known' quantities. In K-theory, this map is given by \[\id - [M_{\mathcal X}] \colon K_*(\cs(G_\reg)) \to K_*(\cs(G)),\] so our natural goal is to show that in homology we get \[\id - H_*(\mathcal X_{\reg}) \colon H_*(G_\reg) \to H_*(G),\] where $\mathcal X_\reg \colon G_\reg \to G$ is the restriction of $\mathcal X$ to $G_\reg$, which is then proper.

\begin{Prop}\label{p:Morita.equiv}
Let $(G,\mathcal X)$ be a self-similar groupoid action.  Then there is a Morita equivalence $G\to \underlying{S_{\mathcal X}}$, restricting to a Morita equivalence $G_{\mathrm{reg}}\to \underlying{S_{\mathcal X}}|_F$ where $F=\{pp^\ast\mid \sour(p)\in G^0_{\mathrm{reg}}\}$.
\end{Prop}
\begin{proof}
The inclusion $G \hookrightarrow S_{\mathcal X}$ becomes an embedding into the underlying groupoid $\underlying{S_{\mathcal X}}$.  Since $p$ is an arrow from $\sour(p)$ to $pp^*$, we see that this embedding is a Morita equivalence, restricting to a Morita equivalence  $G_{\mathrm{reg}}\to \underlying{S_{\mathcal X}}|_F$.    
\end{proof}

The long exact sequence arises from the universal groupoid $\universal{\mathcal X}$, the closed invariant set $X_{\mathrm{tight}} \subseteq \universal{\mathcal X}^0$ of tight filters and its open, discrete complement $U = \universal{\mathcal X}^0 \setminus X_{\mathrm{tight}}$. Recall that $U$ consists of the principal filters $\chi_{pp^*}$ associated to finite paths $p \in \mathcal X^+$ beginning at regular vertices, and the inclusion of $G_\reg$ into $\universal{\mathcal X} |_U$ which sends $g \in G_\reg$ to $[g,\chi_{\sour(g)}]$ is a Morita equivalence as $\universal{\mathcal X} |_U\cong \underlying{\mathcal X}|_F$. There is an isomorphism $H_*(\universal{\mathcal X},A) \cong H_*(G,A)$, which, following \cite[Example 3.10]{miller2023ample}, is induced by a proper \'etale correspondence $\Omega_{S_{\mathcal X}} \colon \underlying{S_{\mathcal X}} \to \universal{\mathcal X}$ and the Morita equivalence $\underlying{S_{\mathcal X}} \sim_M G$.

Let us describe the proper \'etale correspondence $\Omega_S \colon \underlying{S} \to \universal{S}$ for an arbitrary inverse semigroup $S$ with $0$ with idempotent semilattice $E$. The semigroup ring $\Z E$\footnote{For semigroups $S$ with zero, we understand the semigroup ring $\Z S$ to have basis $S^\times$ where the product extends that on $S\subseteq \Z S$.} 
is isomorphic to the function ring $\Z \wh E$ via the map which sends $e \in E$ to the indicator on the compact open set $U_e = \{ \chi \in \wh E \mid \chi(e) = 1 \}$, cf.~\cite{mygroupoidalgebra}. As abelian groups we may view $\Z E$ as the homology or K-theory of the discrete space $E^\times$ and $\Z \wh E$ as the homology or K-theory of $\wh E$, and then this isomorphism is induced by the proper \'etale correspondence
\[ \bigsqcup_{e \in E^\times} U_e \colon E^\times \to \widehat E. \]
Here, the left anchor map picks out the index of the disjoint union and the right anchor map includes each compact open $U_e$ into $\widehat E$. That the correspondence is \'etale and proper is reflected respectively by the openness and compactness of each $U_e$. Moreover, there is an action of $S$ on $\bigsqcup_{e \in E^\times} U_e$ given by $s \cdot (e,\chi) = (ses^*, s \cdot \chi)$ for $s \in S$, $0 \ne e \leq s^* s$ and $\chi \in U_e$. This forms an $S$-equivariant \'etale correspondence in the following sense:

\begin{Def}
Let $S$ be an inverse semigroup and let $X$ and $Y$ be totally disconnected locally compact Hausdorff $S$-spaces. An \textit{$S$-equivariant topological correspondence} $Z \colon X \to Y$ is an \'etale correspondence equipped with an action $S \acts Z$ such that the range and source maps $\rho \colon Z \to X$ and $\sigma \colon Z \to Y$ are $S$-equivariant, and $\rho^{-1}(\mathrm{dom}_X(s)) = \mathrm{dom}_Z(s)$ for each $s \in S$.
\end{Def}

The condition that $\rho^{-1}(\mathrm{dom}_X(s)) = \mathrm{dom}_Z(s)$ enables us to construct an action $S \ltimes X \acts Z$ yielding $S \ltimes Z$, and the source map $\sigma \colon Z \to Y$ induces an \'etale homomorphism $\sigma \colon S \ltimes Z \to S \ltimes Y$. We write $\tilde Z \colon S \ltimes X \to S \ltimes Y$ for the resulting \'etale correspondence, which has bispace $Z \times_Y (S \ltimes Y)$. For the $S$-equivariant proper \'etale correspondence $\bigsqcup_{e \in E^\times} U_e \colon E^\times \to \widehat E$, we obtain the proper \'etale correspondence $\Omega_S \colon \underlying{S} \to \universal{S}$. 

A key feature of the idempotent $pp^*$ associated to a finite path $p$ on a graph which begins at a regular vertex is that it is \textit{pseudofinite}. Following Munn, an idempotent $e$ in an inverse semigroup $S$ is pseudofinite if there is a finite set $J$ of idempotents such that $f<e$ if and only if $f\leq j$ for some $j\in J$. We may, of course, assume that the elements of $J$ are incomparable, in which case they must be the set of maximal elements $\max(e)$ below $e$. It was observed in~\cite{mygroupoidalgebra} that a principal filter $\chi_e$ is isolated in $\wh E$ if and only if $e$ is pseudofinite. 

Consider an $S$-invariant set $F$ of nonzero pseudofinite idempotents and consider the set $\til F = \{\chi_e \in \widehat E \mid e \in F \}$ of principal filters. For $e \in F$ we have $\{ \chi_e \} = U_e \setminus \bigcup_{d \in \max(e)} U_d$. If $Y \subseteq \max(e)$, we put $e_Y=\prod_{e\in Y}e$, with the convention that $e_\emptyset = e$.  Munn calls 
\[ \til e = \prod_{d \in \max(e)}(e-d) = \sum_{Y \subseteq \max(e)} (-1)^{\lvert Y \rvert} e_Y \in \Z E \]
the \emph{Rukolaine idempotent} associated to $e$, as these were first considered by Rukolaine in~\cite{Rukolaine}. For example, if $v$ is a regular vertex of a graph $E$, then for  $S_E$, $\til v = v-\sum_{\ran(e)=v}ee^*$, since the idempotents $ee^*$ in the sum are pairwise orthogonal. Under the isomorphism $\mathbb Z E \to \mathbb Z \wh E$, the image of $\til e$ is $1_{U_e\setminus\bigcup_{d\in \max(e)}U_d} = 1_{\{\chi_e\}}$, using the principle of inclusion-exclusion. Thus the Rukolaine idempotent $\til e$ expresses algebraically the element of the abelian group $\Z E$ to which $1_{\chi_e} \in \Z \til F$ is sent. 

We now mimic the Rukolaine idempotent on the level of \'etale correspondences. For a pseudofinite idempotent $e \in E$, we write 
\begin{align*}
P_+(e) & = \{ Y \subseteq \max(e) \mid e_Y \ne 0, \lvert Y \rvert \text{ even} \}, \\
P_-(e) & = \{ Y \subseteq \max(e) \mid e_Y \ne 0, \lvert Y \rvert \text{ odd} \}.
\end{align*}
Given an $S$-invariant set $F \subseteq E^\times$ of nonzero pseudofinite idempotents, we obtain correspondences $\bigsqcup_{f \in F} P_\pm(f) \colon F \to E^\times$ with anchor maps $\rho(e,Y) = e$ and $\sigma(e,Y) = e_Y$. These are $S$-equivariant topological correspondences and therefore induce \'etale correspondences
\[ R_\pm = \bigsqcup_{f \in F} P_\pm(f) \times_{E^\times} \underlying{S} \colon \underlying{S} |_F \to \underlying{S} \]
which we call the \textit{Rukolaine correspondences}. We call the resulting map in homology
\[ H_*( R_+,A) - H_*( R_-,A) \colon H_*(\underlying{S} |_F,A) \to H_*(\underlying{S},A) \]
the \textit{Rukolaine map}.

\begin{Example}[Inverse semigroups associated to self-similar groupoids]\label{ex:ssgroupoid.ruk}
For a self-similar groupoid action $(G,\mathcal X)$ we consider the set $F = \{ p p^* \in S_{\mathcal X} \mid p \in \mathcal X^+, \sour(p) \in G^0_\reg \}$ of idempotents in $S_{\mathcal X}$ associated to paths beginning  at a regular vertex. This is an invariant set of pseudofinite idempotents, and for a path $p \in \mathcal X^+$ which begins at a regular vertex $v \in G^0_\reg$ we have $P_+(pp^*) = \{\emptyset\}$ and $P_-(pp^*) = \{ \{ px x^* p^* \} \mid x \in \mathcal X, \ran(x) = \sour(p) \}$. Thus $R_+ \colon \underlying{S_{\mathcal X}} |_F \to \underlying{S_{\mathcal X}}$ is given by the inclusion and $R_- \colon \underlying{S_{\mathcal X}} |_F \to \underlying{S_{\mathcal X}}$ has bispace 
\[ \bigsqcup_{pp^* \in F} \{ pxq^* \mid x \in \mathcal X, q \in \mathcal X^+, \ran(x) = \sour(p), \sour(q) = \sour(x) \}. \]
The range map $\ran \colon R_- \to F$ picks out the index of the disjoint union and the source map sends $pxq^*$ to $qq^* \in \underlying{S_{\mathcal X}}^0$. Note that $R_+$ restricts to the inclusion of $G_\reg\hookrightarrow G$.  Also,
there is a commutative diagram up to isomorphism
\[ \begin{tikzcd}[arrow style=math font, column sep = 1.5cm]
	{\underlying{S_{\mathcal X}}|_F}  & {\underlying{S_{\mathcal X}}}\\
	  {G_{\reg}} & {G}
    \arrow["R_-", from=1-1, to=1-2]
	\arrow["{\mathcal X_{\reg}}", from=2-1, to=2-2]
	\arrow[hook, from=2-1, to=1-1]
	\arrow[hook, from=2-2, to=1-2]
\end{tikzcd} \] 
where the inclusions are Morita equivalences and  $\mathcal X_\reg \colon G_\reg \to G$ is the restriction of $\mathcal X$ to $G_\reg$.
Indeed, the left-hand composition is isomorphic to  $\bigsqcup_{v \in G^0_\reg} \{ xq^* \mid x \in \mathcal X, \ran(x) = v, \sour(q)=\sour(x)\}$ and the right-hand composition is isomorphic to $\{xq^* \mid x \in \mathcal X, \ran(x)\in G^0_{\reg}, \sour(q)=\sour(x)\}$.
\end{Example}

\begin{Prop}\label{p:ruk.prop}
Let $S$ be an inverse semigroup with idempotent semilattice $E$ and let $F \subseteq E^\times$ be an invariant set of nonzero pseudofinite idempotents. Consider the Rukolaine correspondences $R_\pm \colon \underlying{S} |_F \to \underlying{S}$ and the set $\til F = \{ \chi_e \mid e \in F \}$ of principal filters from $F$. Then for each $n \geq 0$ and abelian group $A$ the following diagram commutes.
\[\begin{tikzcd}[arrow style=math font, column sep = 2cm]
	{H_n(\universal{S} |_{\til F},A)} && {H_n(\universal{S},A)} \\
	{H_n(\underlying{S} |_F,A)} && {H_n(\underlying{S},A)}
	\arrow["{H_n(\iota,A)}", from=1-1, to=1-3]
	\arrow["\cong", from=2-1, to=1-1]
	\arrow["{H_n(R_+,A) - H_n(R_-,A)}"', from=2-1, to=2-3]
	\arrow["\cong", "{H_n(\Omega_S,A)}"', from=2-3, to=1-3]
\end{tikzcd}\]
\end{Prop}
\begin{proof}
The fact that $H_n(\Omega_S)$ is an isomorphism is~\cite[Example~3.10]{miller2023ample}.
We will use the description of the induced maps in homology from Proposition \ref{p:description of map in homology}, which give us chain maps $C_\bullet(\underlying{S} |_F) \to C_\bullet(\universal{S})$, which induce the corresponding chain maps with coefficients $A$, so it suffices to handle the case $A=\Z$. So, fix $s = (s_1,\dots,s_n) \in C_n(\underlying{S} |_F)$. The compositions $\Omega_S \circ R_\pm \colon \underlying{S} |_F \to \universal{S}$ have underlying bispaces
\[ \bigsqcup_{e \in F, Y \in P_\pm(e)} U_{e_Y}  \times_{\hat E} \universal{S}, \]
which have continuous transversals $Z_\pm = \bigsqcup_{e \in F, Y \in P_\pm(e)} U_{e_Y}$. Following \eqref{ample description of map}, $s = (s_1,\dots,s_n) \in C_n(\underlying{S} |_F)$ is sent to
\[ \sum_{(e,Y) \in Z_\pm, \sour(s_n) = e} 1_{V_{s,Y}} = \sum_{Y \in P_\pm(\sour(s_n))} 1_{V_{s,Y}} \in C_n(\universal{S}) \]
where
\[ V_{s,Y} = \{ ([s_1,x_1],\dots,[s_n,x_n]) \in (\universal{S})^n \mid x_n\in U_{e_Y}, x_{i} = s_{i+1} \cdot x_{i+1} \}. \]
Each element of $V_{s,Y}$ is uniquely specified by an arbitrary $x_n \in U_{e_Y}$. By inclusion-exclusion, 
\[ \sum_{Y \in P_+(\sour(s_n))} 1_{V_{s,Y}} - \sum_{Y \in P_-(\sour(s_n))} 1_{V_{s,Y}} = 1_{([s_1,\chi_{\sour(s_1)}],\dots,[s_n,\chi_{\sour(s_n)}])}. \]
The map $H_n(\Omega_S) \circ (H_n(R_+) - H_n(R_-))$ is therefore induced by the inclusion $\underlying{S} |_F \hookrightarrow \universal{S}$ which sends $s$ to $[s,\chi_{\sour(s)}]$.
\end{proof}

\subsection{Eilenberg--Mac Lane spaces for groups}

We recall the definition and basic properties of Eilenberg--Mac Lane spaces in the setting of a discrete group $G$. A $K(G,1)$, or Eilenberg--Mac Lane space, for $G$ is a CW complex $X$ with $\pi_1(X)\cong G$ and a contractible universal cover. In this case, $H_n(G,A)\cong H_n(X,A)$ for any abelian group $A$, where the right-hand side can be computed via cellular homology~\cite{Browncohomology}.


If $f\colon G\to H$ is a group homomorphism and $X$, $Y$ are Eilenberg--Mac Lane spaces for $G$, $H$, respectively, then there is a unique, up to homotopy, basepoint-preserving cellular map $\p\colon X\to Y$ such that $\p_*=f$.  The induced map $H(\p,A)\colon H_\bullet(X,A)\to H_\bullet(Y,A)$ agrees with $H(f,A)\colon H_\bullet(G,A)\to H_\bullet(H,A)$ under the identification of cellular homology with group homology.

One construction of a $K(G,1)$ is the classifying space $BG$, which  is the geometric realization of the simplicial set $\mathcal NG$ with $(\mathcal NG)_n = G^n$, where the face maps are as in \eqref{eq:face.maps} and the $i^{th}$-degeneracy inserts the identity at the $i^{th}$ object.    The map $B(f)\colon BG\to BH$ associated to a group homomorphism is given by the induced maps $f^n\colon G^n\to H^n$.


If $X$ is a $K(G,1)$ and $Y$ is a $K(H,1)$, then $X\times Y$ is a $K(G\times H,1)$, where $X\times Y$ is given the compactly generated topology.  The $q$-cells of $X\times Y$ are products of the form $e_i\times f_{q-i}$ where $e_i$ is an $i$-cell of $X$ and $f_{q-i}$ is a $(q-i)$-cell of $Y$.  If $K$ is any commutative ring, there is an isomorphism $C_\bullet(X\times Y,K)\cong C_\bullet(X,K)\otimes C_\bullet(Y,K)$ of cellular chain complexes, where we take the usual tensor product of chain complexes of $K$-modules with $n^{th}$-$K$-module $\bigoplus_{i+j=n}C_i(X,K)\otimes_K C_j(Y,K)$. Given orientations of $e_i$ and $f_{q-i}$, the appropriately oriented cell $e_i\times f_{q-i}$ is mapped to $e_i\otimes f_{q-i}$.

\section{Main theorems}
In this section we state the main results concerning homology and K-theory of groupoids and $\cs$-algebras associated to self-similar groupoid actions.  We then apply these tools in the remainder of the paper.

Note that if $G$ is a discrete groupoid and $F$ is an invariant subset, then $G = G|_F \sqcup G|_{G^0 \setminus F}$.
It follows that the inclusion $G|_F\hookrightarrow G$ induces an inclusion $H_\bullet(G|_F)\to H_\bullet(G)$ as a direct summand.  We can therefore write $\id\colon H_\bullet(G|_F)\to H_\bullet(G)$ to mean the inclusion, and similarly for K-theory.

\begin{Thm}\label{t:main.homology}
Let $(G,\mathcal X)$ be a self-similar groupoid action. Then, for each abelian group $A$, there is a long exact sequence in homology
\[ \cdots \to H_{n+1}(\mathscr G_{\mathcal X},A) \to H_n(G_\reg,A) \to H_n(G,A) \to H_n(\mathscr G_{\mathcal X},A) \to \cdots \]
where the middle map is $\id - H_n(\mathcal X_\reg,A)$, with $\mathcal X_\reg\colon G_\reg\to G$ the restriction of $\mathcal X$ to $G_\reg$.
\end{Thm}
\begin{proof}
Let $F=\{pp^*\mid \sour(p)\in G_{\mathrm{reg}}^0\}$ and let $\til F=\{\chi_f\mid f\in F\}$.  By Proposition ~\ref{p:Morita.equiv}, Proposition~\ref{p:ruk.prop} and Example~\ref{ex:ssgroupoid.ruk} we have a commutative diagram
\[\begin{tikzcd}[arrow style=math font, column sep = 2cm]
	{H_n(\universal{\mathcal{X}} |_{\til F},A)} && {H_n(\universal{\mathcal X},A)} \\
	{H_n(\underlying{S_\mathcal X} |_F,A)} && {H_n(\underlying{S_\mathcal X},A)}\\
    {H_n(G_\reg,A)} && {H_n(G,A)}
	\arrow["{H_n(\iota,A)}", from=1-1, to=1-3]
	\arrow["\cong", from=2-1, to=1-1]
	\arrow["{H_n(R_+,A) - H_n(R_-,A)}"', from=2-1, to=2-3]
	\arrow["\cong", "{H_n(\Omega_S,A)}"', from=2-3, to=1-3]
    \arrow["\cong", from=3-1, to=2-1]
    \arrow["\cong", from=3-3, to=2-3]
    \arrow["{\id-H_n(\mathcal X_\reg,A)}", from=3-1, to=3-3]
\end{tikzcd}\]
The result now follows from the long exact sequence of Proposition~\ref{t:tool1} applied to the exact sequence of groupoids $\universal{\mathcal X}|_{\til F}\hookrightarrow \universal{\mathcal X}\twoheadrightarrow \mathscr G_{\mathcal X}$.
\end{proof}

Applying  Propositions~\ref{p:description of map in homology} and~\ref{p:transversal decomposition} to $H_n(\mathcal X_\reg)$ in Theorem~\ref{t:main.homology}, yields the following `groups only' description of the long exact sequence.

\begin{Cor}\label{c:homology.transversal}
Let $(G,E,\sigma)$ be a self-similar groupoid action on a graph $E$ with cocycle $\sigma$. Let $T^0 \subseteq G^0$ be a transversal for $G$ and set $T^0_\reg = T^0 \cap E^0_\reg$. Then, for each abelian group $A$, there is a long exact sequence in homology 
\[\begin{tikzcd}[arrow style=math font,cells={nodes={text height=2ex,text depth=0.75ex}}, column sep = 1cm]
	{\cdots} &  {H_{n+1}(\mathscr G_{(G,E)},A)} \arrow[draw=none]{d}[name=Y, shape=coordinate]{}  & {\bigoplus\limits_{v \in T^0_\reg} H_n(G^v_v,A)} \\
	{\bigoplus\limits_{w \in T^0} H_n(G^w_w,A)} &  {H_n(\mathscr G_{(G,E)},A)} & {\cdots}  
	\arrow[from=1-1, to=1-2]
        \arrow[from=1-2, to=1-3]
	\arrow[from=1-3, to=2-1, labcurarrow2=Y]
	\arrow[from=2-1, to=2-2]
	\arrow[from=2-2, to=2-3]
\end{tikzcd}\]
where $\Phi_n$ admits the following description. Fix, for each $v\in T^0_\reg$, a left $G^v_v$-transversal $T_v$ to $vE$. Consider, for $e \in T_v$, the virtual homomorphism $\sigma_e \colon G_e \to G^{\sour(e)}_{\sour(e)}$, $\sigma_e(g)=g|_e$. 
For each $w \in \sour(\bigsqcup_{v\in T^0_\reg}T_v)$,  pick $h_w \in G$ with $\sour(h_w) = w$ and $\ran(h_w) = t(w) \in T^0$ and set $c_w \colon G^w_w \to G^{t(w)}_{t(w)}$ to be the homomorphism $g \mapsto h_w\inv g h_w$. 
Then \[ \Phi_n = \bigoplus_{v\in T^0_\reg}\sum_{e \in T_v} H_n(c_{\sour(e)},A)\circ H_n(\sigma_e,A)\circ \mathrm{tr}^{G_v^v}_{G_e},\] and it is induced by the chain map
\begin{align*}
\bigoplus_{v \in T^0_\reg} C_\bullet(G^v_v) & \to \bigoplus_{w \in T^0} C_\bullet(G^w_w) \\
(g_1,\dots,g_m) & \mapsto \sum_{\ran(e)=\sour(g_m)} (c_{\sour(e)}(g_1|_{g_2\cdots g_m(e)}),\ldots, c_{\sour(e)} (g_m|_e)).
\end{align*}
In particular, $(\Phi_0)_{w,v} =|\ran\inv(v)\cap \sour\inv(Gw)| \colon H_0(G^v_v,A) \to H_0(G^w_w,A)$.
\end{Cor}

\begin{Rmk}\label{r:graph transversal remark}
Viewing a self-similar groupoid action on a graph $(G,E,\sigma)$ as a choice of transversal $E \subseteq \mathcal X$ for the right action in a self-similar groupoid $(G, \mathcal X)$, the above maps $\Phi_n$ are independent of the choice of transversal. Note also that is always possible to pick $E \subseteq \mathcal X$ such that $\sour(E) \subseteq T^0$, in which case the conjugation homomorphisms $c_w \colon G^w_w \to G^{t(w)}_{t(w)}$ are unnecessary and may be chosen to be trivial.
\end{Rmk}

In most of our applications we are in the following situation, where the statement becomes considerably simpler.

\begin{Cor}\label{c:transitive.transfer}
Let $(G,E,\sigma)$ be a self-similar group action on a finite alphabet $E$ of cardinality at least $2$ with cocycle $\sigma$. For $e\in E$ and let $\sigma_e\colon G_e\to G$ be the virtual endomorphism $\sigma_e(g)=g|_e$. Then there is a long exact sequence
\[\begin{tikzcd}[arrow style=math font,cells={nodes={text height=2ex,text depth=0.75ex}}, column sep = 1cm]
	{\cdots} &  {H_{n+1}(\mathscr G_{(G,E)},A)} \arrow[draw=none]{d}[name=Y, shape=coordinate]{}  & {H_n(G,A)} \\
	{H_n(G,A)} &  {H_n(\mathscr G_{(G,E)},A)} & {\cdots}  
	\arrow[from=1-1, to=1-2]
        \arrow[from=1-2, to=1-3]
	\arrow[from=1-3, to=2-1, labcurarrow2=Y]
	\arrow[from=2-1, to=2-2]
	\arrow[from=2-2, to=2-3]
\end{tikzcd}\]
where $\Phi_n = \sum_{e \in T} H_n(\sigma_e,A)\circ \mathrm{tr}^{G}_{G_e}$ for any $G$-transversal $T \subseteq E$ and is induced by the chain map 
\begin{align*}
C_\bullet(G) & \to  C_\bullet(G) \\
(g_1,\dots,g_m) & \mapsto \sum_{e \in E} (g_1|_{g_2\cdots g_m(e)},\ldots, g_m| _e ).
\end{align*}
In particular, $H_0(\mathscr G_{(G,E)})\cong \mathbb Z/(|E|-1)\mathbb Z$ and $H_1(\mathscr G_{(G,E)}))\cong \coker (\id-\Phi_1)$ where $\Phi_1\colon \ab{G}\to \ab{G}$ is given by $\Phi_1(g[G,G]) = \sum_{e\in E} g|_e[G,G]$. 
\end{Cor}
\begin{proof}
Corollary~\ref{c:homology.transversal} provides everything except the `in particular' statement.   The long exact sequence yields the exact sequence
\[\ab G\to \ab G\xrightarrow{\id -\Phi_1} H_1(\mathscr G_{(G,E)})\to \mathbb Z\xrightarrow{(1-|E|)\id} \mathbb Z\to H_0(\mathscr G_{(G,E))})\to 0.\]  Since $(1-|E|)\id$ is injective because $|E|>1$, we see that $H_0(\mathscr G_{(G,E)})\cong \mathbb Z/(|E|-1)\mathbb Z$ and $H_1(\mathscr G_{(G,E)})\cong \coker (\id -\Phi_1)$.
\end{proof}

The `in particular' statement of Corollary~\ref{c:transitive.transfer} can also be found in~\cite[Theorem~4.3.21]{Nekbook2}.

Proposition~\ref{p:transversal decomposition} also yields a `groups only' description of the six-term sequence in K-theory \eqref{Katsura six term}:

\begin{Thm}\label{t:main.k}
Let $(G,E,\sigma)$ be a self-similar groupoid action on a graph $E$ with cocycle $\sigma$. Let $T^0 \subseteq G^0$ be a transversal for $G$ and set $T^0_\reg = T^0 \cap E^0_\reg$. Then there is a six-term sequence in K-theory
\[\begin{tikzcd}[arrow style=math font]
	{\bigoplus\limits_{v \in T^0_\reg} K_0(\cs(G^v_v))} & {\bigoplus\limits_{w \in T^0} K_0(\cs(G^w_w))} & {K_0(\mathcal O_{\mathcal X})} \\
	{K_1(\mathcal O_{\mathcal X})} & {\bigoplus\limits_{w \in T^0} K_1(\cs(G^w_w))} & {\bigoplus\limits_{v \in T^0_\reg} K_1(\cs(G^v_v))}
	\arrow["{1-\Phi_0}", from=1-1, to=1-2]
	\arrow[from=1-2, to=1-3]
	\arrow[from=1-3, to=2-3]
	\arrow[from=2-1, to=1-1]
	\arrow[from=2-2, to=2-1]
	\arrow["{1-\Phi_1}", from=2-3, to=2-2]
\end{tikzcd}\]
where $\Phi_i$ admits the following description for $i=0,1$. Fix, for each $v\in T^0_\reg$, a left $G^v_v$-transversal $T_v$ to $vE$. Consider for $e \in T_v$ the virtual homomorphism $\sigma_e \colon G_e \to G^{\sour(e)}_{\sour(e)}$, $\sigma_e(g)=g|_e$. 
For each $w \in \sour(\bigsqcup_{v\in T^0_\reg}T_v)$,  pick $h_w \in G$ with $\sour(h_w) = w$ and $\ran(h_w) = t(w) \in T^0$ and set $c_w \colon G^w_w \to G^{t(w)}_{t(w)}$ to be the homomorphism $g \mapsto h_w\inv g h_w$. 
Then \[ \Phi_i = \bigoplus_{v\in T^0_\reg}\sum_{e \in T_v} K_i(c_{\sour(e)}) \circ K_i(\sigma_e)\circ \mathrm{tr}^{G}_{G_e}. \]
\end{Thm}

As in Remark \ref{r:graph transversal remark}, the maps $\Phi_i$ are independent of the graph action presentation of the underlying self-similar groupoid action.

\begin{Cor}\label{c:transitive.transfer.Ktheory}
Let $(G,E,\sigma)$ be a self-similar group action on a finite alphabet $E$ with cocycle $\sigma$.  For $e\in E$ let $\sigma_e\colon G_e\to G$ be the virtual endomorphism $\sigma_e(g)=g|_e$.  Then there is a long exact sequence
\[\begin{tikzcd}[arrow style=math font]
	{K_0(\cs(G))} & {K_0(\cs(G))} & {K_0(\mathcal O_{\mathcal X})} \\
	{K_1(\mathcal O_{\mathcal X})} & {K_1(\cs(G))} & {K_1(\cs(G))}
	\arrow["{1-\Phi_0}", from=1-1, to=1-2]
	\arrow[from=1-2, to=1-3]
	\arrow[from=1-3, to=2-3]
	\arrow[from=2-1, to=1-1]
	\arrow[from=2-2, to=2-1]
	\arrow["{1-\Phi_1}", from=2-3, to=2-2]
\end{tikzcd}\]
where $\Phi_i = \sum_{e \in T} K_i(\sigma_e)\circ \mathrm{tr}^{G}_{G_e}$ for $i = 0,1$ and any $G$-transversal $T \subseteq E$.
\end{Cor}

\section{Computations: miscellaneous examples}
Throughout the computation sections we shall frequently need the well-known computation of the homology of finite cyclic groups; see~\cite[Page~35]{Browncohomology}.
\begin{equation}\label{eq:cyclic.homology}
    H_n(\mathbb Z/m\mathbb Z)= \begin{cases} \mathbb Z, & \text{if}\ n=0, \\ 0, & \text{if}\ n\in 2\mathbb Z,\\ \mathbb Z/m\mathbb Z, & \text{if}\ n\in 2\mathbb Z+1.\end{cases}
\end{equation}

\subsection{Graphs}\label{subs:graph}
We perform here the computations for Example~\ref{Ex:Graphs}.
Matui~\cite{Matui} computed the homology of a graph groupoid for a finite graph.  This was extended to arbitrary graphs by Nyland and Ortega~\cite{OrtegaNylandgraph}.  We handle the case of an arbitrary graph using our methods, giving an easier proof.
Of course, the K-theory of graph $\cs$-algebras is well known.

\begin{Thm}\label{t:graph.case}
Let $E$ be an arbitrary graph.  Let $A$ be the $E^0_{\mathrm{reg}}\times E^0$-matrix with $A_{v,w}$ the number of edges from $w$ to $v$.  Then $H_0(\mathscr G_E)\cong \coker (\id -A^T)$, $H_1(\mathscr G_E)\cong \ker (\id -A^T)$ and $H_n(\mathscr G_E)=0$ for $n\geq 2$.  Moreover, $K_0(\cs(E))\cong  \coker (\id -A^T)$ and $K_1(\cs(E))\cong \ker(\id -A^T)$.
\end{Thm}
\begin{proof}
In this case, the groupoid $G=G^0=E^0$ has trivial isotropy, $\cs(G)=C_0(E^0)$ and $\cs(G_{\mathrm{reg}})= C_0(E^0_{\mathrm{reg}})$.   Therefore, by Corollary~\ref{c:homology.transversal} and Theorem~\ref{t:main.k} we have $H_q(\mathscr G_E)=0$ for, $q\geq 2$,  and exact sequences 
\[
\begin{tikzcd}[arrow style=math font, column sep = 0.5cm]
0\ar{r} & H_1(\mathscr G_E)\ar{r} & \bigoplus\limits_{v\in E_{\mathrm{reg}}^0}\mathbb Z\ar{rr}{\id -\Phi_0} & {} & \bigoplus\limits_{w\in E^0}\mathbb Z\ar{r}& H_0(\mathscr G_E)\ar{r} & 0 \\ 0\ar{r} & K_1(\cs(E))\ar{r}  &\bigoplus\limits_{v\in  E_{\mathrm{reg}}^0}\mathbb Z\ar{rr}{\id -\Phi_0} & {} & \bigoplus\limits_{w\in E^0}\mathbb Z\ar{r} & K_0(\cs(E))\ar{r}& 0
\end{tikzcd}
\]
where $(\Phi_0)_{w,v} = |vEw|=A_{v,w}$.  The result follows.
\end{proof}

\subsection{Exel--Pardo--Katsura algebras}\label{subs:EPK}
We generalize the result of Nyland and Ortega~\cite{OrtegaNylandKatsura} on homology of groupoids associated to Katsura algebras considered in Example~\ref{ex:EPK}.   Our results allow arbitrary cardinality row finite graphs and sources, while previous results stuck to countable row finite graphs and no sources.   

\begin{Lemma}\label{l:cyclic.transfer}
Let $G$ be an infinite cyclic group with generator $a \in G$ and let $H$ be a group. Let $\mathcal X\colon G\to H$ be a proper \'etale correspondence and $E$ a right $H$-transversal for $\mathcal X$. 
The maps $K_i(\mathcal X)\colon K_i(\cs(G))\to K_i(\cs(H))$  are given by $K_0(\mathcal X)([1]_0) = |E|\cdot [1]_0$ and  $K_1(\mathcal X)([u_a]_1)=\sum_{e\in E}[u_{a|_e}]_1$.
\end{Lemma}
\begin{proof}
Note that $K_0(C^*(G))\cong \mathbb Z$ with generator $[1]_0$ and  $K_1(\cs(G))\cong \mathbb Z$ with generator $[u_a]_1$.  
Without loss of generality, we may assume $\mathcal X=E\times H$ with the left action $g(e,h) = (g(e),g|_eh)$.  Let $T$ be a transversal to $G\backslash E$. Then $K_i(\mathcal X)=\sum_{t\in T}K_i(\sigma_t)\circ \tr^G_{G_t}$ by Proposition~\ref{p:transversal decomposition}. Then $K_0(\mathcal X)([1]_0) = \sum_{t\in T}[G:G_t][1]_0 = |E|\cdot [1]_0$ by Proposition~\ref{p:transfer.monomial.form}. 
Fix $t\in T$.  Supposing that $[G:G_t]=m_t$, we have $G_t=\langle a^{m_t}\rangle$.  We then compute that $\sigma_t(a^{m_t}) = a|_{a^{m_t-1}(t)}\cdots a|_{a(t)}\cdot a|_t=\prod_{e\in Gt}a|_e$. 
Choosing $1,a,\ldots, a^{m_t-1}$ as our transversal to $G/G_t$,  the map $\psi_t\colon \cs(G)\to M_{m_t}(\cs(G_t))$ induced by transfer sends $u_a$ to the matrix $\mathrm{diag}(1,1,\ldots, 1,u_{a^{m_t}})P$ where $P\in M_{m_t}(\mathbb C)$ is the $m_t \times m_t$-permutation matrix obtained by cyclically permuting the columns of the identity matrix to the left. It follows that $\mathrm{tr}^G_{G_t} ([u_a]_1)= [u_{a^{m_t}}]_1$.  Therefore,  $K_1(\mathcal X)([u_a]_1) = \sum_{t\in T}K_1(\sigma_t)([u_{a^{m_t}}]_1) = \sum_{e\in E}[u_{a|_e}]_1$. 
\end{proof}

\begin{Cor}\label{c:Katsura}
Let $A,B$ be integer matrices over some index set $J$ with $A$ the adjacency matrix of a row finite graph such that $A_{ij}=0$ implies $B_{ij}=0$.  Let $J'\subseteq J$ be the set of indices of zero rows. Let $A',B'$ be the matrices obtained from $A,B$, respectively, by removing the rows corresponding to indices in $J'$. Then we have:
\begin{enumerate}
  \item $H_0(\mathscr G_{A,B})\cong \coker (\id -(A')^T)$;
  \item $H_1(\mathscr G_{A,B})\cong \ker(\id -(A')^T)\oplus \coker (\id -(B')^T)$;
  \item $H_2(\mathscr G_{A,B})\cong \ker (\id -(B')^T)$;
\end{enumerate}
and $H_n(\mathscr G_{A,B})=0$ for $n\geq 3$.  Moreover, $K_0(\mathcal O_{A,B})\cong \coker (\id -(A')^T)\oplus \ker (\id -(B')^T)$ and $K_1(\mathcal O_{A,B)}) \cong \ker(\id -(A')^T)\oplus \coker (\id -(B')^T)$.
\end{Cor}
\begin{proof}
Recall that $G=\mathbb Z\times J$ and $G_{\mathrm{reg}}= \mathbb Z\times J\setminus J'$.
It follows from Corollary~\ref{c:homology.transversal}, Theorem~\ref{t:main.k} and the fact that $H_q(\mathbb Z)=0$ for $q\geq 2$ and $K_0(\cs(\mathbb Z))\cong \mathbb Z\cong K_1(\cs(\mathbb Z))$, that we have exact sequences 
\[\begin{tikzcd}[arrow style=math font,cells={nodes={text height=2ex,text depth=0.75ex}}, column sep = 1cm]
	{0} &  {H_2(\mathscr G_{A,B})} \arrow[draw=none]{d}[name=Y, shape=coordinate]{}  & {\bigoplus\limits_{j\in J\setminus  J'}\mathbb Z} \\
	{\bigoplus\limits_{j\in J}\mathbb Z} &  {H_1(\mathscr G_{A,B})} \arrow[draw=none]{d}[name=Z, shape=coordinate]{} & {\bigoplus\limits_{j\in J\setminus J'}\mathbb Z}  \\
 	{\bigoplus\limits_{j\in J}\mathbb Z} &  {H_0(\mathscr G_{A,B})} & {0} 
	\arrow[from=1-1, to=1-2]
        \arrow[from=1-2, to=1-3]
	\arrow[from=1-3, to=2-1, labcurarrow3=Y]
	\arrow[from=2-1, to=2-2]
	\arrow[from=2-2, to=2-3]
        \arrow[from=2-3, to=3-1, labcurarrow4=Z]
	\arrow[from=3-1, to=3-2]
	\arrow[from=3-2, to=3-3]
\end{tikzcd}\]
and 

\[\begin{tikzcd}[arrow style=math font]
	{\bigoplus\limits_{j\in J\setminus J'} \mathbb Z} & {\bigoplus\limits_{j\in J} \mathbb Z} & {K_0(\mathcal O_{A,B})}\\
	{K_1(\mathcal O_{A,B})} & {\bigoplus\limits_{j\in J} \mathbb Z} & {\bigoplus\limits_{j\in J\setminus J'} \mathbb Z}
	\arrow["{1-\Phi_0}", from=1-1, to=1-2]
	\arrow[from=1-2, to=1-3]
	\arrow[from=1-3, to=2-3]
	\arrow[from=2-1, to=1-1]
	\arrow[from=2-2, to=2-1]
	\arrow["{1-\Phi_1}", from=2-3, to=2-2]
\end{tikzcd}\]

Note that for both homology and K-theory we have by Corollary~\ref{c:homology.transversal} and by Lemma~\ref{l:cyclic.transfer},
 $(\Phi_0)_{i,j} = |\ran\inv(j))\cap \sour\inv(i)|=A'_{ji}$,  whereas  $(\Phi_1)_{i,j}=0=B'_{ji}$ if $A_{ji}=0$, and  otherwise, for $j\in J\setminus J'$, we have  $(\Phi_1)_{i,j} =  \sum_{n=0}^{A_{ji}-1} (1,j)|_{e_{j,i,\ov n}}$.

 Notice that \[\sum_{n=0}^{A_{ji}-1}(B_{ji}+n)= \sum_{n=0}^{A_{j,i}-1} ((1,j)|_{e_{j,i,\ov n}}A_{ji}+(1,j)(e_{j,i,\ov n}))\] where we identify $\mathbb Z/A_{j,i}\mathbb Z$ with $\{0,\ldots, A_{ji}-1\}$ when convenient.  But $\sum_{n=0}^{A_{ji}-1} n = \sum_{n=0}^{A_{ji}-1}(1,j)(e_{j,i,\ov n})$ since the right hand sum is a reordering of the left hand sum.  Thus $A_{ji}B_{ji} = A_{ji}\sum_{n=0}^{A_{j,i}-1}(1,j)|_{e_{j,i,\ov n}}$.  Since $A_{ji}>0$, we have $(\Phi_1)_{i,j} = B'_{ji}$.  The result now follows for homology and for $K$-theory, as well, upon noting that the images of $H_1(\mathscr G_{A,B})$ and $K_i(\mathcal O_{A,B})$ must be free abelian, and so the short exact sequences extracted from the above exact sequences must split.
\end{proof}

The groupoids $\mathscr G = \mathscr G_{A,B}$ enjoy the HK property.

\subsection{The Aleshin automaton}\label{subs:Aleshin}

The rest of our computations are of the homology and K-theory of groupoids and algebras associated to self-similar groups. Consider the Aleshin automaton for a self-similar action of the free group ${F_3}$ in Example~\ref{ex:aleshin}.

\begin{Thm}
Let $\mathscr G$ be the groupoid associated to the Aleshin automaton.  Then $H_1(\mathscr G)\cong \mathbb Z/2\mathbb Z$ and $H_n(\mathscr G)=0$ for $n\neq 1$. Moreover, $K_0(\cs(\mathscr G))=0$ and $K_1(\cs(\mathscr G))\cong \mathbb Z/2 \Z$.
\end{Thm}
\begin{proof}
By Corollary~\ref{c:transitive.transfer}, and since $H_q({F_3})=0$ for $q\geq 2$, we obtain that $H_q(\mathscr G)=0$ for $q\neq 1,2$ and  an exact sequence $0\to H_2(\mathscr G)\to \ab{F_3}\xrightarrow{\id - \Phi_1}\ab{F_3}\to H_1(\mathscr G)\to 0$ where $\Phi_1(g[{F_3},{F_3}]) = g|_0[{F_3},{F_3}]+g|_1[{F_3},{F_3}]$.  Writing $\ov g$ for $g[{F_3},{F_3}]$, we have $\Phi_1(\ov a) = \ov c+\ov b$, $\Phi_1(\ov b) = \ov b+\ov c$ and $\Phi_1(\ov c) = 2\ov a$.  Thus $\id-\Phi_1$ is given by the matrix \[A= \begin{bmatrix} 1 & 0 & -2\\ -1 & 0 & 0\\ -1  & -1 & 1\end{bmatrix}\] which has determinant $-2$.  Thus $H_1(\mathscr G)\cong \coker (\id-\Phi_1)\cong \mathbb Z/2\mathbb Z$ and $H_2(\mathscr G)\cong \ker (\id -\Phi_1)=0$. 

It is well known, cf.~\cite{Cuntz_K_amenability}, that $K_0(\cs({F_3}))\cong\mathbb Z$ generated by $[1]_0$ and $K_1(\cs({F_3}))\cong \mathbb Z^3$ with basis $[u_a]_1,[u_b]_1,[u_c]_1$.  Let $x\in \{a,b,c\}$, and let $i_x\colon \langle x\rangle\to {F_3}$ be the inclusion. The correspondence of the self-similar action is $\mathcal X=\{0,1\}\times {F_3}$, and if we compose this with the correspondence $I_x$  with bispace ${F_3}$ corresponding to $i_x$, we obtain $\mathcal X\circ I_x={F_3}\times_{{F_3}} \mathcal X\cong {}_{\langle x\rangle}\mathcal X$, which is the bispace $\mathcal X$ with left action restricted to $\langle x\rangle$. It follows that $K_0(\mathcal X)([1]_0) = K_0(\mathcal X)\circ K_0(i_x)([1]_0) = K_0({}_{\langle x\rangle}\mathcal X)([1]_0) = 2[1]_0$  and $K_1(\mathcal X)([u_x]_1) =  K_1(\mathcal X)\circ K_1(i_x)([u_x]_1) = K_1({}_{\langle x\rangle}\mathcal X)([u_x]_1) = [u_{x|_0}]_1+[u_{x|_1}]_1$ by Lemma~\ref{l:cyclic.transfer}.  We therefore have an exact sequence 
\[\begin{tikzcd}[arrow style=math font]
	{\mathbb Z} & {\mathbb Z} & {K_0(\cs(\mathscr G))}\\
	{K_1(\cs(\mathscr G))} & {\mathbb Z^3} & {\mathbb Z^3}
	\arrow["{-\id}", from=1-1, to=1-2]
	\arrow[from=1-2, to=1-3]
	\arrow[from=1-3, to=2-3]
	\arrow[from=2-1, to=1-1]
	\arrow[from=2-2, to=2-1]
	\arrow["{A}", from=2-3, to=2-2]
\end{tikzcd}\]
where we retain the previous notation.  It follows that $K_0(\cs(\mathscr G))\cong \ker A=0$ and $K_1(\cs(\mathscr G))\cong \coker A\cong \mathbb Z/2\mathbb Z$.
\end{proof}

\subsection{The Hanoi towers group}\label{subs:hanoi}
We compute here the homology and K-theory for the Hanoi towers group $H$ from Example~\ref{ex:hanoi}. The associated groupoid $\mathscr G$ is minimal, effective, amenable and Hausdorff.

\begin{Thm}\label{t:Hanoi}
Let $\mathscr G$ be the ample groupoid associated to the Hanoi tower group $H$.  Then
\[H_n(\mathscr G) = \begin{cases}\mathbb Z/2\mathbb Z, & \text{if}\ n=0,\\ (\mathbb Z/2\mathbb Z)^3, & \text{if}\ n\geq 1,\end{cases}\] 
and $K_0(\cs(\mathscr G))\cong \mathbb Z^3\cong K_1(\cs(\mathscr G))$, with $[1]_0 = 0 \in K_0(\cs(\mathscr G))$.  
\end{Thm}
\begin{proof}
We may work with $G=A\ast B\ast C$ instead of $H$ by Theorem~\ref{t:nek.cond}.
There is the following commutative diagram of correspondences
\[\begin{tikzcd}[arrow style=math font, column sep = 1.2cm]
	{A\sqcup B\sqcup C} & {} & {A\sqcup B\sqcup C\sqcup 1\sqcup 1\sqcup 1} & {A \sqcup B\sqcup C} \\
	{A \ast B\ast C} & {} & {(A \ast B\ast C)_0} & {A \ast B\ast C}
	\arrow["{\id \sqcup \mathrm{tr}^A_1\sqcup \mathrm{tr}^B_1\sqcup \mathrm{tr}^C_1}", from=1-1, to=1-3]
	\arrow[from=1-1, to=2-1]
	\arrow["{\id\sqcup \iota}", from=1-3, to=1-4]
	\arrow["{c_b\sqcup c_a\sqcup \id_C\sqcup \lambda}", from=1-3, to=2-3]
	\arrow[from=1-4, to=2-4]
	\arrow["\tr^{A\ast B\ast C}_{(A\ast B\ast C)_0}", from=2-1, to=2-3]
	\arrow["\sigma_0", from=2-3, to=2-4]
\end{tikzcd}\]
up to isomorphism where the downward maps from $A\sqcup B\sqcup C$ separately include $A,B,C$, $\iota \colon 1\sqcup 1\sqcup 1\to A\sqcup B\sqcup C$ is the inclusion, $c_b \colon A \to (A \ast B\ast C)_0$ is the conjugation map $a \mapsto b\inv ab$, $c_a\colon B\to (A\ast B\ast C)_0$ is $b\mapsto a\inv ba$, $\lambda$ maps all three identities of $1\sqcup 1\sqcup 1$ to the identity and the downward maps and right-hand square should be viewed as the correspondences associated to the displayed groupoid homomorphisms.     

The commutativity of the right-hand square is immediate from the definition of the cocycle $\sigma$ and the action as $\sigma_0(b\inv ab)=a$, $\sigma_0(a\inv ba)=b$ and $\sigma_0(c)=c$.   The commutativity of the left-hand square follows  by Proposition~\ref{p:mackey}.  Indeed,    $A\backslash G/G_0= \{AG_0, AbG_0\}$,  $A\cap G_0=\{1\}$ and $A\cap bG_0b\inv = A$, yielding $\tr^{G}_{G_0}\circ I_A = \iota_1\circ \tr^A_1\sqcup c_b$, where $\iota_1$ includes $1$, as $\tr^A_A$ is the identity correspondence on $A$.  Similarly, $B\backslash G/G_0= \{BG_0, BaG_0\}$,   $B\cap G_0=\{1\}$ and $B\cap aG_0a\inv = B$ implies $\tr^{G}_{G_0}\circ I_B = \iota_1\circ \tr^B_1\sqcup c_a$.  Finally,  $C\backslash G/G_0=\{CG_0, CaG_0\}$, $C\cap G_0=C$ and $C\cap aG_0a\inv=\{1\}$, whence $\tr^{G}_{G_0}\circ I_C = \id_C\sqcup \iota_1\circ\tr^C_1$.

By Corollary~\ref{c:transitive.transfer}, we have that $H_0(\mathscr G)\cong \mathbb Z/2\mathbb Z$.   Recall that, for $n\geq 1$, the  Mayer--Vietoris sequence~\cite[Corollary~7.7]{Browncohomology} yields that the inclusions of $A,B,C$ induce an isomorphism $H_n(A)\oplus H_n(B)\oplus H_n(C)\to H_n(A\ast B\ast C)$ for $n\geq 1$.  Therefore, the left- and right-most downward maps in the diagram induce isomorphisms on homology.  By commutativity of the diagram, we deduce that $H_n(\sigma_0)\circ \tr^{G}_{G_0} = \id$ for $n\geq 1$ as the homology of the trivial group is $0$ for $n\geq 1$, whence $\id -H_n(\sigma_0)\circ \tr^{G}_{G_0}=0$ for $n\geq 1$.  The long exact sequence in Corollary~\ref{c:transitive.transfer} and \eqref{eq:cyclic.homology} then imply that $H_n(\mathscr G)\cong (\mathbb Z/2\mathbb Z)^3$ for $n\geq 1$.

Next we turn to K-theory.  By a theorem of Cuntz~\cite[Page~192]{Cuntz_K_amenability}, the inclusions $A,B,C \to A \ast B\ast C$ induce an isomorphism $K_1(\cs(A\ast B\ast C)) =K_1(\cs(A))\oplus K_1(\cs(B))\oplus K_1(\cs(C))=0$ (as $A,B,C$ are finite) and an isomorphism of $K_0(\cs(A\ast B\ast C))$ with the quotient of $K_0(\cs(A))\oplus K_0(\cs(B))\oplus K_0(\cs(C))$ that identifies the classes of the unit in each of the three algebras.  Thus $K_0(\cs(A\ast B\ast C))\cong \mathbb Z^4$ with basis $[1]_0$, $[p_a]_0$, $[p_b]_0$, $[p_c]_0$ where $p_x =\frac{1}{2}(1-u_x)$ for $x=a,b,c$. The transfer to the trivial group takes $p_x$ to a rank $1$ projection matrix and hence to $[1]_0$.  Tracing the commutative diagram across the top and using Cuntz's theorem yields \[\id -K_0(\sigma_0)\circ \tr^{G}_{G_0} = \begin{bmatrix}-2 & -1 &-1&-1\\ 0 & 0&0 &0\\0 & 0 & 0 &0\\ 0&0&0&0 \end{bmatrix}\]
and so $K_0(\cs(\mathscr G))\cong \mathbb Z^3\cong K_1(\cs(\mathscr G))$ by Corollary~\ref{c:transitive.transfer.Ktheory}. Note that $[1]_0$ is in the image of $\id -K_0(\sigma_0)\circ \tr^{G}_{G_0}$ and therefore the class of the unit in $K_0(\cs(\mathscr G))$ vanishes.
\end{proof}

\section{Computations: multispinal self-similar groups}\label{s:multispinal}

We now consider multispinal groups such as the Grigorchuk group. The reader is referred to Section~\ref{ss:multispinal} for notation. As in the proof of Theorem \ref{t:Hanoi}, both the homology and K-theory groups associated to the free product $A \ast B$ of groups can be expressed in terms of those of $A$ and $B$.

\begin{Lemma}\label{l:multispinal}
Let $(A,B,\Phi)$ be the data defining a multispinal group. Let $C$ be any abelian group and let $F_*$ be $K_*(\cs(-))$ or $H_*(-,C)$. Then, for each $n \geq 0$, the diagram
\[\begin{tikzcd}[arrow style=math font]
	{F_n(A) \oplus F_n(B)} && {F_n(A) \oplus F_n(B)} \\
	{F_n(A \ast B)} && {F_n(A \ast B)}
	\arrow["M", from=1-1, to=1-3]
	\arrow[from=1-1, to=2-1]
	\arrow[from=1-3, to=2-3]
	\arrow["{F_n(\sigma_1) \circ \tr^{A \ast B}_{(A \ast B)_1}}"', from=2-1, to=2-3]
\end{tikzcd}\]
commutes, where $1 \in A$ is the unit and
\[ M = {\begin{array}{c} \begin{bmatrix} F_n(\iota_A) \circ \tr^A_1 & \sum_{a\in A_1} F_n(\Phi_a)\\ 0 & \sum_{a\in A_0} F_n(\Phi_a)\end{bmatrix} \end{array}} \in \End(F_n(A) \oplus F_n(B)) \]
with $\iota_A \colon 1 \to A$ the inclusion.
\end{Lemma}
\begin{proof}
We view the downwards maps as induced by the groupoid homomorphism $ A\sqcup B\to A\ast B$ which separately includes $A$ and $B$ into $A \ast B$. The result is implied by commutativity of the diagram of correspondences
\[\begin{tikzcd}[arrow style=math font, column sep = 2cm]
	{A\sqcup B} & {1 \sqcup \bigsqcup_{a \in A} B} & {A \sqcup B} \\
	{A \ast B} & {(A \ast B)_1} & {A \ast B}
	\arrow["{\mathrm{tr}^A_1\sqcup \bigsqcup_{a \in A} B}", from=1-1, to=1-2]
	\arrow[from=1-1, to=2-1]
	\arrow["{\iota_A\sqcup \bigsqcup_{a \in A} \Phi_a }", from=1-2, to=1-3]
	\arrow["{\iota \sqcup \bigsqcup_{a \in A} c_a}", from=1-2, to=2-2]
	\arrow[from=1-3, to=2-3]
	\arrow["\tr^{A\ast B}_{(A\ast B)_1}", from=2-1, to=2-2]
	\arrow["\sigma_1", from=2-2, to=2-3]
\end{tikzcd}\]
up to isomorphism where $\iota \colon 1 \to (A \ast B)_1$ is the inclusion, $c_a \colon B \to (A \ast B)_1$ is the conjugation map $b \mapsto a \inv b a$, 
and the downward maps and right-hand square should be viewed as the correspondences associated to the displayed groupoid homomorphisms. The right-hand square commutes as $\sigma_1(a\inv ba)=\Phi_a(b)$. The left-hand square commutes by Proposition~\ref{p:mackey} as $(A\ast B)_1$ is normal, $A\ast B=A(A\ast B)_1$ and $A\cap (A\ast B)_1=\{1\}$ and $B\backslash (A\ast B)/(A\ast B)_1= \{Ba(A\ast B)_1 \mid a \in A \}$ and $B\cap a(A\ast B)_1a\inv =B$ for all $a\in A$, noting $\tr^B_B=B$ as a bispace.
\end{proof}

Note that that $\tr^A_1$ is $0$ on $K_1$ and  $H_n$ for $n\geq 1$.  On $H_0$ it is multiplication by $|A|$ and on  $K_0$ it sends $[1]_0$ to $|A|[1]_0$. 

\begin{Thm}\label{t:multispinal}
Let $G=G_{(A,B)}$ be a multispinal group coming from the data $(A,B,\Phi)$.
Let $\mathscr G=\mathscr G_{(G,A)}$.  Let $C$ be an abelian group without $(|A|-1)$-torsion.  Then $H_0(\mathscr G,C)=C/(|A|-1)C$ and there is a long exact sequence $\cdots \to H_{n+1}(\mathscr G,C)\to  H_n(B,C)\to H_n(B,C)\to H_n(\mathscr G,C) \to \cdots \to H_1(\mathscr G,C)\to 0$ where the map $H_n(B,C)\to H_n(B,C)$ is given by 
$\id -\sum_{a\in A_0}H_n(\Phi_a,C)$. In particular, $H_n(\mathscr G)$ is a finite group for all $n\geq 0$. 
\end{Thm}
\begin{proof}
By Corollary~\ref{c:contracting.case}, we may identify $\mathscr G$ with $\mathscr G_{(A\ast B,A)}$, and we do so from now on.  
We use Corollary~\ref{c:transitive.transfer}, which in particular implies the result for $H_0(\mathscr G,C)$. By the Mayer--Vietoris sequence~\cite[Corollary~7.7]{Browncohomology} in group homology, there is an isomorphism $H_n(A,C)\oplus H_n(B,C)\to H_n(A\ast B,C)$ for $n\geq 1$ induced by the inclusions. By Lemma \ref{l:multispinal} the map $H_n(\sigma_1,C)\circ \mathrm{tr}^{A\ast B}_{(A\ast B)_1}$ for $n\geq 1$ is given (under these identifications) by the matrix   
\begin{equation}\label{eq:transfer.matrix.multispinal}
M=\begin{bmatrix} 0 & \sum_{a\in A_1}H_n(\Phi_a,C)\\ 0 &\sum_{a\in A_0}H_n(\Phi_a,C)\end{bmatrix} \in \End(H_n(A,C) \oplus H_n(B,C))
\end{equation}
since $\mathrm{tr}^A_1$ factors through the homology of the trivial group.
Therefore, we have that \[\id -M = \begin{bmatrix} \id & -\sum_{a\in A_1}H_n(\Phi_a,C)\\ 0 &\id -\sum_{a\in A_0}H_n(\Phi_a,C)\end{bmatrix}.\]
  It follows that $\ker (\id -M) \cong \ker (\id -\sum_{a\in A_0}H_n(\Phi_a,C))$ and $\coker (\id -M)\cong \coker (\id -\sum_{a\in A_0}H_n(\Phi_a,C))$. The latter isomorphism is clear as the image of $\id -M$ is $H_n(A,C)\oplus \image(\id -\sum_{a\in A_0}H_n(\Phi_a,C))$. For the former, notice that $\ker (\id -M)$  consists of $(x,y)$ with $y\in \ker  (\id -\sum_{a\in A_0}H_n(\Phi_a,C))$ and $x= \sum_{a\in A_1}H_n(\Phi_a,C)(y)$.
The result now follows from the long exact sequence in Corollary~\ref{c:transitive.transfer}  and the observation that multiplication by $|A|-1$ is injective on $C$.  The final statement follows because $H_0(\mathscr G)\cong \mathbb Z/(|A|-1)\mathbb Z$ and the homology of any finite group is finite in degree greater than $0$.  Therefore, $H_n(\mathscr G)$ is finite for $n\geq 1$ from the long exact sequence.  
\end{proof}

\begin{Rmk}
We sketch here a topological proof that the matrix of $H_n(\sigma_1)\circ \mathrm{tr}^{A\ast B}_{(A\ast B)_1}$ is given by \eqref{eq:transfer.matrix.multispinal}.  
Let $X_A$ be a $K(A,1)$ and $X_B$ a $K(B,1)$ with a single vertex.  We may take $X=X_A\vee X_B$ as our $K(A\ast B,1)$ by a well-known result of Whitehead~\cite{Browncohomology}.  The Mayer--Vietoris sequences tells us that  $H_n(X,C)\cong H_n(X_A,C)\oplus H_n(X_B,C)\cong H_n(A,C)\oplus H_n(B,C)$, for $n\geq 1$, with the isomorphism induced by the inclusions of $X_A,X_B$.  

For each $a\in A_0$ (respectively, $a\in A_1)$, we can choose a cellular map $\p_a\colon X_B\to X_B$ (respectively, $\p_a\colon X_B\to X_A$) realizing $\Phi_a$  on fundamental groups.
The action of $A\ast B$ on $A$ is the composition of the projection to $A$ with the regular action. Thus the stabilizer $(A\ast B)_1$ is the normal closure of $B$ in $A\ast B$.  This is isomorphic to $\Asterisk_{a\in A}a\inv Ba$.  Indeed,  the $|A|$-fold regular covering space $Y$ of $X$ associated to this normal closure is constructed as follows.  Let $\til X_A$ be the universal covering space of $X_A$.  It is an $|A|$-fold contractible covering space of $X_A$ with $|A|$-vertices.  Fix a base vertex $x_1$ of   $\til X_A$. For each $a\in A$, let $x_a=a\inv x_1$ with respect to the deck action of $A$ on $\til X_A$.  Let $Y$ be the space obtained from wedging a copy $Y_a$ of $X_B$ at the vertex $x_a$ for each $a\in A$.  Then $Y$ is an $|A|$-fold regular covering space of $X$ by extending the covering $\til X_A\to X_A$ by mapping each $Y_a$ identically to $X_B$.  The deck transformation action is obtained by projecting to $A$ and performing the natural free action of $A$ on $\til X_A$, and extending it to $Y$ by shuffling rigidly the $|A|$ copies of $X_B$.   Since $\til X_A$ contracts to a point, $Y$ is homotopy equivalent to $\bigvee_{a\in A} Y_a$. As $a\inv x_1=x_a$,  $(A\ast B)_1$ is isomorphic to $\Asterisk_{a\in A}a\inv Ba$.  Notice that $Y$ is a $K(A\ast B)_1$ as it has the same universal cover as $X$. 

According to~\cite[Page~82, (E)]{Browncohomology} there is chain map, called the pretransfer, from $C_\bullet(X,C)$ to $C_\bullet(Y,C)$ taking a cell to the sum of its $|A|$ lifts, which in turn induces the transfer homomorphism as the composition $H_n(A\ast B,C)\xrightarrow{\cong} H_n(X,C)\to H_n(Y,C)\xrightarrow{\cong} H_n((A\ast B)_1,C)$.    The pretransfer sends each $n$-cell of $X_A$ to a sum of the $|A|$ $n$-cells of $\til X_A$ that lift it.  Each cell of $X_B$ is sent to the sum of the corresponding copies of that cell in the $Y_a$ with $a\in A$.   The homomorphism $\sigma_1\colon (A\ast B)_1\to A\ast B$ has $\sigma_1(a\inv ba) = \Phi_a(b)$ and is realized cellularly  via the map $f\colon Y\to X$  collapsing $\til X_A$ to the wedge point and mapping $Y_a$ to $X_A\vee X_B$ by $\p_a$ followed by the inclusion. It follows that the composition of the pretransfer with the map of chain complexes induced by $f$ sends each $n$-cell $c$ of $X_A$ to $0$, and of $X_B$ to $\sum_{a\in A}\p_a(c)\in C_n(X_A\vee X_B)$.   In particular, the induced map $H_n(\sigma_1,C)\circ \mathrm{tr}^{A\ast B}_{(A\ast B)_1}$ is given by the matrix  in \eqref{eq:transfer.matrix.multispinal}  for $n\geq 1$. 
\end{Rmk}

Using Li's work~\cite{li2022ample}, we may now prove that a large number of R\"over--Nekrashevych groups are rationally acyclic. 
\begin{Cor}\label{c:rationally acyclic}
Let $G$ be a multispinal group or the Hanoi towers group.  Then the R\"over--Nekrashevych group $V(G)$ and its commutator subgroup $V(G)'$ are rationally acyclic.    
\end{Cor}
\begin{proof}
The groupoid associated to a self-similar group action $(G,X,\sigma)$ contains a copy of the boundary path groupoid of the bouquet of $|X| \geq 2$ loops, with the same unit space, and is hence minimal and purely infinite with unit space a Cantor space. Therefore,~\cite[Corollary~C]{li2022ample} applies to conclude that $V(G)$ and $V(G)'$ are rationally acyclic if $H_k(\mathscr G_{(G,X)},\mathbb Q)=0$ for $k>0$.  The result follows from Theorems~\ref{t:Hanoi} and~\ref{t:multispinal}.
\end{proof}

We next consider the corresponding K-theoretic computation. 
Recall that the groupoid associated to any contracting group is amenable (cf., Corollary~\ref{c:contracting.amenable}),
and so the universal and reduced $\cs$-algebras coincide for groupoids associated to multispinal groups.  All \v{S}uni\'c groups are amenable~\cite{sunicgroups}, and so the amenability of their groupoids also follows from that.

A (complex) character $\chi \colon H \to \mathbb C$ of a finite group $H$ is the trace of a finite dimensional unitary representation $\pi_\chi$, which is determined up to unitary equivalence by $\chi$. If $\pi_\chi$ is irreducible we call $\chi$ irreducible, and we denote by $\wh H$ the set of irreducible complex characters of $H$. Each irreducible character $\chi \in \wh H$ determines a matrix subalgebra $M_{\chi} = (\ker \pi_\chi)^\perp \subseteq \cs(H)$ of degree $\chi(1)$ and $\cs(H) = \bigoplus_{\chi \in \wh H} M_{\chi}$. Note that $\mathrm{Aut}(H)$ acts on the left of $\wh H$ by $(f,\chi)\mapsto \chi\circ f\inv$.

Recall that the Schreier graph of a left action of group $H$ on the left of a set $X$ with respect to a set of generators $S$ is the graph with vertex set $X$ and edge set $S\times X$ where $\sour(s,x) = x$ and $\ran(s,x)=sx$. 

\begin{Thm}\label{t:multispinal.k}
Let $G_{(A,B)}$ be a multispinal group coming from the data $(A,B,\Phi)$, and let $\mathscr G$ be the corresponding groupoid.  
Let $A_0=\Phi\inv(\mathrm{Aut}(B))$ and let $d=|A|$.  Let  $T$ be the adjacency matrix for the Schreier graph of the left action of $\langle A_0\rangle$ on the set of nontrivial characters of $B$ with respect to the generating set $A_0$. Then
\begin{enumerate}
    \item $K_0(\cs(\mathscr G))\cong \mathbb Z/(d-1)\mathbb Z\oplus \coker(\id -T)$,
    \item $K_1(\cs(\mathscr G))\cong \ker(\id -T)$,
\end{enumerate}
and $[1]_0 \in K_0(\cs(\mathscr G))$ is given by $1 \in \Z /(d-1)\Z$.
\end{Thm}
\begin{proof}
By a theorem of Cuntz~\cite[Page~192]{Cuntz_K_amenability}, the inclusions $A,B \to A \ast B$ induce isomorphisms $K_1(\cs(A\ast B)) \cong K_1(\cs(A))\oplus K_1(\cs(B))$ (which vanishes as $A$ and $B$ are finite) and \[K_0(\cs(A\ast B))\cong (K_0(\cs(A))\oplus K_0(\cs(B)))/\langle ([1]_0,-[1]_0)\rangle\] where $1$ is the unit.   Putting $G=A\ast B$, we have that $G_1$ is a normal subgroup of $G$ with $G/G_1\cong A$, $B\subseteq G_1$ 
and that $\sigma_1(a\inv ba) = \Phi_a(b)$ for $a\in A$.  

Corollary~\ref{c:transitive.transfer.Ktheory} and the above discussion gives us an exact sequence
\[0\to K_1(\cs(\mathscr G))\to K_0(\cs(G))\xrightarrow{1-\Lambda} K_0(\cs(G))\to K_0(\cs(\mathscr G))\to 0\]
where $\Lambda=K_0(\sigma_1)\circ \mathrm{tr}^G_{G_1}$. 
For each $\chi \in \wh A$ and $\theta \in \wh B$ pick minimal projections $p_\chi \in M_\chi$ and $p_\theta \in M_\theta$. By Cuntz's theorem and the representation theory of finite groups, $K_0(\cs(G))$ is a free abelian group with basis $[1]_0$, the  $[p_{\chi}]_0$ with $\chi\in \wh A$ nontrivial and the  $[p_{\theta}]_0$ with $\theta\in \wh B$ nontrivial. 
By Lemma \ref{l:multispinal}, if we order the basis for $K_0(\cs(G))$ so that $[1]_0$ precedes $[p_{\chi}]_0$ with $\chi\in \wh A$ nontrivial, which in turn precedes $[p_{\theta}]$ with $\theta\in \wh B$ nontrivial, then the matrix of $\Lambda$ has the upper triangular form
\[\Lambda=\begin{bmatrix}d & \ast & 0 \\ 0 & 0 & \ast\\ 0 & 0 & \sum_{a\in A_0}P_a\end{bmatrix}\] where $P_a$ is the permutation matrix encoding the action $\theta \mapsto \theta \circ \Phi_a \inv$ of $\Phi_a$ on the nontrivial characters of $\wh B$.  The last column follows because if $a \notin A_0$, the coefficient of $[1]_0 \in K_0(\cs(A))$ in $[\Phi_a(p_\theta)]_0$ is picked out by $K_0(\cs(A))\to \Z$ induced by the trivial representation $A \to \mathbb C$, which composes with $\Phi_a \colon B \to A$ to the trivial representation $B \to \mathbb C$, and since $\theta$ is nontrivial the coefficient vanishes. If $a\in A_0$, then  $[\Phi_a(p_\theta)]_0$ is the class of a minimal projection in $M_{\theta\circ \Phi_a\inv}$. It follows that $1-\Lambda$ has block form 
\begin{equation}\label{eq:block.form.k}
1-\Lambda=\begin{bmatrix}1-d & \ast & 0\\ 0 & 1 & \ast\\ 0 & 0 & 1-\sum_{a\in A_0}P_a\end{bmatrix}.
\end{equation} 

Since $d\geq 2$, we conclude that $\ker (1-\Lambda)$ is isomorphic to the free abelian group $\ker(1-\sum_{a\in A_0}P_a)$.  Now $\sum_{a\in A_0}P_a$ is the adjacency matrix $T$ for the Schreier graph of the action of $\langle A_0\rangle$ on the nontrivial characters of $B$ with respect to the generators $A_0$.  This proves (2).  
In light of \eqref{eq:block.form.k}, we see that $\coker(\id -\Lambda)\cong \mathbb Z/(d-1)\mathbb Z\oplus \coker(\id -T)$, establishing (1), and that $[1]_0$ maps to $1\in \mathbb Z/(d-1)\mathbb Z$.  
\end{proof}

\subsection{Spinal groups}\label{subs:Sunic}
In many special cases, like \v{S}uni\'c groups from Example~\ref{ex:sunic.groups} and GGS groups from Example~\ref{ex:Gupta-Sidki groups}, $|A_0|=1$, in which case we can give a more precise computation. Such groups are  examples of spinal groups in the sense of~\cite{branchgroups}.

\begin{Cor}\label{c:sunic.k}
Let $G_{(A,B)}$ be a multispinal group coming from the data $(A,B,\Phi)$ with exactly one $a\in A$ such that $\Phi_a\in \mathrm{Aut}(B)$ and let $d=|A|$. Let $n$ be the number of orbits of $\Phi_a$ on the set of nontrivial conjugacy classes of $B$.  Then 
\begin{enumerate}
    \item $K_0(\cs(\mathscr G_{(A,B)}))\cong \mathbb Z/(d-1)\mathbb Z\oplus \mathbb Z^n$,
    \item $K_1(\cs(\mathscr G_{(A,B)}))\cong \mathbb Z^n$,
\end{enumerate}
and $[1]_0 \in K_0(\cs(\mathscr G_{(A,B)}))$ is given by $1 \in \Z /(d-1)\Z$. If particular, if $f\in \mathbb F_p[x]$ is a primitive polynomial, then $K_0(\cs(\mathscr G_{p,f}))\cong \mathbb Z/(p-1)\mathbb Z\oplus \mathbb Z$ and $K_1(\cs(\mathscr G_{p,f}))\cong \mathbb Z$. 
\end{Cor}
\begin{proof}
The Schreier graph of $\Phi_a$ acting on the nontrivial irreducible characters of $B$ is just a union of cycles, one for each orbit of $\Phi_a$.  If $T$ is the adjacency matrix, then $\ker(\id -T)$  is the eigenspace of $1$, which is spanned by vectors that are constant on orbits of $\Phi_a$.  On the other hand, $\coker(\id -T)$ is the matrix for the linear transformation corresponding to the set theoretic map collapsing each orbit to a point.  Thus $\ker(\id -T)$ and $\coker(\id -T)$ are free abelian of rank the number of orbits of $B$ on nontrivial characters of $B$.  But since $(\theta\circ\Phi_a\inv)(g) =\theta(\Phi_a\inv(g))$ for all $\theta\in \wh B$, Brauer's permutation lemma (cf.~\cite{Kovacs.perm}) implies that $\Phi_a$ has the same number of orbits on $\wh B$ as $\Phi_a\inv$ does on the set of conjugacy classes of $B$.  Since $\Phi_a$ fixes the trivial character and $\Phi_a\inv$ fixes the trivial conjugacy class, we deduce that the number of orbits of $\Phi_a$ on nontrivial characters of $B$ equals the number of orbits of $\Phi_a$ on nontrivial conjugacy classes of $B$.  An application of Theorem~\ref{t:multispinal.k} completes the proof of (1), (2) and the unit computation.

The `in particular' statement follows because if $f$ is a primitive polynomial, then $\mathscr G_{p,f}$ is a multispinal group with $A=\mathbb Z/p\mathbb Z$, $B=\mathbb F_p^{\deg f}$, $\Phi_i\notin \mathrm{Aut}(B)$ for $0\leq i<p-1$ and $\Phi_{p-1}=C_f$ acts transitively on $\mathbb F_p^{\deg f}\setminus \{0\}$. 
\end{proof}

Every groupoid $\mathscr G_{p,f}$ is amenable and has a nontrivial free abelian subgroup in $K_i(\cs(\mathscr G_{p,f}))$ for $i=0,1$.  On the other hand, the homology of $\mathscr G_{p,f}$ consists entirely of finite groups by Theorem~\ref{t:multispinal}.  Therefore, all these groups fail the rational HK property in both degrees $0,1$. 
Note that the groupoid $\mathscr G_{p,f}$ is Hausdorff if and only if $\deg f=1$ and $p=2$.

Let $\mathscr G=\mathscr G_{2,x-1}$.  It is the groupoid associated the self-similar action of the infinite dihedral group $D_{\infty}$ described in Example~\ref{ex:infinite.dihedral.gp}.  Ortega and Sanchez~\cite{ortegadihedral} computed the K-theory of $\cs(\mathscr G)$ and computed the homology of $\mathscr G$ in degrees $0,1,2$ and observed that it is a torsion group in higher degrees in order to give a counterexample to the rational HK property in both degrees $0$ and $1$.  Here we compute the homology in $\mathscr G$ in every degree using our methods and give an easier computation of its K-theory.  The method extends to GGS-groups (see Example \ref{ex:Gupta-Sidki groups}) such as the groups $G_{p,x-1}$ studied in~\cite{FG85,justinfinite} and the Gupta--Sidki $p$-groups~\cite{GuptaSidki}.
GGS groups are generated by bounded automata and hence are amenable~\cite{boundedaut}. Thus their groupoids are amenable (also they are contracting groups).  

\begin{Cor}\label{c:dihedral}
Let $G$ be a GGS group over an $m$-element alphabet.\footnote{Recall that we assume $\gcd(\Phi_0(b),\cdots,\Phi_{m-2}(b),m)=1$ for the group to be multispinal.} Let $\mathscr G$ be the associated groupoid.  
Then
\begin{align*}
 H_0(\mathscr G)  & =  \mathbb Z/(m-1)\mathbb Z, & K_0(\cs(\mathscr G))  & =  \mathbb Z/(m-1)\mathbb Z\oplus \mathbb Z^{m-1}, \\   
H_n(\mathscr G) & =  \mathbb Z/m\mathbb Z,  (n\geq 1),  &  K_1(\cs(\mathscr G))  & =  \mathbb Z^{m-1}.
\end{align*}
with $[1]_0 \in K_0(\cs(\mathscr G))$ given by $1\in \mathbb Z/(m-1)\mathbb Z$. 
\end{Cor}
\begin{proof}
 Recall that $\Phi_{m-1}=\id_B$ and $\Phi_i\colon B\to A$ for $0\leq i\leq m-2$. 
It follows from Theorem~\ref{t:multispinal} that  $H_0(\mathscr G)=\mathbb Z/(m-1)\mathbb Z$ and, since $\id -H_n(\id_B)=0$, we have exact sequences, for each odd $n$, $0\to H_{n+1}(\mathscr G)\to \mathbb Z/m\mathbb Z\to 0$ and $0\to \mathbb Z/m\mathbb Z\to H_n(\mathscr G)\to 0$ by \eqref{eq:cyclic.homology}. Therefore,  $H_n(\mathscr G)\cong \mathbb Z/m\mathbb Z$ for all $n\geq 1$.  The K-theory computation is immediate from Corollary~\ref{c:sunic.k} since $\Phi_{m-1}=\id_B$ has $m-1$ orbits on the nontrivial elements of $B$.
\end{proof}

In particular, for a GGS group $G$ over an $m$-element alphabet with $\gcd(\Phi_i,m)=1$ for all $0\leq i\leq m-2$, then $\mathscr G$ is an effective, minimal and Hausdorff groupoid failing the rational HK property in both degrees. This applies in particular to the infinite dihedral group $G_{2,x-1}$  and the Gupta--Sidki $3$-group $G_3$.

%

We next consider the homology of more general \v{S}uni\'c groups.  The reader is referred back to Example~\ref{ex:sunic.groups}.

\begin{Thm}\label{t:sunic.group}
Let $G_{p,f}$ be the \v{S}uni\'c group associated to $f\in \mathbb F_p[x]$.  If the order of the companion matrix $C_f$ is not divisible by $p$, then $H_0(\mathscr G_{p,f})=\mathbb Z/(p-1)\mathbb Z$,  $H_1(\mathscr G_{p,f})$ is an $\mathbb F_p$-vector space of dimension $\dim \ker (\id -C_f)$ and $H_n(\mathscr G_{p,f})$ is an $\mathbb F_p$-vector space of dimension $\dim \ker (\id -H_n(C_f,\mathbb F_p))$ for all $n\geq 2$. 
\end{Thm}
\begin{proof}
By \eqref{eq:cyclic.homology} and the K\"unneth theorem, $H_n(B)$ is a finite dimensional $\mathbb F_p$-vector space for all $n\geq 1$. 
By Theorem~\ref{t:multispinal} we have $H_0(\mathscr G_{p,f})=\mathbb Z/(p-1)\mathbb Z$, $\ker (\id -H_0(C_f))=0$ and $H_1(\mathscr G)\cong \coker (\id -C_f)\cong \ker(\id -C_f)$ as $\mathbb F_p$-vector spaces (since $H_1(C_f)=C_f$), and so the result is true for $n=0,1$.

By the naturality of the Universal Coefficient Theorem and the fact that $H_q(B)$ is an $\mathbb F_p$-vector space for $q\geq 1$, we have a commuting diagram with exact rows for $n\geq 2$:
\begin{equation}\label{eq:make.a.module}
\begin{tikzcd}[arrow style=math font]
 0\ar{r} & H_n(B)\ar{r}\ar{d}{H_n(C_f)} & H_n(B,\mathbb F_p)\ar{r}\ar{d}{H_n(C_f,\mathbb F_p)}& H_{n-1}(B)\ar{r}\ar{d}{H_{n-1}(C_f)}& 0\\
  0\ar{r} & H_n(B)\ar{r} & H_n(B,\mathbb F_p)\ar{r}& H_{n-1}(B)\ar{r}& 0
\end{tikzcd}
\end{equation}
Let $r$ be the order of the automorphism $C_f$.   Then we can view these three vector spaces as $\mathbb F_pC_r$-modules, where $C_r$ is the cyclic group of order $r$.   The commutativity of the above diagram shows that the exact sequence appearing in the two rows is an exact sequence of $\mathbb F_pC_r$-modules.

If $H$ is any finite group of order prime to $p$ and $M$ is an $\mathbb F_pH$-module, then the natural map $\pi\colon M^H\to M_H$ given by $m\mapsto \ov m$ has inverse $\tau\colon M_H\to M^H$ given by $\tau(\ov m) = \frac{1}{|H|}\sum_{h\in H}hm$. In particular, the invariant and coinvariant functors are exact on $\mathbb F_pH$-modules.  
Thus, $\ker (\id -H_n(C_f))=H_n(B)^{C_r} \cong H_n(B)_{C_r}=\coker (\id-H_n(C_f))$ and $\ker (\id -H_n(C_f,\mathbb F_p))= H_n(B,\mathbb F_p)^{C_r} \cong  H_n(B,\mathbb F_p)_{C_r} = \coker (\id - H_n(C_f,\mathbb F_p))$.  
Assuming the result for $n-1\geq 1$, we have that $H_{n-1}(\mathscr G)=\Tor_1^{\Z}(H_{n-1}(\mathscr G),\mathbb F_p)$ by induction. 
We then have a commutative diagram for $n\geq 2$, shown in Figure~\ref{fig:cool.3x3},
\begin{figure}
\begin{tikzcd}[arrow style=math font, column sep = 0.6cm]
	& 0 & 0 & 0 \\
	0 & {H_n(B)_{C_r}} & {H_n(\mathscr G)} & {H_{n-1}(B)^{C_r}} & 0 \\
	0 & {H_n(B,\mathbb F_p)_{C_r}} & {H_n(\mathscr G,\mathbb F_p)} & {H_{n-1}(B,\mathbb F_p)^{C_r}} & 0 \\
	0 & {H_{n-1}(B)_{C_r}} & {H_{n-1}(\mathscr G)} & {H_{n-2}(B)^{C_r}} & 0 \\
	& 0 & 0 & 0
	\arrow[from=1-2, to=2-2]
	\arrow[dashed, from=1-3, to=2-3]
	\arrow[from=1-4, to=2-4]
	\arrow[from=2-1, to=2-2]
	\arrow[from=2-2, to=2-3]
	\arrow[from=2-2, to=3-2]
	\arrow[from=2-3, to=2-4]
	\arrow[from=2-3, to=3-3]
	\arrow[from=2-4, to=2-5]
	\arrow[from=2-4, to=3-4]
	\arrow[from=3-1, to=3-2]
	\arrow[from=3-2, to=3-3]
	\arrow[from=3-2, to=4-2]
	\arrow[from=3-3, to=3-4]
	\arrow[from=3-3, to=4-3]
	\arrow[from=3-4, to=3-5]
	\arrow[from=3-4, to=4-4]
	\arrow[from=4-1, to=4-2]
	\arrow[from=4-2, to=4-3]
	\arrow[from=4-2, to=5-2]
	\arrow[from=4-3, to=4-4]
	\arrow[from=4-3, to=5-3]
	\arrow[from=4-4, to=4-5]
	\arrow[from=4-4, to=5-4]
\end{tikzcd}
\caption{A commutative diagram with exact rows and columns\label{fig:cool.3x3}}
\end{figure}
with exact rows from Theorem~\ref{t:multispinal} and columns from the Universal Coefficient Theorem. The first and last column are exact by applying $C_r$-coinvariants and invariants to the exact sequence in \eqref{eq:make.a.module},\footnote{The third column is exact for $n=2$ since $H_1(C_f)=C_f=H_1(C_f,\mathbb F_p)$ and $\ker (\id-H_0(C_f))=0$.} and the middle column is exact by the Universal Coefficient Theorem except possibly at $0\to H_n(\mathscr G)\to H_n(\mathscr G,\mathbb F_p)$.
But exactness there follows from the snake lemma or a diagram chase.  In particular, $H_n(\mathscr G)$ is an $\mathbb F_p$-vector space.

Assuming the result for $n-1\geq 1$, we have that by exactness in Figure~\ref{fig:cool.3x3} and induction that
$\dim H_n(B,\mathbb F_p)^{C_r}=\dim H_n(B,\mathbb F_p)_{C_r} = \dim H_n(\mathscr G,\mathbb F_p)-\dim H_{n-1}(B,\mathbb F_p)^{C_r}=  \dim H_n(\mathscr G,\mathbb F_p)-\dim H_{n-1}(\mathscr G) = \dim H_n(\mathscr G)$.  This completes the proof.
\end{proof}

\subsection{The Grigorchuk group}\label{subs:Grig}

The most famous \v{S}uni\'c group is the Grigorchuk group $G_{2,1+x+x^2}$; see Example~\ref{ex:grigorchuk.gr}.  We compute the homology of the groupoid $\mathscr G_{2,1+x+x^2}$.  The following lemma can be deduced from the theory of Brauer characters, but we provide an elementary proof.

\begin{Lemma}\label{l:brauer.lift}
Let $p$ be a prime and let $A\in \mathrm{GL}_n(\mathbb Z)$ have finite order $k$ coprime to $p$.  Let $B\in \mathrm{GL}_n(\mathbb F_p)$ be the reduction of $A$.  Then the multiplicities of $1$ as an eigenvalue both of $A$ and $B$ are the same.  
\end{Lemma}
\begin{proof}
Both $A$ and $B$ satisfy the polynomial $x^k-1$, which has distinct roots over fields of characteristic $0$ and $p$ as $p\nmid k$.  In particular, $1$ is a semisimple eigenvalue of both $A$ and $B$.   Let $h(x)$ be the characteristic polynomial of $A$.  Then $h(x)=(x-1)^mq(x)$ in $\mathbb Z[x]$ where $q(1)\neq 0$ and $m$ is the multiplicity of $1$ as an eigenvalue of $A$. To prove our result it suffices to show that $p\nmid q(1)$, as the characteristic polynomial of $B$ is obtained from $h(x)$ by reducing mod $p$. We can factor $q(x)=q_1(x)\cdots q_r(x)$ with $q_i(x)\in \mathbb Z[x]$ irreducible over $\mathbb Q$.  Moreover, since $h(x)$ has the same irreducible factors as the minimal polynomial of $A$, each $q_i(x)$ is  a cyclotomic polynomial $\Phi_d$ where $1\neq d\mid k$.  Note that $\Phi_d$ divides $f(x)=1+x+\cdots+x^{k-1}$ in $\mathbb Z[x]$.  Now  $q(1)=q_1(1)\cdots q_r(1)$, and so if $p\mid q(1)$, then $p\mid q_i(1)$ for some $i$.    But then $p\mid q_i(1)\mid f(1)=k$, a contradiction.  Therefore $p\nmid q(1)$ as required.
\end{proof}

\begin{Thm}\label{t:grigorchuk.group}
Let $\mathscr G=\mathscr G_{2,1+x+x^2}$ be the groupoid associated to the Grigorchuk group.  Then 
\[H_n(\mathscr G) = \begin{cases}0, & \text{if}\ n=0,\\ (\mathbb Z/2\mathbb Z)^{\frac{n}{3}+1}, & \text{if}\ n\equiv 0\bmod 3, n\geq 1,\\ (\mathbb Z/2\mathbb Z)^{\frac{n-1}{3}}, & \text{if}\ n\equiv 1\bmod 3,\\ (\mathbb Z/2\mathbb Z)^{\frac{n+1}{3}}, & \text{if}\ n\equiv 2\bmod 3.\end{cases}\]  On the other hand, the associated Nekrashevych algebra $\mathcal O_{\mathrm{Grig}} = \cs(\mathscr G)$ satisfies $K_0(\mathcal O_{\mathrm{Grig}})\cong \mathbb Z\cong K_1(\mathcal O_{\mathrm{Grig}})$ with $[1]_0 = 0$.
\end{Thm}
\begin{proof}
 The K-theory computation is immediate from Corollary~\ref{c:sunic.k} as $1+x+x^2$ is a primitive polynomial.
 
Let $V$ be the Klein $4$-group with elements $b,c,d,1$ with $1$ the identity. Put $f=1+x+x^2$, which has companion matrix $C_f$ as in Example~\ref{ex:grigorchuk.gr}. Then the action of $C_f$ on $V$ is the $3$-cycle $(b,c,d)$.  Taking as a $K(V,1)$ the space $\mathbb {RP}^\infty\times \mathbb {RP}^\infty$   one can compute by the K\"unneth theorem that 
\[H_n(V)\cong \begin{cases}\mathbb Z, &  \text{if}\ n=0,\\ (\mathbb Z/2\mathbb Z)^{\frac{n}{2}}, & \text{if}\ n\in 2\mathbb Z, n>0,\\ (\mathbb Z/2\mathbb Z)^{\frac{n+3}{2}},  & \text{if}\ n\in 1+2\mathbb Z.\end{cases}\]  

The automorphism $C_f$ has order $3$, and so we can apply Theorem~\ref{t:sunic.group}.  Therefore, $H_0(\mathscr G)=0$.  Moreover, $H_1(\mathscr G)=0$, as 
$C_f$ has characteristic polynomial $f=1+x+x^2$, and hence $1$ is not an eigenvalue.  For $n\geq 2$, we have that $H_n(\mathscr G)$ is an $\mathbb F_2$-vector space of dimension the multiplicity of $1$ as an eigenvalue of $H_n(C_f,\mathbb F_2)$.

We shall make use of two $K(V,1)$s in the proof, both $\mathbb {RP}^\infty\times \mathbb RP^\infty$  and $BV$.  The CW structure on $\mathbb {RP}^{\infty}$ has a single cell in each dimension.  Since the cellular boundary maps for $\mathbb RP^\infty$ are $0$ in odd degree and multiplication by $2$ in even degree (greater than $0$), the boundary maps in the cellular chain complex $C_\bullet(\mathbb {RP}^\infty\times \mathbb {RP}^\infty,\mathbb F_2)\cong C_\bullet(\mathbb {RP}^{\infty},\mathbb F_2)\otimes C_\bullet(\mathbb {RP}^{\infty},\mathbb F_2)$ (the tensor product of chain complexes of $\mathbb F_2$-vector spaces) are all zero.  
We write $e_i$ for the unique cell of dimension $i$ of $\mathbb RP^\infty$.  Then the $e_i\otimes e_{n-i}$, with $0\leq i\leq n$, form a basis for the degree $n$ component of  $C_\bullet(\mathbb {RP}^{\infty},\mathbb F_2)\otimes C_\bullet(\mathbb {RP}^{\infty},\mathbb F_2)$. Since the boundary maps in this complex are all zero, we shall not distinguish between the chain vector spaces and the homology spaces.

\begin{Claim*}
The map $H_n(C_f,\mathbb F_2)$ is given by $e_j\otimes e_{n-j} \mapsto \sum_{i=0}^n\binom{n-i}{j}e_i\otimes e_{n-i}$ (modulo $2$). 
\end{Claim*}

We defer the proof of the claim in order to show how the result follows from it.  As mentioned above, $H_n(\mathscr G)$ is an $\mathbb F_2$-vector space of dimension the multiplicity of $1$ as an eigenvalue of the 
matrix $B$ of $H_n(C_f,\mathbb F_2)$.   Let $A$ be the $(n+1)\times (n+1)$ integer matrix, with rows and columns indexed by $0,\ldots, n$, with $A_{ij} = (-1)^i\binom{n-i}{j}$.  Then $A$ reduces modulo $2$ to $B$ by the claim.  It was observed by M.~Wildon~\cite{Wildon} that $A^3=(-1)^n\id$ (see Lemma~\ref{l:wildon} for a proof).  It follows that $A^4$ has order $3$ and reduces to $B$ modulo $2$.    Since $2\nmid 3$, the multiplicities of $1$ as an eigenvalue of $A^4$ and $B$ are the same by Lemma~\ref{l:brauer.lift}.  So, we are reduced to computing the multiplicity of $1$ as an eigenvalue of $A^4$.
 
The trace of $A$ is 
\begin{equation}\label{eq:trace}
\sum_{i=0}^n (-1)^i\binom{n-i}{i} = \begin{cases}1, & \text{if}\ n\equiv 0,1\bmod 6\\0, & \text{if}\ n\equiv 2,5\bmod 6\\ -1, & \text{if}\ n\equiv 3,4\bmod 6\end{cases}
\end{equation}
see~\cite{binomialcoeffs}.

First we handle the case that $n$ is even, and so $A=A^4$ has order $3$.  Since $A$ has order $3$ it diagonalizes over $\mathbb C$ and its complex eigenvalues come in conjugate pairs.  Let $k$ be the multiplicity of the eigenvalue $1$ and $m$ the multiplicity of the pair of conjugate primitive $3^{rd}$-roots of unity. Then $n+1=k+2m$, and since the sum of the two primitive $3^{rd}$ roots of units is $-1$, the trace in \eqref{eq:trace} is $k-m$.  Running through the three cases $n\equiv 0,2,4\bmod 6$, 
we see that $k=1+n/3$ if $n\equiv 0\bmod 6$,  $k=(n+1)/3$ if $n\equiv 2\bmod 6$ and $k=(n-1)/3$ if $n\equiv 4\bmod 6$.

Next consider the case that $n$ is odd.  Then $A^3=-\id$ and so the minimal polynomial of $A$ divides $x^3+1=(x+1)(x^2-x+1)$.  So $A$ diagonalizes with eigenvalues $-1$ and the two primitive $6^{th}$-roots of unity, and the complex eigenvalues come in conjugate pairs.    Note that the sum of the two complex eigenvalues is $1$. Let $k$ be the multiplicity of the eigenvalue $-1$ and $m$ the multiplicity of the conjugate pair of complex eigenvalues.  Then $k$ is the multiplicity of $1$ as an eigenvalue of $A^4$. Note that $n+1=k+2m$ and $m-k$ is the right-hand side of \eqref{eq:trace}.  Breaking up into the three cases $n\equiv 1,3,5\bmod 6$, we see that $k=(n-1)/3$ if $n\equiv 1\bmod 6$, $k=n/3+1$ if $n\equiv 3\bmod 6$ and $k=(n+1)/3$ if $n\equiv 5\bmod 6$.    

Since $\dim H_n(\mathscr G)$ is multiplicity of $1$ as an eigenvalue of $B$, which is the same as the multiplicity of $1$ as an eigenvalue of $A^4$, this completes the proof of the homology computation assuming the claim.  
\end{proof}

\begin{proof}[Proof of claim.]
Unfortunately, it does not seem easy to find an explicit cellular map on $\mathbb {RP}^\infty\times \mathbb {RP}^\infty$ that realizes $C_f$ on the fundamental group.  But it is easy to compute the effect of $C_\bullet(C_f,\mathbb F_2)$ on $C_\bullet(BV,\mathbb F_2)$. 
Fortunately, the Eilenberg--Zilber and Alexander--Whitney maps give explicit chain homotopy inverse morphisms between the cellular chain complexes of these $K(V,1)$s.  We can conjugate the action of $C_\bullet(C_f,\mathbb F_2)$ by these maps to get an action on the chain complex of $\mathbb {RP}^\infty\times \mathbb {RP}^\infty$ computing $H_\bullet(\Phi,\mathbb F_2)$.

 Suppose that $G,K$ are groups.  We recall the definitions of the Alexander--Whitney~\cite{Browncohomology,MacHomology} and Eilenberg--Zilber (or shuffle) maps~\cite{MacHomology} for $BG$ and $BK$ for the case of $\mathbb F_2$-coefficients only.  The Alex\-ander--Whitney map $\gamma\colon C_\bullet(B(G\times K),\mathbb F_2)\to C_\bullet(BG,\mathbb F_2)\otimes C_\bullet(BK,\mathbb F_2)$ is given on nondegenerate $n$-simplices by \[\gamma((g_1,k_1),\ldots,(g_n,k_n)) = \sum_{i=1}^{n+1} (g_1,\ldots, g_{i-1})\otimes (k_{i},\ldots k_n)\]  where any degenerate simplex is treated as $0$. The Eilenberg--Zilber map $\eta\colon C_\bullet(BG,\mathbb F_2)\otimes C_\bullet(BK,\mathbb F_2)\to C_\bullet(B(G\times K),\mathbb F_2)$ is more complicated to define.
 A $(p,q)$-shuffle is a permutation $\sigma$ of $1,\ldots, p+q$ such that $\sigma|_{[1,p]}$ and $\sigma|_{[p+1,p+q]}$ are order-preserving.
 We can similarly define a $(p,q)$-shuffle of any set of $p+q$ elements with a fixed a linear order, and we write the result of the shuffle as a $(p+q)$-tuple.  
 
 Because we are working over $\mathbb F_2$, we can ignore the signs that appear in the usual definition of the Eilenberg--Zilber map over $\mathbb Z$ to obtain that $\eta((g_1,\ldots, g_i)\otimes (k_{i+1},\ldots, k_n))$ is the sum of all $(i,n-i)$-shuffles of \[(g_1,1)\ldots, (g_i,1),(1,k_{i+1}),\ldots, (1,k_n).\]  Then $\gamma\eta$ and $\eta\gamma$ are chain homotopic to identity maps~\cite[Chapter~VIII.8]{MacHomology}. 
 
 Specializing to the case $G=K=\mathbb Z/2\mathbb Z$ and using that $B(\mathbb Z/2\mathbb Z)$ and $\mathbb {RP}^{\infty}$ are isomorphic CW complexes, it follows that $\gamma C_\bullet(C_f,\mathbb F_2)\eta$ is a chain map on  $C_\bullet(\mathbb {RP}^{\infty},\mathbb F_2)\otimes C_\bullet(\mathbb {RP}^{\infty},\mathbb F_2)$, which we can identify with its homology since all boundary maps are zero, inducing $H_\bullet(C_f,\mathbb F_2)$.  To prove the claim, we show that $\gamma C_\bullet(C_f,\mathbb F_2)\eta(e_j\otimes e_{n-j}) = \sum_{i=0}^n\binom{n-i}{j}e_i\otimes e_{n-i}$ (modulo $2$).

 If $X$ is a set and $x\in X$, it will be convenient to write $x^{(n)}$ for the $n$-tuple $(x,\ldots, x)$.  
 Write $V=\langle b\rangle\times \langle c\rangle$.  Note that $d=(b,c)$. We have that the unique nondegenerate $n$-cell of $B(\langle b\rangle)\cong \mathbb {RP}^{\infty}$ is $b^{(n)}$, and similarly the unique nondegenerate $n$-cell of $B(\langle c\rangle)$ is $c^{(n)}$.  Thus $e_j\otimes e_{n-j} = b^{(j)}\otimes c^{(n-j)}$.  
 
 The Eilenberg--Zilber map sends $b^{(j)}\otimes c^{(n-j)}$ to the sum of all shuffles of $(b,1),\ldots,(b,1)$ and $(1,c),\ldots, (1,c)$ with $j$ copies of $(b,1)$ and $n-j$ copies of $(1,c)$.  Note that $C_f((b,1))=(1,c)$ and $C_f((1,c))=(b,c)$, and so $C_\bullet(C_f,\mathbb F_2)\circ \eta$ sends $e_j\otimes e_{n-j}$ to the sum of all shuffles of $(1,c),\ldots, (1,c)$ ($j$-copies) and $(b,c),\ldots, (b,c)$ ($(n-j)$-copies).   This is the sum of all elements of the form $((x_1,c),(x_2,c),\ldots, (x_n,c))$, where exactly $j$ of the $x_i$ are $1$, and the remaining $n-j$ are $b$.  
 Next we compute the effect of the Alexander--Whitney map on a simplex $\tau=((b,c),(b,c),\ldots,(b,c),(1,c),(x_{r+2},c),\ldots, (x_n,c))$, where $x_k\in \{1,b\}$, $r$ is number of leading $(b,c)$s and exactly $n-j-r$ of the $x_k$s are $b$s.  Since any simplex of $B(\langle b\rangle)$ containing a $1$ is degenerate, it follows that this simplex is sent to $\sum_{k=0}^r b^{(k)}\otimes c^{(n-k)}$.
 So an occurrence of $b^{(i)}\otimes c^{(n-i)}$ in $\gamma C_n(C_f,\mathbb F_2)\eta(b^{(j)}\otimes c^{(n-j)})$  corresponds to such a $\tau$ with at least $i$ leading $(b,c)$s.  We then have $n-j-i$ remaining $(b,c)$s to place in $n-i$ locations. There are $\binom{n-i}{n-j-i}=\binom{n-i}{j}$ such simplices.  This establishes the claim.
 \end{proof}

\begin{Cor}
R\"over's simple group $V(G)$ containing the Grigorchuk group $G$ is rationally acyclic but not acyclic as it has Schur multiplier $H_2(V(G))\cong \mathbb Z/2\mathbb Z$.   \end{Cor}
\begin{proof}
We  saw that $V(G)$ is rationally acyclic in Corollary~\ref{c:rationally acyclic}.  Since the groupoid $\mathscr G$ associated to the Grigorchuk group satisfies $H_n(\mathscr G)=0$ for $n<2$ and $H_2(\mathscr G)\cong \mathbb Z/2\mathbb Z$ by Theorem~\ref{t:grigorchuk.group},  the result follows from~\cite[Corollary~D]{li2022ample}.
\end{proof}

Since $\mathcal O_{\mathrm{Grig}}$ is a UCT Kirchberg algebra, our K-theory computation identifies it with the 2-adic ring $\cs$-algebra $\mathcal Q_2$ of the integers~\cite{LarsenLi12} by the Kirchberg--Phillips theorem.

\subsection{The Grigorchuk--Erschler group}\label{subs:GE}

The Grigorchuk--Erschler group is the multispinal group $G_{2,1+x^2}$; see Example~\ref{ex:Grig-Erschler}. The following computation offers a taste of how to compute the homology in the setting of \v{S}uni\'c groups for which Theorem \ref{t:sunic.group} does not apply.

\begin{Thm}
Let $\mathscr G=\mathscr G_{2,1+x^2}$ be the groupoid associated to the Grigor\-chuk--Erschler group $G_{2,1+x^2}$.  Then
\[H_n(\mathscr G)=\begin{cases}
0, & \text{if}\ n=0,\\
(\mathbb Z/2\mathbb Z)^{\frac{n}{2}+1}, & \text{if}\ n\equiv 0\bmod 2, n>0,\\
(\mathbb Z/2\mathbb Z)^{\frac{n+1}{2}}, & \text{if}\ n\equiv 1\bmod 4,\\
(\mathbb Z/2\mathbb Z)^{\frac{n-1}{2}}\oplus \mathbb Z/4\mathbb Z, & \text{if}\ n\equiv 3\bmod 4.\\
\end{cases}\]
Moreover, $K_0(\cs(\mathscr G)) \cong \mathbb Z^2\cong K_1(\cs(\mathscr G))$ with $[1]_0 = 0$.
\end{Thm}
\begin{proof}
Let $f=1+x^2$ with companion matrix
\[ C_f = \begin{bmatrix} 0 & 1 \\ 1 & 0 \end{bmatrix}. \]
The K-theory computation follows from  Corollary~\ref{c:sunic.k} as $\Phi_1=C_f$ swaps to the two standard basis elements, and thus has two orbits on nonzero elements:  $\{(1,1)\}$ and $\{(1,0),(0,1)\}$.  

Again, let $V=\mathbb Z/2\mathbb Z\times \mathbb Z/2\mathbb Z$ and take $\mathbb {RP}^\infty\times \mathbb{RP}^\infty$ as our $K(V,1)$.  The automorphism $C_f$ of $V$ is then induced on the fundamental group by the cellular map swapping the two coordinates, and so we can use this complex directly to compute $H_n(C_f)$ and $H_n(C_f,\mathbb F_2)$.  
The automorphism $C_f$ has order $2$, and so generates a copy of the cyclic group $C_2$.   

Then $H_n(V)$ and $H_n(V,\mathbb F_2)$ are $\mathbb F_2C_2$-modules.  
 By the universal coefficient theorem, we have the same commutative diagram as \eqref{eq:make.a.module} with $B=V$ and $p=2$.  Hence there is an exact sequence of $\mathbb F_2C_2$-modules
\begin{equation}\label{eq:exact.c2.modules}
0\to H_n(V)\to H_n(V,\mathbb F_2)\to H_{n-1}(V)\to 0.
\end{equation}

Note that
$\coker (\id -H_n(C_f)) = H_n(V)_{C_2}$, $\ker (\id -H_n(C_f)) = H_n(V)^{C_2}$, $\coker (\id -H_n(C_f,\mathbb F_2)) = H_n(V,\mathbb F_2)_{C_2}$ and $\ker (\id-H_n(V,\mathbb F_2)) = H_n(V,\mathbb F_2)^{C_2}$. Moreover, $\id -H_1(C_f)=\id -C_f=\id -H_1(C_f,\mathbb F_2)$. Also note that $H_n(\mathscr G)\otimes_{\mathbb Z}\mathbb F_2\cong H_n(\mathscr G)/2H_n(\mathscr G)$ and that the functor $H \mapsto \Tor_1^{\Z}(H,\mathbb F_2)=\{h\in H\mid 2h=0\}$ on abelian groups $H$ is left exact and fixes $\mathbb F_2$-vector spaces.  Since $H_{n-1}(V)_{C_2}$ and $H_{n-2}(V)^{C_2}$ are $\mathbb F_2$-vector spaces for $n\geq 2$,  applying coinvariants and invariants to \eqref{eq:exact.c2.modules} and using Theorem~\ref{t:multispinal} and the Universal Coefficient Theorem, we have the  commutative diagram in Figure~\ref{fig:bigdiagram} with exact rows and columns for $n\geq 2$.

\begin{figure}
\begin{tikzcd}[arrow style=math font, column sep = 0.3cm]
    & {H_1(C_2,H_n(V))} & \\
	& {H_1(C_2,H_n(V, \mathbb F_2))} & 0  \\
	& {H_1(C_2,H_{n-1}(V))} & {2H_n(\mathscr G)} & 0  \\
	0 & {H_n(V)_{C_2}} & {H_n(\mathscr G)} & {H_{n-1}(V)^{C_2}}  & 0 \\
	0 & {H_n(V,\mathbb F_2)_{C_2}} & {H_n(\mathscr G;\mathbb F_2)} & {H_{n-1}(V,\mathbb F_2)^{C_2}} & 0  \\
	0 & {H_{n-1}(V)_{C_2}} & {\Tor_1^{\Z}(H_{n-1}(\mathscr G),\mathbb F_2)} & {H_{n-2}(V)^{C_2}} \\
	& 0 & 0 &  \\
    \arrow[from=1-2, to=2-2]
	\arrow[from=2-2, to=3-2]
	\arrow[from=2-3, to=3-3]
	\arrow[from=3-2, to=4-2]
	\arrow[from=3-3, to=4-3]
	\arrow[from=3-4, to=4-4]
	\arrow[from=4-1, to=4-2]
	\arrow[from=4-2, to=4-3]
	\arrow[from=4-2, to=5-2]
	\arrow[from=4-3, to=4-4]
	\arrow[from=4-3, to=5-3]
    \arrow[from=4-4, to=4-5]
	\arrow[from=4-4, to=5-4]
	\arrow[from=5-1, to=5-2]
	\arrow[from=5-2, to=5-3]
	\arrow[from=5-2, to=6-2]
	\arrow[from=5-3, to=5-4]
	\arrow[from=5-3, to=6-3]
    \arrow[from=5-4, to=5-5]
	\arrow[from=5-4, to=6-4]
	\arrow[from=6-1, to=6-2]
	\arrow[from=6-2, to=6-3]
	\arrow[from=6-2, to=7-2]
	\arrow[from=6-3, to=6-4]
	\arrow[from=6-3, to=7-3]
\end{tikzcd}
\caption{A commutative diagram with exact rows and columns\label{fig:bigdiagram}}
\end{figure}

The snake lemma or a diagram chase provides an isomorphism  of $2H_n(\mathscr G)$ with the cokernel of  $H_1(C_2,H_n(V,\mathbb F_2))\to H_1(C_2, H_{n-1}(V))$. We proceed to compute this cokernel.  
 
By Shapiro's Lemma $H_1(C_2,\mathbb F_2C_2)\cong H_1(1,\mathbb F_2) =0$.  On the other hand, the Universal Coefficient Theorem yields $H_1(C_2,\mathbb F_2)\cong H_1(C_2)\otimes_{\mathbb F_2}\mathbb F_2\cong \mathbb F_2$.   

In order to use the diagram in Figure~\ref{fig:bigdiagram}, we need three claims.
\begin{Claim}\label{cl:claim1}
As $\mathbb F_2C_2$-modules we have
\[H_n(V,\mathbb F_2)\cong \begin{cases}\mathbb F_2\oplus \mathbb F_2C_2^{\frac{n}{2}}, & \text{if}\ n\equiv 0\bmod 2,\\ \mathbb F_2C_2^{\frac{n+1}{2}}, & \text{if}\ n\equiv 1\bmod 2.\end{cases}\]
Therefore,
\[H_1(C_2,H_n(V,\mathbb F_2))\cong \begin{cases}\mathbb F_2, & \text{if}\ n\equiv 0\bmod 2, \\ 0, & \text{if}\ n\equiv 1\bmod 2.\end{cases}\]
\end{Claim}
\begin{proof}[Proof of claim.]
First note that $H_n(V,\mathbb F_2)$ has basis $e_i\otimes e_{n-i}$ where $e_j$ is the unique $j$-cell of $\mathbb{RP}^{\infty}$.  The action of $C_2$ swaps the two 
coordinates.  We see that if $n$ is even, then $H_n(V,\mathbb F_2)\cong \mathbb F_2\oplus \mathbb F_2C_2^{\frac{n}{2}}$ where the copy of $\mathbb F_2$ is spanned by $e_{n/2}\otimes e_{n/2}$ and the copies of $\mathbb F_2C_2$ are spanned by $e_i\otimes e_{n-i},e_{n-i}\otimes e_i$ with $0\leq i<n/2$. 
If $n$ is odd, then $H_n(V,\mathbb F_2)\cong \mathbb F_2C_2^{\frac{n+1}{2}}$ with the copies spanned by $e_i\otimes e_{n-i}, e_{n-i}\otimes e_i$ with $0\leq i\leq (n-1)/2$. 
The second equation follows because $H_1(C_2,\mathbb F_2)\cong \mathbb F_2$ and $H_1(C_2,\mathbb F_2C_2)=0$.
\end{proof}

\begin{Claim}\label{cl:claim3}
As $\mathbb F_2C_2$-modules we have
\[H_n(V) = \begin{cases}\mathbb F_2C_2^{\frac{n}{4}}, & \text{if}\ n\equiv 0\bmod 4,\\ \mathbb F_2C_2^{\frac{n+3}{4}}, & \text{if}\ n\equiv 1\bmod 4\\ \mathbb F_2\oplus \mathbb F_2C_2^{\frac{n-2}{4}}, & \text{if}\ n\equiv 2\bmod 4\\ \mathbb F_2\oplus\mathbb F_2C_2^{\frac{n+1}{4}}, & \text{if}\ n\equiv 3\bmod 4. \end{cases}\]
Therefore,
\[H_1(C_2,H_n(V)) = \begin{cases}0, & \text{if}\ n\equiv 0,1\bmod 4,\\ 
\mathbb F_2, & \text{if}\ n\equiv 2,3\bmod 4. \end{cases}\]
\end{Claim}
\begin{proof}[Proof of claim.]
If $n$ is even, $H_n(V)$ has basis the $e_i\otimes e_{n-i}$ with $i$ odd  under the embedding in \eqref{eq:exact.c2.modules}.  If $n\equiv 0\bmod 4$, this implies that $H_n(V)\cong \mathbb F_2C_2^{\frac{n}{4}}$ since $n/2$ is even.  If $n\equiv 2\bmod 4$, then $H_n(V)\cong \mathbb F_2\oplus \mathbb F_2C_2^{\frac{n-2}{4}}$ where $e_{n/2}\otimes e_{n/2}$ spans the copy of $\mathbb F_2$ (note that $n/2$ is odd in this case).

If $n$ is odd, $H_n(V)$ has basis the elements $e_0\otimes e_n$, $e_n\otimes e_0$ and $e_i\otimes e_{n-i}+e_{i+1}\otimes e_{n-i-1}$ with $1\leq i\leq n-2$ odd.  
If $n\equiv 1\bmod 4$, then these basis elements are swapped in pairs since $i$ odd and $i=n-i-1$ implies $(n-1)/2$ is odd.  Thus $H_n(V)\cong \mathbb F_2C_2^{\frac{n+3}{4}}$.
If $n\equiv 3\bmod 4$, then $(n-1)/2$ is odd and $e_{(n-1)/2}\otimes e_{(n+1)/2}+ e_{(n+1)/2}\otimes e_{(n-1)/2}$ is fixed by the action.  The remaining basis elements are swapped in pairs, and so $H_n(V)\cong \mathbb F_2\oplus \mathbb F_2C_2^{\frac{n+1}{4}}$.  The second equation follows because $H_1(C_2,\mathbb F_2)\cong \mathbb F_2$ and $H_1(C_2,\mathbb F_2C_2)=0$.
\end{proof}

\begin{Claim}\label{cl:claim2}
For $n\geq 2$, $H_n(\mathscr G,\mathbb F_2)=\mathbb F_2^{n+1}$.
\end{Claim}
\begin{proof}
Since $H_n(\mathscr G,\mathbb F_2)$ is an $\mathbb F_2$-vector space, from the exactness in Figure~\ref{fig:bigdiagram}, we have that $\dim H_n(\mathscr G,\mathbb F_2)=\dim H_n(V,\mathbb F_2)_{C_2}+\dim H_{n-1}(V,\mathbb F_2)^{C_2}$.  Observing that $(\mathbb F_2)_{C_2}\cong \mathbb F_2\cong (\mathbb F_2)^{C_2}$ and $(\mathbb F_2C_2)_{C_2}\cong \mathbb F_2\cong (\mathbb F_2C_2)^{C_2}$, the claim follows from Claim~\ref{cl:claim1}.
\end{proof}

We conclude from the exactness in Figure~\ref{fig:bigdiagram} and Claims~\ref{cl:claim1} and~\ref{cl:claim3} that  $2H_n(\mathscr G)\cong \coker (H_1(C_2,H_n(V,\mathbb F_2))\to H_1(C_2, H_{n-1}(V)))$ is $\mathbb F_2$ if $n\equiv 3\bmod 4$ and is $0$ if $n\equiv 0,1,2\bmod 4$.  The only nonobvious case is when $n\equiv 0\bmod 4$.  In this case  $H_1(C_2,H_n(V))\cong 0$, and so $\mathbb F_2\cong H_1(C_2,H_n(V,\mathbb F_2))\to H_1(C_2,H_{n-1}(V))\cong \mathbb F_2$ must be injective, hence an isomorphism.

We now turn to the proof of the theorem.
The theorem for $n=0$ follows from Theorem~\ref{t:multispinal}.  For $n=1$, since $V$ is abelian, we have that $H_1(\mathscr G)\cong \coker(\id -C_f)\cong \mathbb Z/2\mathbb Z$, as $\id -C_f$ is the $2\times 2$ all ones matrix over $\mathbb F_2$. 
This handles the base cases.

Assume the result is true for $n-1$ and let $n\geq 2$.
First assume that $n$ is even.  Then $2H_n(\mathscr G)=0$, and hence $H_n(\mathscr G)$ is an $\mathbb F_2$-vector space.   By induction, either $H_{n-1}(\mathscr G)\cong (\mathbb Z/2\mathbb Z)^{\frac{n}{2}}$ or $H_{n-1}(\mathscr G)\cong \mathbb (\Z/2\mathbb Z)^{\frac{n-2}{2}}\oplus \mathbb Z/4\mathbb Z$.  In either case, $\Tor_1^{\mathbb Z}(H_{n-1}(\mathscr G),\mathbb F_2)\cong  (\mathbb Z/2\mathbb Z)^{\frac{n}{2}}$.  Applying exactness of the middle column of Figure~\ref{fig:bigdiagram} and using Claim~\ref{cl:claim2}, we see that $\dim H_n(\mathscr G) =\frac{n}{2}+1$, as required. 
Next suppose that $n$ is odd.  Then $\Tor_1^{\mathbb Z}(H_{n-1}(\mathscr G),\mathbb F_2) \cong (\mathbb Z/2\mathbb Z)^{\frac{n+1}{2}}$ by induction. We have two cases.  If $n\equiv 1\bmod 4$, then $2H_n(\mathscr G)=0$, and so $H_n(\mathscr G)$ is an $\mathbb F_2$-vector space of dimension $\frac{n+1}{2}$ by Claim~\ref{cl:claim2} and exactness in the middle column of Figure~\ref{fig:bigdiagram}.  
Next assume that $n\equiv 3\bmod 4$. In this case, $2H_n(\mathscr G)\cong \mathbb Z/2\mathbb Z$.  By exactness of the middle column in Figure~\ref{fig:bigdiagram} and Claim~\ref{cl:claim2}, we have that $H_n(\mathscr G)/2H_n(\mathscr G)\cong \mathbb (\mathbb Z/2\mathbb Z)^{\frac{n+1}{2}}$. It follows from the structure theorem for finite abelian groups that $H_n(\mathscr G)\cong (\mathbb Z/2\mathbb Z)^{\frac{n-1}{2}}\oplus \mathbb Z/4\mathbb Z$.   This completes the proof of the homology computation.    
\end{proof}

\section{Computations: solvable self-similar groups}\label{subsub:solvable}

\subsection{Lamplighter groups}\label{subs:lamp}
We compute the homology and $K$-theory for groupoids associated to self-similar actions of lamplighter groups given in Example~\ref{ex:lamplighter}.

First we compute the homology of $A\wr \mathbb Z$.  Note that $\mathbb Z$ acts on $H_n(\bigoplus_{i \in \mathbb Z}A \delta_i)$, and hence it makes sense to take $\mathbb Z$-coinvariants.

\begin{Prop}\label{p:lamplighter.homology}
Let $A$ be a finite abelian group.  Then 
\[H_n(A\wr \mathbb Z) = \begin{cases} \mathbb Z, & \text{if}\ n=0,\\
A\oplus \mathbb Z, & \text{if}\ n=1,\\
H_n(\bigoplus\limits_{i\in \mathbb Z}A\delta_i)_{\mathbb Z}, & \text{if}\ n\geq 2,\end{cases}\]
where for $n\geq 2$, the isomorphism is induced by the map $H_n(\bigoplus_{i\in \mathbb Z}A\delta_i)\to H_n(A\wr \mathbb Z)$ coming from the inclusion.
\end{Prop}
\begin{proof}
The result for $n=0$ is trivial.  For $n=1$, it suffices to observe that the abelianization map $A\wr \mathbb Z\to A\oplus \mathbb Z$ is given by $(x,k)\mapsto (\varepsilon(x),k)$ where $\varepsilon(s\delta_i)= s$.

Since the action of $G$ on $H_n(G)$ induced by conjugation is trivial for any group $G$ (cf.~\cite[Proposition~8.1]{Browncohomology}), we observe that the natural map  $H_n(\bigoplus_{i\in \mathbb Z}A\delta_i)\to H_n(A\wr \mathbb Z)$ factors through $H_n(\bigoplus_{i\in \mathbb Z}A\delta_i)_{\mathbb Z}$ for all $n$.  We shall use the Lyndon--Hochschild--Serre spectral sequence~\cite{Browncohomology} \[E^2_{p,q} = H_p(\mathbb Z,H_q(\bigoplus_{i\in \mathbb Z}A\delta_i))\Rightarrow H_{p+q}(A\wr \mathbb Z).\]   We claim that $E^2_{p,q}=0$ unless $p=0$ or $p=1$, $q=0$.  

First note that since $\mathbb Z$ has cohomological dimension $1$ (the circle is a $K(\mathbb Z,1)$), it follows that $H_p(\mathbb Z,H_q(\bigoplus_{i\in \mathbb Z}A\delta_i))=0$ for $p\geq 2$. Since the circle is an orientable $1$-manifold, $\mathbb Z$ is an orientable Poincar\'e duality group of dimension $1$, cf.~\cite[Chapter~VIII.10]{Browncohomology}.  It follows  $H_1(\mathbb Z,H_q(\bigoplus_{i\in \mathbb Z}A\delta_i))\cong H^0(\mathbb Z,H_q(\bigoplus_{i\in \mathbb Z}A\delta_i))\cong H_q(\bigoplus_{i\in \mathbb Z}A\delta_i)^{\mathbb Z}$.  Now we have $H_q(\bigoplus_{i\in \mathbb Z}A\delta_i)=\varinjlim_{n\in \mathbb N} H_q(\bigoplus_{i=-n}^nA\delta_i)$. Moreover, we have retractions $\rho_n\colon \bigoplus_{i\in \mathbb Z}A\delta_i\to \bigoplus_{i=-n}^nA\delta_i$,  and so the maps of this direct system are injective. 
Suppose that $0\neq z\in H_q(\bigoplus_{i\in \mathbb Z}A\delta_i)^{\mathbb Z}$ with $q\geq 1$.  Fix $n$ such that $z$ comes from $H_q(\bigoplus_{i=-n}^nA\delta_i)$.  Then $H_q(\rho_n)(z)=z\neq 0$.  But if we apply the shift $2n+1$ times to $z$, we obtain an element that maps to $0$ under $H_q(\rho_n)$ since $\rho_n(\bigoplus_{i\geq n+1}A\delta_i)=0$.  This contradicts that $z$ is fixed by the shift.  Thus $H_q(\bigoplus_{i\in \mathbb Z}A\delta_i))^{\mathbb Z}=0$.  

It now follows from~\cite[Corollary~10.29]{Rotmanhom} that the edge morphisms give isomorphisms $H_n(\bigoplus_{i\in \mathbb Z}A\delta_i)_{\mathbb Z}\cong H_n(A\wr \mathbb Z)$ for $n\geq 2$, and from the construction of  the Lyndon--Hochschild--Serre spectral sequence this edge morphism is induced by the natural map.
\end{proof}

\begin{Thm}
Let $A$ be a finite abelian group and $\mathscr G$ the groupoid associated to a Skipper--Steinberg self-similar action of $A\wr \mathbb Z$ on $A$. Then $H_0(\mathscr G)\cong \mathbb Z/(|A|-1)\mathbb Z\cong H_1(\mathscr G)$ and $H_n(\mathscr G)=0$ for $n\geq 2$.
\end{Thm}
\begin{proof}
The computation of $H_0(\mathscr G)$ follows from Corollary~\ref{c:transitive.transfer}. 
We have that $H_1(A\wr \mathbb Z)\cong A\oplus \mathbb Z$ where the classes of $t=(0,1)$ and $(s\delta_0,0)$ map to $(0,1)$ and $(s,0)$ respectively. 
Now $(s\delta_0)|_d =(s\delta_0)|_0\in \bigoplus_{i \in \mathbb Z}A \delta_i$ for all $d\in A$, as was discussed in Example~\ref{ex:lamplighter}.  Thus $\sum_{d\in A} (s\delta_0)|_d=|A|\cdot (s\delta_0)|_0=0$.  On the other hand, $t|_d =(d\delta_0)t$, see Example~\ref{ex:lamplighter}. By Corollary~\ref{c:transitive.transfer}, the map $\id-H_1(\sigma_0)\circ \mathrm{tr}^{G}_{G_0}\colon  A\oplus \mathbb Z\to A\oplus \mathbb Z$ is given by $(d,0)\mapsto (d,0)$ and $(0,1)\mapsto (\sum_{d\in A} d, 1-|A|)$ and $H_1(\mathscr G) \cong (A\oplus \mathbb Z)/(A+\langle (\sum_{d\in A}d,1-|A|)\rangle)\cong \Z/(|A|-1)\mathbb Z$. 

Let $\sigma$ be the $1$-cocycle.  
We saw in Example~\ref{ex:lamplighter} that $\sigma$ restricts to a $1$-cocycle $\til \sigma \colon (\bigoplus_{i\in \mathbb Z}A\delta_i) \times A\to \bigoplus_{i\in \mathbb Z}A\delta_i$ and that $g|_d=g|_0$ for all $g\in \bigoplus_{i \in \mathbb Z}A \delta_i$ and $d\in A$. We obtain a diagram
\[\begin{tikzcd}[arrow style=math font, column sep = 2cm]
	{\bigoplus_{i \in \Z}A \delta_i} & {\bigoplus_{i \in \Z}A \delta_i} \\
	{A \wr \Z} & {A \wr \Z}
	\arrow["{\til \sigma_0 \circ \tr^H_{H_0}}", from=1-1, to=1-2]
	\arrow[from=1-1, to=2-1]
	\arrow[from=1-2, to=2-2]
	\arrow["{\sigma_0 \circ \tr^G_{G_0}}"', from=2-1, to=2-2]
\end{tikzcd}\]
which commutes up to isomorphism, with $H = \bigoplus_{i \in \Z}A \delta_i$. We show that $\Psi_n = H_n(\til \sigma_0)\circ \mathrm{tr}^{H}_{H_0} \colon H_n(\bigoplus_{i\in \mathbb Z}A\delta_i)\to H_n(\bigoplus_{i\in \mathbb Z}A\delta_i)$ is $0$ for all $n\geq 1$.
By Corollary~\ref{c:transitive.transfer}, $\Psi_n$ is induced by the chain map on $C_\bullet(\bigoplus_{i \in \mathbb Z} A \delta_i)$ given at $(f_1,\ldots, f_n) \in (\bigoplus_{i \in \mathbb Z} A \delta_i)^n$ by $(f_1,\ldots, f_n) \mapsto \sum_{d\in A}(f_1|_{f_2\cdots f_n(d)}  ,f_n|_d) = |A|(f_1|_0,\ldots, f_n|_0)$.   If $K$ is a finite group of exponent $r$, then $r\cdot H_n(K)=0$ for all $n\geq 1$, cf.~\cite[Corollary~9.95]{Rotmanhom}.   Therefore,  $|A|\cdot H_n(\bigoplus_{i=-m}^mA\delta_i)=0$, and hence $|A|\cdot H_n(\bigoplus_{i \in \mathbb Z} A\delta_i) =0$, as $H_n(\bigoplus_{i\in \mathbb Z}A\delta_i)=\varinjlim_{m\in \mathbb N} H_n(\bigoplus_{i=-m}^mA\delta_i)$. 
Since this chain map has image contained in $|A|\cdot C_n(\bigoplus_{i \in \mathbb Z}A \delta_i)$, we conclude that $\Psi_n=0$ for $n\geq 1$.

It now follows from Proposition~\ref{p:lamplighter.homology} that $\id -H_n(\sigma_0)\circ \mathrm{tr}^{G}_{G_0}=\id$ for all $n\geq 2$.  Note also that $\id-H_1(\sigma_0)\circ \mathrm{tr}^{G}_{G_0}\colon  A\oplus \mathbb Z\to A\oplus \mathbb Z$ computed in the first paragraph is injective.  We deduce from the long exact sequence of Corollary~\ref{c:transitive.transfer} that $H_n(\mathscr G)=0$ for $n\geq 2$.  
\end{proof}
The groupoid $\mathscr G$ is second countable, amenable, effective, minimal and Hausdorff with torsion-free isotropy. Since $H_n(\mathscr G)=0$ for $n\geq 2$ it satisfies the HK property by~\cite[Remark~3.5]{PY22}. This can also be deduced via Corollary~\ref{c:transitive.transfer.Ktheory} and the K-theory computation for $\cs(A\wr \mathbb Z)$ in~\cite{FPV17}.

\subsection{Solvable Baumslag--Solitar groups}\label{subs:BS}
Next we consider the example of solvable Baumslag--Solitar groups $BS(1,m)$ from Example~\ref{ex:BS.group}.

\begin{Thm}\label{t:BSgroup}
If $n\geq 2$ is relatively prime to $m$, then the homology of the groupoid $\mathscr G_{(m,n)}$ associated to the self-similar action of the  Baumslag--Solitar group $BS(1,m)$ from Example~\ref{ex:BS.group} is given by
\begin{enumerate}
  \item $H_0(\mathscr G_{(m,n)}) \cong \mathbb Z/(n-1)\mathbb Z$;
  \item $H_1(\mathscr G_{(m,n)})\cong \mathbb Z/(m-1)\mathbb Z\oplus \mathbb Z/(n-1)\mathbb Z$;
  \item $H_2(\mathscr G_{(m,n)})\cong \mathbb Z/(m-1)\mathbb Z$;
\end{enumerate}
and $H_q(\mathscr G_{(m,n)})=0$ for $q\geq 3$.
\end{Thm}
\begin{proof}
It is immediate from the presentation $BS(1,m) = \langle a,b\mid bab\inv =a^m\rangle$ that $\ab{BS(1,m)}\cong \mathbb Z/(m-1)\mathbb Z \oplus \mathbb Z$ where the class of $a$ maps to $(1,0)$ and the class of $b$ to $(0,1)$.  Lyndon's Identity Theorem implies that if $G$ is a torsion-free one-relator group, then its  presentation $2$-complex is a $K(G,1)$.  Hence,
$H_n(BS(1,m))=0$ for $n\geq 2$ since the defining relator doesn't belong to the commutator subgroup of the free group on $a,b$.  See~\cite[Chapter II.4, Example~3]{Browncohomology}.

The computation of $H_0(\mathscr G_{(m,n)})$ follows from Corollary~\ref{c:transitive.transfer}.  Let us write $\til g$ for the image of $g$ in $\ab {BS(1,m)}$.  Since $H_q(BS(1,m))=0$ for $q\geq 2$, we have $H_q(\mathscr G)=0$ for $q\geq 3$, and an exact sequence \[0\to H_2(\mathscr G)\to H_1(BS(1,m))\xrightarrow{\id -\Phi_1} H_1(BS(1,m))\to H_1(\mathscr G)\to 0\] from Corollary~\ref{c:transitive.transfer} where $\Phi_1(\til a) = \sum_{i=0}^{n-1}\til {a|_{\ov i}}$ and  $\Phi_1(\til b) = \sum_{i=0}^{n-1}\til {b|_{\ov i}}$. Then we see that $\Phi_1(\til a) = \til a$, and so $(\id -\Phi_1)(\til a)=0$.  To compute $\Phi_1(\til b)$, note that $\Phi_1(\til b) = \sum_{i=0}^{n-1}\lfloor mi/n\rfloor \til a+ n\til b$, and hence $(\id -\Phi_1)(\til b) =- \sum_{i=0}^{n-1}\lfloor mi/n\rfloor \til a - (n-1)\til b $.  But by definition $\sum_{i=0}^{n-1}mi = \sum_{i=0}^{n-1}(\lfloor mi/n\rfloor + b(\ov i))$.  Since $b$ permutes $\mathbb Z/n\mathbb Z$, we have that $\sum_{i=0}^{n-1}b(\ov i)=\sum_{i=0}^{n-1}i=\binom{n}{2}$.  Therefore,  \[\sum_{i=0}^{n-1}\lfloor mi/n\rfloor  = (m-1)\binom{n}{2}\equiv 0\bmod {(m-1)}.\] It follows that $(\id -\Phi_1)\til b = -(n-1)\til b$, and hence  $\coker (\id -\Phi_1)\cong   \mathbb Z/(m-1)\mathbb Z\oplus \mathbb Z/(n-1)\mathbb Z$.   Clearly, $\ker (\id -\Phi_1) = \langle \til a\rangle \cong \mathbb Z/(m-1)\mathbb Z$ since $(\id -\Phi_1)\til a =0$ and $(\id -\Phi_1)\til b=-(n-1)\til b$ has infinite order. This completes the proof.
\end{proof}

\begin{Rmk}
Using Theorem~\ref{t:BSgroup} and the spectral sequence of~\cite{PV18}, one can show that $K_1(\cs(\mathscr G_{(m,n)}))\cong \mathbb Z/(m-1)\mathbb Z\oplus \mathbb Z/(n-1)\mathbb Z$ and that there is an exact sequence $0\to \mathbb Z/(n-1)\mathbb Z\to K_0(\mathscr G_{(m,n)})\to \mathbb Z/(m-1)\mathbb Z\to 0$.  This can also be obtained using our methods and the well-known K-theory of $\cs(BS(1,m))$~\cite{PV18}.  It is not immediately clear for which $m,n$ this sequence splits, and thus when the HK property holds for this groupoid.    
\end{Rmk}

\subsection{Free abelian groups}\label{subs:freeabelian}
Next we consider self-similar actions of free abelian groups as per Example~\ref{Ex:FreeAb}, whose notation we retain.

We write $\Lambda^q(C)$ for the $q^{th}$-exterior power of a matrix $C$, where we take $\Lambda^q(C)=0$ if $q<0$.  The trick in the following lemma is inspired by~\cite{EHR11}. 

\begin{Lemma}\label{l:ext.power.trick}
Let $e_1,\ldots, e_n$ be the standard basis for $G=\mathbb Z^n$, let $d_1,\ldots, d_n\geq 1$ and put $d=d_1\cdots d_n$.  Let $H$ be the subgroup with basis $f_i=d_ie_i$   and let $\iota\colon H\to G$ be the inclusion.  Then, for $0\leq q\leq n$, 
\[T(e_{i_1}\wedge\cdots \wedge e_{i_q}) = d\left(\frac{1}{d_{i_1}}f_{i_1}\wedge\cdots \wedge \frac{1}{d_{i_q}}f_{i_q}\right)\]
is the unique homomorphism $T\colon \Lambda^q(G)\to \Lambda^q(H)$ such that $\Lambda^q(\iota)\circ T=d\cdot\id=T\circ \Lambda^q(\iota)$. 
\end{Lemma}
\begin{proof}
The map $\Lambda^q(\iota)$ is injective, and so invertible over $\mathbb Q$.  Thus over $\mathbb Q$ there is a unique such $T$, namely $d\Lambda^q(\iota)\inv$.  We check that $d\Lambda^q(\iota)\inv$ is defined over $\mathbb Z$.  Indeed,
$\Lambda^q(\iota)(f_{i_1}\wedge\cdots \wedge f_{i_q}) = d_{i_1}e_{i_1}\wedge\cdots\wedge d_{i_q}e_q = d_{i_1}\cdots d_{i_q}e_{i_1}\wedge\cdots 
\wedge e_{i_q}$.  
It follows that 
\begin{align*}
d\Lambda^q(\iota)\inv (e_{i_1}\wedge\cdots\wedge e_{i_q}) & = \frac{d}{d_{i_1}\cdots d_{i_q}}(f_{i_1}\wedge\cdots \wedge f_{i_q}) \\
& = d\left(\frac{1}{d_{i_1}}f_{i_1}\wedge\cdots \wedge \frac{1}{d_{i_q}}f_{i_q}\right)
\end{align*} which belongs to  $\Lambda^q(H)$.
\end{proof}

We can now compute the homology of the groupoid associated to a transitive self-similar group action of $\mathbb Z^n$.

\begin{Thm}\label{t:homology.self.similar.free.ab}
Let $G=\mathbb Z^n$ have a self-similar transitive action on a set $X$ of cardinality $d\geq 2$ and let $\mathscr G$ be the corresponding groupoid.  Fix $x\in X$,  and let $A$ the matrix of $\sigma_x\otimes 1_{\mathbb Q}$ with respect to some basis.  Then
\[H_q(\mathscr G) = \ker(\id -d\Lambda^{q-1}(A))\oplus \coker (\id -d\Lambda^q(A)).\]
In particular, $H_q(\mathscr G)=0$ if $q>n+1$.
\end{Thm}
\begin{proof}
Without loss of generality we may assume that  $e_1,\ldots, e_n$ is a basis for $\mathbb Z^n$ such that $f_1=d_1e_1,\ldots, f_n=d_ne_n$ is a basis for $G_x$  and $d=d_1\cdots d_n$.  Let $B\in M_n(\mathbb Z)$ be the matrix for the virtual endomorphism $\sigma_x$ with respect to these bases and $A$ the matrix for $\sigma_x\otimes 1_{\mathbb Q}$ with respect to the basis $e_1,\ldots, e_n$.

Recall~\cite[Chapter~V, Section~6]{Browncohomology} that, for an abelian group $H$, there is a  homomorphism $\psi\colon \Lambda^*(H)\to H_*(H)$, natural in $H$, induced by the identification $H\to H_1(H)$, where $H_*(H)$ is an anti-commutative ring via the Pontryagin product. 
Moreover, $\psi$ is an isomorphism whenever $H$ is torsion-free.  In particular, $H_q(G)\cong \mathbb Z^{\binom{n}{q}}$, for all $q\geq 0$.  

We show that under the identification $\Lambda^q(G)\cong H_q(G)$, the map $H_q(\sigma_x)\circ \mathrm{tr}^G_{G_x}$ is given by $d\Lambda^q(A)$.   
Let $\iota\colon G_x\to G$ be the inclusion.  Then $H_q(\iota)\circ\mathrm{tr}^G_{G_x} = d\cdot \id$ by~\cite[Chapter III, Proposition~9.5]{Browncohomology}.  
Under the natural identification $H_q(G)\cong\Lambda^q(G)$ and $H_q(G_x)\cong\Lambda^q(G_x)$, we have that $H_q(\iota)$ corresponds to $\Lambda^q(\iota)$, and so $\mathrm{tr}^G_{G_x}(e_{i_1} \wedge \cdots \wedge e_{i_q}) = d(\frac{1}{d_{i_1}}f_{i_1}\wedge \cdots \wedge \frac{1}{d_{i_q}}f_{i_q})$ by Lemma~\ref{l:ext.power.trick}.   From the naturality of $\psi$,  and the fact that $Ae_i = \frac{1}{d_i}\sigma_x(f_i)$, we see that $H_q(\sigma_x)\circ \mathrm{tr}^G_{G_x}\colon H_q(G)\to H_q(G)$ is given by $d\Lambda^q(A)$ for $q\geq 0$ under our identifications.  The result follows from  Corollary~\ref{c:transitive.transfer}.
\end{proof}

Recall from Example~\ref{Ex:FreeAb} that if the self-similar action is self-replicating, then $A\inv$ is an integer dilation matrix with image $G_x$, where $x\in X$, and so $d=[\mathbb Z^n:G_x]=|\det(A\inv)|$.  It follows that $d\Lambda^q(A)\cdot \Lambda^q(A\inv) = d\cdot\id$.
The following is essentially~\cite[Proposition~4.6]{EHR11} (where we note they work with the transpose of $A\inv$), but we give an easier proof.

\begin{Prop}\label{p:self.replicating.contracting}
Suppose that $A$ has spectral radius less than $1$ and $\sigma_x$ is surjective.  
\begin{enumerate}
    \item $1 -d\Lambda^0(A) = 1-d<0$. 
    \item $\det(\id -d\Lambda^q(A))\neq 0$ for $1\leq q<n$.
    \item $1 -d\Lambda^n(A)=\begin{cases}0, & \text{if}\ \det A>0\\ 2, & \text{if}\ \det A< 0.\end{cases}$
\end{enumerate}
\end{Prop}
\begin{proof}
The first item is trivial.  For the second item, note that $d=|\det A\inv|$, and so if $\lambda_1,\ldots, \lambda_n$ are the complex eigenvalues of $A$ with multiplicity, then $|\lambda_1\cdots\lambda_n|=1/d$.  Now the eigenvalues of $\Lambda^q(A)$ are well known to be all products $\lambda_{i_1}\dots \lambda_{i_q}$ with $i_1<\cdots<i_q$.  If $1\leq q<n$, then we cannot have  $\lambda_{i_1}\dots \lambda_{i_q}=1/d$ as the spectral radius of $A$ is less than $1$, and so we would obtain the contradiction  $|\lambda_1\cdots\lambda_n|<1/d$.  Thus (2) holds.  Since $d=|\det A\inv|$, we have $1-d\Lambda^n(A) = 1- \det A/|\det A|$, and (3) follows.
\end{proof}
\begin{Cor}\label{c:self.rep.ab}
Let $G=\mathbb Z^n$ have a self-similar contracting and self-rep\-li\-ca\-ting action on a set $X$ of cardinality $d\geq 2$ and let $\mathscr G$ be the corresponding groupoid.  Fix $x\in X$,  and let $A$ the matrix of $\sigma_x\otimes 1_{\mathbb Q}$ with respect to some basis.  Then
\[H_q(\mathscr G) = \begin{cases} \mathbb Z/(d-1)\mathbb Z, & \text{if}\ q=0,\\
\coker (\id -d\Lambda^q(A)), & \text{if}\ 1\leq q<n,\\
\mathbb Z, & \text{if}\ n \leq q \leq n+1, \det A>0,\\
\mathbb Z/2\mathbb Z & \text{if}\ q=n, \det A<0,\\
0, & \text{else.}
\end{cases}\]   
\end{Cor}
\begin{proof}
This is immediate from Theorem~\ref{t:homology.self.similar.free.ab} and Proposition~\ref{p:self.replicating.contracting}.
\end{proof}

We now turn to the K-theory of the $\cs$-algebras of self-similar actions of free abelian groups.  These results generalize those of~\cite{EHR11}, which in light of~\cite{LRRW14} correspond to the case of self-replicating contracting free abelian groups.

The K-theory $K_*(B) = K_0(B) \oplus K_1(B)$ of a commutative $\cs$-algebra $B$ has the structure of a $\Z /2 \Z$-graded ring.
It is a well-known result, cf.~\cite{Elliot84,Ji86}, that  $K_*(\cs(\mathbb Z^n))$ is graded isomorphic to the exterior algebra $ \Lambda^*(\mathbb Z^n)$, with the grading into even degree and odd degree wedge products. If $e_1,\ldots, e_n$ is the standard basis for $\mathbb Z^n$, then $[u_{e_i}]_1 \mapsto e_i\in \Lambda^1(\mathbb Z^n)$ under the isomorphism (and $[1]_0$ maps to the empty wedge product).
It follows easily from this that if $A\in M_n(\mathbb Z)$ is a matrix, then the map on K-theory induced by the endomorphism $A$ of $\mathbb Z^n$, which is a ring homomorphism, is conjugate via the above isomorphism to $\Lambda^*(A)$.    

\begin{Thm}\label{t:FreeAbK}
Let $G=\mathbb Z^n$ have a self-similar transitive action on a set $X$ of cardinality $d\geq 2$ and let $\mathscr G$ be the corresponding groupoid.  Fix $x\in X$,  and let $A$ the matrix of $\sigma_x\otimes 1_{\mathbb Q}$ with respect to some basis of $G$.  Then:
\begin{enumerate}
\item $K_0(\cs(\mathscr G)) =  \bigoplus_{q\geq 0}\ker(\id -d\Lambda^{2q-1}(A))\oplus \coker(\id -d\Lambda^{2q}(A))$.
\item $K_1(\cs(\mathscr G))= \bigoplus_{q\geq 0}\ker(\id -d\Lambda^{2q}(A))\oplus \coker(\id -d\Lambda^{2q+1}(A))$.
\end{enumerate}
The class $[1]_0\in K_0(\cs(\mathscr G))$ of the unit is the generator of the summand $\coker(\id-d\Lambda^0(A))\cong \mathbb Z/(d-1)\mathbb Z$.
\end{Thm}
\begin{proof}
%
Without loss of generality, we may assume that we have chosen our basis $e_1,\ldots, e_n$ of $\mathbb Z$ so that there are positive integers $d_1,\ldots,d_n$ such that $f_1=d_1e_1,\ldots, f_n=d_ne_n$ is a basis for $G_x$. Note that $d=d_1\cdots d_n$. Let $\iota\colon G_x\to G$ be the inclusion.   By Corollary~\ref{c:inc.transfer}, we have that $\mathrm{tr}^G_{G_x}\circ K_*(\iota)=d\cdot \id$.  In particular, $K_*(\iota)$ is invertible over $\mathbb Q$ and $\mathrm{tr}^G_{G_x}$ must be $dK_*(\iota)\inv$.  
Under the identification of the $K$-theory of $\cs(G)$ and $\cs(G_x)$ with $\Lambda^*(G)$ and $\Lambda^*(G_x)$, respectively, we see from Lemma~\ref{l:ext.power.trick} that $\mathrm{tr}^G_{G_x}(e_{i_1}\wedge\cdots\wedge e_{i_q}) = d(\frac{1}{d_{i_1}}f_{i_1}\wedge\cdots\wedge \frac{1}{d_{i_q}}f_{i_q})$.  As $Ae_i = \frac{1}{d_i}\sigma_x(f_i)$, the result now follows from Corollary~\ref{c:transitive.transfer.Ktheory} and the observation that $\ker(\id -d\Lambda^q(A))$ is free abelian. 
\end{proof}

The groupoid $\mathscr G$ is Hausdorff and amenable, and thus satisfies the HK property by comparing Theorem~\ref{t:homology.self.similar.free.ab} and Theorem~\ref{t:FreeAbK}.

In the case that the action is contracting and self-replicating, we obtain the follow simplification in light of Proposition~\ref{p:self.replicating.contracting}, recovering the K-theoretic computation in~\cite{EHR11} by~\cite[Corollary~3.10]{LRRW14}.

\begin{Cor}\label{c:FreeAbK}
Let $G=\mathbb Z^n$ have a self-similar contracting and self-replicating action on a set $X$ of cardinality $d\geq 2$ and let $\mathscr G$ be the corresponding groupoid.  Fix $x\in X$,  and let $A$ the matrix of $\sigma_x\otimes 1_{\mathbb Q}$ with respect to some basis.  Then:
\begin{enumerate}
    \item If $\det A>0$ and $n$ is odd, then 
    \begin{align*}
        K_0(\cs(\mathscr G)) & = \mathbb Z\oplus \left(\bigoplus_{0\leq q\leq \frac{n-1}{2}} \coker(\id -d\Lambda^{2q}(A))\right)\\
        K_1(\cs(\mathscr G)) & =  \bigoplus_{0\leq q\leq \frac{n-1}{2}} \coker(\id -d\Lambda^{2q+1}(A)).
    \end{align*}
    \item If $\det A>0$ and $n$ is even, then
    \begin{align*}
        K_0(\cs(\mathscr G)) & = \bigoplus_{0\leq q\leq \frac{n}{2}} \coker(\id -d\Lambda^{2q}(A))\\
        K_1(\cs(\mathscr G)) & = \mathbb Z\oplus \left(\bigoplus_{0\leq q< \frac{n}{2}} \coker(\id -d\Lambda^{2q+1}(A))\right).
    \end{align*}
    \item If $\det A<0$, then
    \begin{align*}
        K_0(\cs(\mathscr G)) & = \bigoplus_{0\leq q\leq \lfloor \frac{n}{2}\rfloor} \coker(\id -d\Lambda^{2q}(A))\\
        K_1(\cs(\mathscr G)) & = \bigoplus_{0\leq q\leq  \lfloor\frac{n-1}{2}\rfloor} \coker(\id -d\Lambda^{2q+1}(A)).
    \end{align*}
\end{enumerate}     
\end{Cor}

As an example, we compute the homology and K-theory of the groupoid and $\cs$-algebra associated to the sausage automaton from Example~\ref{ex:sausage} for free abelian groups of prime rank.  This the action is self-replicating and contracting, so we can apply the previous corollary.

\begin{Thm}
Let $\mathscr G$ be the groupoid associated to the self-similar action of $\mathbb Z^n$ on $\{0,1\}$ given via the sausage automaton with $n$ prime.  Then
\[H_q(\mathscr G) = \begin{cases}\left(\mathbb Z/(2^{n-q}-1)\mathbb Z\right)^{\frac{1}{n}\binom{n}{q}}, & \text{if}\ 1\leq q\leq n-1\\
 \mathbb Z/(1+(-1)^{n})\mathbb Z, & \text{if}\ q=n\\   (1+(-1)^{n-1})\mathbb Z, & \text{if}\ q=n+1\\ 0, & \text{else.}\end{cases}\]
Moreover $K_0(\cs(\mathscr G))=\bigoplus_{q\geq 0} H_{2q}(\mathscr G)$ and $K_1(\cs(\mathscr G))\cong \bigoplus_{q\geq 0}H_{2q+1}(\mathscr G)$.
\end{Thm}
\begin{proof}
The K-theory statement follows from Corollary~\ref{c:FreeAbK}.
The result for $q=0$ and $q\geq n$ is clear from Corollary~\ref{c:self.rep.ab} and the observation that $\det (A) = (-1)^{n-1}/2$. 

It remains to compute $\coker(\id -2\Lambda^q(A))$ for $1\leq q\leq n-1$.
If $I=\{i_1,\ldots, i_q\}\subseteq [n]$ with $i_1<\cdots<i_q$, put $e_I = e_{i_1}\wedge \cdots\wedge e_{i_q}$.   Identifying $[n]$ with $\mathbb Z/n\mathbb Z$, we can define $I -1 = \{i-1\mid i\in I\}$ (taken modulo $n$).  Since $1\leq q<n$ and $n$ is prime, it follows that $I, I-1,\ldots, I-(n-1)$ are distinct.  There are then $\frac{1}{n}\binom{n}{q}$ orbits like this.   Notice that \[2\Lambda^q(A) e_I = \begin{cases}2e_{I-1},  & \text{if}\ 0\notin I\\ (-1)^{q-1}e_{I-1}, & \text{if}\ 0\in I. \end{cases}\]
It follows that $\coker(\id -2\Lambda^q(A))$ has one cyclic summand per orbit.  Since there are $q$ terms in the orbit $I,\ldots, I -(n-1)$ which contain $0$, in the cokernel the summand corresponding to the orbit of  $e_I$ satisfies the relation $e_I = 2^{n-q}((-1)^{q-1})^qe_I = 2^{n-q}(-1)^{q(q-1)}e_I = 2^{n-q}e_I$.  Therefore, this summand is isomorphic to $\mathbb Z/(2^{n-q}-1)\mathbb Z$, and so $\coker(\id -2\Lambda^q(A))\cong (\mathbb Z/(2^{n-q}-1)\mathbb Z)^{\frac{1}{n}\binom{n}{q}}$, as required. 
\end{proof}

\appendix\section{Wildon's lemma}

Since Wildon's lemma~\cite{Wildon} is not formally  published, we include a proof. 

\begin{Lemma}[Wildon]\label{l:wildon}
Let $A$ be the $(n+1)\times (n+1)$-matrix, indexed by $0,\ldots, n$, with $A_{ij}= (-1)^i\binom{n-i}{j}$.  Then $A^3= (-1)^n\id$.
\end{Lemma}
\begin{proof}
We compute 
\begin{align*}
A^3_{ij} &= \sum_{k,\ell=0}^n(-1)^{i+k+\ell} \binom{n-i}{k}\binom{n-k}{\ell}\binom{n-\ell}{j}\\ & =(-1)^{i+n}\sum_{k,\ell=0}^n(-1)^{k+\ell-n}\binom{n-i}{k}\binom{n-k}{n-k-\ell}\binom{n-\ell}{j}\\ & = (-1)^{i+n}\sum_{k,\ell=0}^n\binom{n-i}{k}\binom{-\ell -1}{n-k-\ell}\binom{n-\ell}{j}\\ 
& =(-1)^{i}\sum_{\ell=0}^n(-1)^{\ell}(-1)^{n-\ell} \binom{n-i-\ell -1}{n-\ell}\binom{n-\ell}{j} \\ & =
(-1)^{i}\sum_{\ell=0}^n(-1)^{\ell}\binom{i}{n-\ell}\binom{n-\ell}{j}\\ & = (-1)^{i+n}\sum_{r=0}^n (-1)^{r}\binom{i}{r}\binom{r}{j} \\ & = (-1)^n\delta_{ij}
\end{align*}
where the third and fifth equalities use the identity $(-1)^r\binom{s}{r} = \binom{r-s-1}{r}$, the fourth equality uses Vandermonde's identity and the last equality uses that
\[\sum_{r=0}^n(-1)^r\binom{i}{r}\binom{r}{j} = (-1)^i\delta_{ij};\] see, for instance,~\cite{binomialcoeffs}.
\end{proof}

\newcommand{\etalchar}[1]{$^{#1}$}
\def\malce{\mathbin{\hbox{$\bigcirc$\rlap{\kern-7.75pt\raise0,50pt\hbox{${\tt m}$}}}}}\def\cprime{$'$} \def\cprime{$'$} \def\cprime{$'$} \def\cprime{$'$} \def\cprime{$'$} \def\cprime{$'$} \def\cprime{$'$} \def\cprime{$'$} \def\cprime{$'$} \def\cprime{$'$}

\end{document}